\documentclass
{article}
\usepackage[T2A]{fontenc}
\usepackage[cp1251]{inputenc}
\usepackage[english,russian]{babel}
\usepackage[tbtags]{amsmath}
\usepackage{amsfonts,amssymb,mathrsfs,amscd}
\usepackage{verbatim}
\usepackage{eucal}
\usepackage{graphicx}
\usepackage{enumerate}
\usepackage[pdftex,unicode,colorlinks,linkcolor=blue,
citecolor=red,bookmarksopen,pdfhighlight=/N]{hyperref}
\usepackage[hyper]{msb-a}

\JournalName{}

\numberwithin{equation}{section}
\theoremstyle{plain}
\newtheorem{theorem}{Теорема}
\newtheorem{lemma}{Лемма}[section]
\newtheorem{propos}{Предложение}[section]
\newtheorem{corollary}{Следствие}[section]
\newtheorem{addition}{Дополнение}[section]
\newtheorem{Conj}{Гипотеза}
\theoremstyle{definition}
\newtheorem{definition}{Определение}
\newtheorem{proof}{Доказательство}
\newtheorem{remark}{Замечание}[section]
\newtheorem{example}{Пример}[section]
\newcommand{\const}{{\rm const}}

\renewcommand{\leq}{\leqslant}
\renewcommand{\geq}{\geqslant}
\newcommand{\e}{\varepsilon}
\newcommand{\leftd}{\text{\tiny\rm left}}
\newcommand{\rightd}{\text{\tiny\rm right}}

\def\const{\mathrm{const}}
\def\RR{\mathbb R}
\def\CC{\mathbb C}
\def\NN{\mathbb N}
\def\DD{\mathbb D}

\DeclareMathOperator{\dist}{dist}
\DeclareMathOperator{\scl}{scl}
\DeclareMathOperator{\clos}{clos}

\DeclareMathOperator{\Har}{har}
\DeclareMathOperator{\Hol}{Hol}

\DeclareMathOperator{\Zero}{Zero}
\DeclareMathOperator{\sbh}{sbh}
\DeclareMathOperator{\dom}{dom}
\DeclareMathOperator{\dsbh}{\text{$\delta${\rm -sbh}}}
\DeclareMathOperator{\supp}{supp}
\DeclareMathOperator{\comp}{c}
\DeclareMathOperator{\decr}{decr}

\DeclareMathOperator{\monot}{monot}

\DeclareMathOperator{\loc}{loc}

\begin{document}

\title{Распределение нулей и масс\\голоморфных и субгармонических функций: 
условия типа Адамара и  Бляшке}
\author[B.\,N.~Khabibullin]{Б.\,Н.~Хабибуллин}
\address{Башкирский государственный университет}
\email{khabib-bulat@mail.ru}
\author[N.\,R.~Tamindarova]{Н.\,Р.~Таминдарова}
\address{Башкирский государственный университет}
\email{nargiza89@gmail.com}

\date{02.06.2018}
\udk{517.53 : 517.574}

\maketitle

\begin{fulltext}

\begin{abstract}
Пусть $M$ --- субгармоническая функция-мажоранта в области $D$ комплексной плоскости $\mathbb C$ с мерой Рисса $\nu_M$, $f$ --- ненулевая голоморфная в $D$ функция,  $\log |f|\leq M$ на $D$ и функция $f$ обращается в нуль на последовательности точек ${\tt Z}=\{{\tt z}_k\}_{k=1,2, \dots}\subset D$ {\large(}соответственно $u\not\equiv -\infty$ --- субгармоническая функция с мерой Рисса, или распределением масс, $\nu_u$, $u\leq M$ на $D${\large)}.  Тогда   ограничения на рост меры Рисса $\nu_M$ функции $M$ вблизи границы области $D$ влекут за собой определённые  ограничения на  распределение точек последовательности $\tt Z$ (соответственно на распределение масс $\nu_u$).  Количественная форма исследования этого явления даётся сразу в субгармоническом обрамлении.  Устанавливаются также и результаты в обратном направлении.  Детально разобраны случаи, когда   $D$ --- это $\CC$,  единичный круг, внешность этого круга,  концентрическое кольцо, а  $M$ --- радиальная функция; $D$ ---  регулярная область, а мажоранта  $M$ постоянна на линиях уровня функции Грина области $D$;  $D$ --- область гиперболического типа, а $M$ --- суперпозиция выпуклой функции с функциями, зависящими от гиперболического радиуса; $D$ --- регулярная область, а $M$  --- суперпозиция  выпуклой  функции с функцией, зависящей от расстояния до границы $\partial D$ области $D$ или до некоторого подмножества границы $\partial D$.  Все основные результаты и их реализации в более или менее конкретных ситуациях  новые не только для субгармонических функций, но и для голоморфных, уже в случае, когда $D$ ---  это $\CC$, единичный круг,  кольцо и т.\,д. Таким образом, эта работа не обзор. 

Библиография:  81 название

Distribution of zeros (masses) for holomorphic (subharmonic) functions. I. Hadamard-  and Blaschke-type conditions

Let $M$ be a subharmonic function on a domain  $D$ in the complex plane $\mathbb C$ with the Riesz measure $\nu_M$.
Let  $f$ be a non-zero holomorphic  function on  $D$ such that   $\log |f|\leq M$ on $D$ and the function $f$ vanish on a sequence  ${\tt Z}=\{{\tt z}_k\}_{k=1,2, \dots}\subset D$ {\large(}$u\not\equiv -\infty$ be a subharmonic function  on $D$ with the Riesz measure or the mass distribution $\nu_u$, and $u\leq M$ on $D$ resp.{\large)}.   
Then restrictions on the growth of the Riesz measure $\nu_M$ of the function $M$ near the boundary of the domain $D$ entail certain restrictions on the distribution of points of the sequence $\tt Z$ (to the mass distribution $\nu_u$ resp.).
 A quantitative form of research of this phenomenon is given immediately in the subharmonic framework. We also establish results in the inverse direction. We investigated in detail the cases when $D$ is $\mathbb C$, the unit disk, exterior of the unit disk, a concentric annulus, and $M$ is a radial function; $D$ is a regular domain and $M$ are constant on the level lines of Green's function of this  domain $D$; $D$ is a domain  of hyperbolic type, and $M$ are  the superpositions of convex functions with functions that depend on the hyperbolic radius; $D$ is a regular domain and $M$ is the superposition of convex functions with a  function dependent on the distance to some subset of the boundary of the domain $D$. All our main results and their implementation in more or less concrete situations are new not only for subharmonic functions $u$, and also for holomorphic functions $f$ even in the case when $D$ is $\mathbb C$, the unit disk, an annulus etc. Thus this is not a review.
\end{abstract}


\markright{Распределение нулей  голоморфных функций \dots }

\footnotetext[0]{Исследование выполнено за счёт гранта Российского
	научного фонда (проект №~18-11-00002), при поддержке РФФИ 
(проект №\,16-01-00024).}

\tableofcontents

\section{Введение}\label{s1}

\subsection{Истоки}\label{111}

Исторически первой отправной точкой нашего исследования можно, по-видимому, считать 
 элементарное  следствие Теоремы Э.~Безу (XVIII в.), которое переформулируем в подходящей форме:   {\it многочлен  от комплексной переменной,   растущий не быстрее $n$-ой степени модуля этой переменной, не может иметь более чем $n$ корней  (нулей) даже с учётом кратности на комплексной плоскости $\CC$}. 
Последующие ключевые этапы ---    классические результаты конца XIX--начала XX вв. о распределении нулей целых функций на  $\CC$ и  голоморфных функций  в единичном круге $\DD:=\bigl\{ z\in \CC \colon |z|<1\bigr\}$  при ограничении на рост их модуля  соответственно при стремлении $z\in \CC$ к $\infty$ и при приближении $z\in \DD$ к  единичной окружности  $\partial \DD:=\{z\in\CC \colon |z|=1\}$ --- границе круга $\DD$. Для $0<r\leq +\infty$ положим  $D(r):=r\DD:=\{ rz\colon z\in \DD\}=\{ z\in \CC \colon |z|<r\}$
 --- открытый круг радиуса $r$ с центром в нуле.  Используем понятия из   \cite{Rans}--\cite{St}.

\vskip 1mm
	{\sc Теорема A} {(с условием Ж.~Адамара).} {\it  Пусть  $0<R_0<+\infty$, $M$ --- субгармоническая  в\/ $\CC$ функция с мерой Рисса $\nu_M$, а  
	$f$ --- ненулевая целая функция, обращающаяся в нуль на последовательности\/ ${\tt Z}=\{{\tt z}_k\}_{k=1,2,\dots}\subset \CC$ 
	и $\log |f|\leq M$ на\/ $\CC$ (поточечно). Тогда из условия Адамара для мер  с показателем   $q> 0$ 
		\begin{equation}\label{BcnuWA}
		\int\limits_{\CC\setminus D(R_0)} \frac1{|z|^q}\, d \nu_M(z)<+\infty
	\end{equation}
	следует условие Адамара для последовательностей  с тем же показателем  $q$:
		\begin{equation}\label{BcZWA}
		\sum_{{\tt z}_k\in \CC\setminus D(R_0)} \frac1{|{\tt z}_k|^q}<+\infty	.	
	\end{equation}}
 \vskip 1mm

\noindent
Здесь  {\it тестовая\/} подынтегральная функция в \eqref{BcnuWA}
\begin{equation*}
	z\longmapsto \frac{1}{|z|^q}, \quad z\in \CC\setminus D(R_0),
\end{equation*}
дающая  и слагаемые в  \eqref{BcZWA},   {\it положительна\/} в $\CC\setminus D(R_0)$, {\it стремится к нулю  при\/}  $z\to \infty$ и {\it субгармоническая\/} в  <<проколотой>> плоскости  $\CC \setminus \{0\} \supset \CC\setminus D(R_0)$. 

	\vskip 1mm
	{\sc Теорема B} (с условием В.~Бляшке). {\it  Пусть  $0<r_0<1$, $M$ --- субгармоническая  в\/ $\DD$ функция с мерой Рисса $\nu_M$, 
а $f$ --- ненулевая голоморфная функция в круге\/ $\DD$, обращающаяся в нуль на последовательности\/ ${\tt Z}=\{{\tt z}_k\}_{k=1,2,\dots}\subset \DD$ и $\log |f|\leq M$ на\/ $\DD$ (поточечно). Тогда  из  условия Бляшке для мер в одной из эквивалентных друг другу форм: 
		\begin{equation}\label{Bcnu}
		\left(\,\int\limits_{\DD\setminus D(r_0)} \log \frac1{|z|}\, d \nu_M(z)<+\infty\right)  \; \Longleftrightarrow \;
	\left(\,\int\limits_{\DD\setminus D(r_0)} \bigl(1-| z|\bigr)\, d \nu_M(z)<+\infty\right).	
	\end{equation}
следует условие Бляшке для последовательностей в любой из эквивалентных друг другу форме}
		\begin{equation}\label{BcZ}
		\left(\sum_{{\tt z}_k\in \DD\setminus D(r_0)} \log \frac1{|{\tt z}_k|}<+\infty	\right) \; \Longleftrightarrow \;
	\left(\sum_{{\tt z}_k\in \DD\setminus D(r_0)} \bigl(1-|{\tt z}_k|\bigr)<+\infty\right).	
	\end{equation}
	 \vskip 1mm

\noindent
Здесь {\it тестовые\/} подынтегральные функции в \eqref{Bcnu}
\begin{equation}\label{testDl}
	z\mapsto \log\frac{1}{|z|}\, , \quad z\mapsto 1-|z|, \quad z\in \DD\setminus D(r_0),
\end{equation}
дающие и  слагаемые в \eqref{BcZ}, также {\it положительны\/} и {\it ограничены\/} в $\DD \setminus D(r_0)$,   {\it стремятся к нуля  при приближении  изнутри\/ $\DD$ к\/} $\partial \DD$, а первая из них\footnote{Замена второй супергармонической функции в \eqref{testDl} на эквивалентную ей в данных вопросах функцию $z\mapsto \frac{1}{|z|}-1$ также преобразует её в субгармоническую с сохранением остальных свойств: положительность и стремление к нулю при приближении изнутри к $\partial \DD$.} --- функция Грина для круга $\DD$ с полюсом в нуле ---  гармоническая, а стало быть и {\it субгармоническая\/} в <<проколотом>> круге  $\DD\setminus \{0\} \supset \DD\setminus D(r_0)$. 

\vskip 1mm
{\sc Основной Вопрос.} {\it Пусть $D\subset \CC_{\infty}:=\CC\cup\{\infty\}$ --- область, $D\neq \CC_{\infty}$, $D_0$ --- её  относительно компактная подобласть, $M$ --- субгармоническая  функция в $D$ с мерой Рисса $\nu_M$, $f$ ---  голоморфная функция в $D$, обращающаяся в нуль на  последовательности ${\tt Z}=\{{\tt z}_k\}_{k=1,2,\dots}\subset D$ и $\log |f|\leq M$ на $D\setminus D_0$ (поточечно). Для каких  тестовых функций $v\geq 0$ на $D\setminus D_0$ при ненулевой функции $f$ (пишем  $f\neq 0$) справедлива импликация}
\begin{equation}\label{imp} 
	\left( \int_{D\setminus D_0}v\, d \nu_M <+\infty\right) \; \Longrightarrow \; \left( \sum_{{\tt z}_k \in D\setminus D_0}v ({\tt z}_k) <+\infty\right) \text{\;\large ?}
\end{equation}
\vskip 1mm
  Отрицание подобной импликации позволяет получать теоремы единственности, а именно: если ряд  в \eqref{imp} расходится, а интеграл в \eqref{imp} конечен, то из $\log |f|\leq M$ на $D\setminus D_0$ следует, что $f$ --- нулевая функция (пишем $f=0$). 
			
Наше Следствие \ref{cor:1}, сформулированное в разделе \ref{CorMT},  даёт  следующий
 
\vskip 1mm
{\sc Ответ.} {\it Достаточно, чтобы функция $v$ была положительной    субгармонической в  открытой окрестности множества 
$D\setminus D_0$ и  стремилась к нулю  при приближении изнутри области  $D$ к её границе\/ $\partial D$  в\/ $\CC_{\infty}$. 

\vskip 1mm}
Более общее утверждение --- Теорема \ref{th:1} из подраздела \ref{CorMT}, в которой рассматриваются равномерно ограниченные сверху в $D\setminus D_0$ классы {\it тестовых\/} субгармонических функций $v\geq 0$ и получены равномерные относительно такого класса оценки сумм из  \eqref{imp} через интегралы из \eqref{imp}. Основной Теоремой из раздела \ref{ss:23} охвачен и случай произвольной $\delta$-субгармонической функции-мажоранты $M$\/
\cite{Ar_d}--\cite{Gr}, т.\,е. разности субгармонических функций.  

Основные результаты нашей работы частично были анонсированы  в \cite{KhT15}. Библиографические ссылки ниже
не претендуют  на сколь-ни\-будь близкое  к полноте освещение истории тематики, связанной с содержанием нашей работы.    
Так, ссылка на источник означает, что в нём содержатся и дальнейшие дополнительные сведения. 
Некоторые ранее полученные результаты, тесно связанные с нашими  Следствием \ref{cor:1}  и Теоремой \ref{th:1} изложены в 
\cite{Colwell}--\cite{Kh10}.  
В \cite{Colwell} --- история условия Бляшке; в 
\cite{Dj73}--\cite{Kh06} --- круг $\DD$ и радиальная мажоранта $M$; в \cite{GLO}--\cite{KV}   ---  функции на $\CC$; в 
\cite{BGK09}--\cite{FRdisc13} --- круг  и нерадиальная  мажоранта, равная суперпозиции некоторой  положительной убывающей неограниченной функции на $(0,+\infty)$ с функцией расстояния до некоторого подмножества   единичной окружности $\partial \DD$ --- версия для $\RR^n$ см. в \cite{FR13_D};
в \cite{GK12} --- область-дополнение до $\CC_{\infty}$ конечного числа отрезков на $\RR$ и мажоранта, задаваемая через отрицательные степени расстояний до этих отрезков и их концов;  \cite{FG12_unb} ---   область $D$ в $\CC_{\infty}$, содержащая точку $\infty$ и совпадающая с объединением всех  открытых кругов некоторого фиксированного радиуса, в ней содержащихся, с мажорантой $M$ на $D$, задаваемой как суперпозиция положительной убывающей неограниченной функции на $(0,+\infty)$ с функцией расстояния до границы $\partial D$ --- версии для $\RR^n$ см. \cite{FR13}\label{foot}; 
 в \cite{Khab01}--\cite{Kh10} --- различные виды областей  $D\subset \CC$, вплоть до произвольных, с достаточно общего вида мажорантой $M$. 
В подавляющем большинстве известных результатов или их наглядных иллюстрациях мажоранта $M$ задаётся в той или иной степени явно, так что оценка поведения её меры Рисса $\nu_M$, участвующей в посылке импликации  \eqref{imp}, или даже явное её вычисление как $\nu_M:=\frac{1}{2\pi} \Delta M$  в смысле теории распределений Л.~Шварца, или теории обобщенных функций, через действие оператора Лапласа $\Delta$ бывают возможны {\large(}
см. подраздел \ref{OLsf}{\large)}.   
Мы глубоко признательны \fbox{А.\,Ф.~Гришину}, С.\,Ю.~Фаворову,  Ю.~Риихентаусу  за полезные обсуждения и/или информацию и в особенности  Ф.\,Г.~Авхадиеву за  ценные консультации по геометрической теории функций комплексного переменного (см. пункты \ref{excr}, \ref{distt}, \ref{MgcR}, \ref{MdistED} ниже). 

\subsection{Cтруктура}\label{strs} Основная Теорема нашей работы  из \S~\ref{S2}, подраздел \ref{ss:23}, проиллюстрирована её частными случаями в подразделе \ref{CorMT} Введения --- Теорема \ref{th:1} и её Следствие \ref{cor:1}.  Их точность, как и  Основной Теоремы  на уровне примеров не обсуждается, поскольку
установлено гораздо б\'ольшее:  Теорема  \ref{th:1} и  Основная Теорема для любой области $D$ с неполярной границей допускают  обращение  для непрерывной $\delta$-субгармонической мажоранты  $M$ как в субгармоническом, так и в голоморфном обрамлении --- Теорема \ref{th:inver} и Следствие  \ref{holD}    из \S~\ref{ConverRes}. В подразделе \ref{an_cklf}  рассмотрены обращения для произвольных собственных областей $D\subset \CC_{\infty}$, основанные только на тестовых функциях Грина (Теорема \ref{th3}
и Следствие \ref{holDg}), а также только на тестовых аналитических или полиномиальных дисках (Теорема \ref{th:4} и Следствие \ref{holDgad}).
В \S~\ref{S3} детально разобран случай радиальной  субгармонической  --- см. подраздел \ref{ssCr} --- мажоранты $M$  на плоскости  --- \ref{ssCrC}, в круге и его внешности --- \ref{caseDD}, а также в кольце  --- \ref{Sann}. В \S~\ref{S6tf}
разработана техника построения разнообразных классов тестовых функций с использованием  функций Грина, гиперболического и конформного радиуса области $D$, проработаны их взаимосвязи с функциями расстояний до подмножеств границы $\partial D$. Эти построенные классы тестовых функций позволяют в \S~\ref{Snonrad}  дать серию различных версий Теоремы \ref{th:1} с более или менее явными тестовыми функциями и мерами Рисса мажорант-суперпозиций $M$,  вычисленными в подразделе  \ref{OLsf}. 
О  перспективах развития данного исследования несколько слов сказано в заключительном п.~\ref{persp}  статьи.

\subsection{Основные обозначения, определения и соглашения}\label{s1.1}
\subsubsection{\underline{Множества}}\label{1_1_1}
Как обычно,  $\NN$ --- множество {\it натуральных чисел,\/} а $\RR$ и  $\CC$
 --- множества соответственно всех {\it  вещественных\/} и  {\it комплексных чисел.\/} Через $\CC_{\infty}:=\CC\cup\{\infty\}$ обозначаем {\it сферу Римана,}\/ или расширенную комплексную плоскость,  для которой здесь мы предпочитаем избегать использования естественного  сферического расстояния и для $z\in \CC_{\infty}$ через $|z|$ обозначаем обычный модуль числа $z\in \CC$, а  $|\infty|:=+\infty$. Для подмножества $S\subset \CC_{\infty}$ через $\clos S:=:\overline{S}$, $\partial S$ и  $\complement S:= \CC_{\infty}\setminus S$ обозначаем соответственно {\it замыкание\/}  в $\CC_{\infty}$, {\it границу\/} $S$ в $\CC_{\infty}$ и {\it дополнение\/} $S$ до $\CC_{\infty}$. Подмножество $S_0\subset S\subset \CC_{\infty}$ --- {\it  собственное подмножество в\/} $S$, если $S_0\neq S$. 
Для $S_0 \subset S\subset \CC_{\infty}$ пишем $S_0\Subset S$, если $S_0$ --- {\it относительно компактное подмножество\/} в $S$ в топологии, индуцированной с $\CC_{\infty}$ на $S$. Для   
 $r\in (0,+\infty]$ и $z\in \CC$ полагаем $	D(z,r):=\{z' \in \CC \colon |z'-z|<r\}$  ---  открытый  круг с центром $z$ радиуса $r$. Таким образом, $D(r)=D(0,r)$ и $D(z,+\infty):= \CC$. Для $z=\infty$ нам удобно принять $D(\infty,r):=\{\infty\}\cup \{z\in \CC \colon |z|>1/r\}$ и $D(\infty, +\infty):=\CC_{\infty}\setminus \{0\}$. Кроме того, $\overline{D}(z,r):=\clos D(z,r)$ --- замкнутый круг с центром $z$ радиуса $r>0$, 
но $\overline{D}(z,0):=\{z\}$, $\overline{D}(+\infty):=\CC_{\infty}$.  Таким образом, 
открытые (замкнутые с непустой внутренностью) круги ненулевого радиуса с центром $z\in \CC_{\infty}$ образуют открытую (соответственно замкнутую) базу окреcтностей точки $z\in  \CC_{\infty}$.

{\it Пустое  множество\/}  обозначаем символом $\varnothing$.
Одним и тем же символом $0$ обозначаем, по контексту, число нуль, нулевой вектор, нулевую функцию, нулевую меру и т.\,п. 
 Для подмножества  $X$ упорядоченного векторного пространства  чисел, функций, мер и т.\,п. с отношением порядка \;$\geq$\;
полагаем   	$X^+:=\{x\in X\colon x \geq 0\}$
 --- все положительные элементы из $X$; $x^+:=\max \{0,x\}$. 
{\it Положительность\/} всюду понимается, в соответствии с контекстом, как $\geq 0$.  
{\it На множествах функций\/}   с упорядоченным  множеством значений {\it отношение  порядка\/} индуцируется с множества значений как  {\it поточечное.} 

\subsubsection{\underline{Функции}}\label{1_1_2}  Функция $f\colon X\to Y$ с упорядоченными  $(X, \leq)$, $(Y, \leq)$ {\it возрастающая}, если для любых $x_1,x_2\in X$
из $x_1\leq x_2$ следует $f(x_1)\leq f(x_2)$, и {\it строго возрастающая,\/} если из $x_1< x_2$ следует $f(x_1)< f(x_2)$. Аналогично для (строгого) убывания. Функция  (строго) возрастающая или  убывающая --- ({\it строго\/}) {\it монотонная.} 

Для топологического пространства $T$, как обычно, $C(T)$ --- векторное пространство над $\RR$ или $\CC$ непрерывных функций со значениями  в $\RR$ или $\CC$. 

Через $\dist (\cdot , \cdot) $ обозначаем {\it функции евклидова расстояния\/} между парами точек, между  точкой и множеством, 
между множествами в $\CC$. По определению $\dist (\cdot , \varnothing):=:\dist (\varnothing,\cdot ):=\inf \varnothing:=+\infty
=:\dist (z,\infty):=:\dist (\infty, z)$ при  $z\in \CC$. 
	
	Для открытого множества   $\mathcal O \subset \CC_{\infty}$ через  $\Hol ({\mathcal O})$  и $\Har ({\mathcal O})$ обозначаем  векторные  пространства соответственно над полем $\CC$ {\it голоморфных\/} и над полем $\RR$ {\it гармонических\/} в ${\mathcal O}$ функций; $\sbh ({\mathcal O})$ --- выпуклый конус над $\RR^+$ {\it субгармонических\/} в ${\mathcal O}$ функций \cite{Rans}--\cite{HK}.  Субгармоническую функцию, {\it тождественно равную\/} $-\infty$ на ${\mathcal O}$, обозначаем символом $\boldsymbol{-\infty}\in \sbh ({\mathcal O})$. Как обычно, функция $u$ голоморфная или (суб)гармоническая в открытой  окрестности  $\mathcal O$ точки $\infty \in \CC_{\infty}$, если при {\it инверсии\/} $\star \colon z\mapsto z^{\star}:=1/\bar z$, $z\in \CC_{\infty}$, 
		функция $u^{\star}\colon z\mapsto u(z^{\star})$, $z\in \CC_{\infty}$, такая же, но уже в окрестности $\mathcal O^{\star}:=\{z^{\star}\colon z\in \mathcal O\}$ точки $0\in \CC_{\infty}$.
			Для произвольного подмножества $S\subset \CC_{\infty}$ класс 
		 $\sbh (S)$ состоит из сужений на $S$ функций, субгармонических в некотором, вообще говоря, своём открытом множестве $\mathcal O\subset \CC_{\infty}$, включающем в себя $S$; \; $\sbh^+(S):=\bigl(\sbh(S)\bigr)^+$. 	
			Для ${u}\in \sbh ( S)$ полагаем ${(-\infty)}_{u}(S):=\{z\in S \colon {u}(z)=-\infty\}$ --- это {\it $(-\infty)$-мно\-ж\-е\-с\-т\-во\/} функции ${u}$  в $S$ \cite[3.5]{Rans}, где зачастую пишем просто ${(-\infty)}_{u}$, не указывая множество $S$.

Для функции  или меры  $a$ её сужение на множество $X$ обозначаем как $a\bigm|_X$.

\subsubsection{\underline{Меры}}\label{1_1_3}  Далее $\mathcal M (S)$ --- класс {\it борелевских вещественных мер,\/} или {\it зарядов,\/}  на   $S\subset \CC_{\infty}$;   
$ \mathcal M_{\comp}(S)$  --- подкласс мер в $\mathcal M (S)$ с компактным {\it носителем\/} $\supp \nu\subset S$, 
$\mathcal M^+ (S):=\bigl(\mathcal M (S)\bigr)^+$. Мера $\mu \in  \mathcal M (S)$ {\it сосредоточена\/} на измеримом по мере $\mu$ подмножестве $S_0\subset S$, если $\mu (S')=\mu (S'\cap S_0)$ для любого измеримого по мере  $\mu$ подмножества $S'\subset S$.
Для $\nu\in \mathcal M (S )$  полагаем $\nu (z,r):=\nu \bigl(D(z,r)\bigr)$, если, конечно, $D (z,r)\subset S$.
 
Меру Рисса функции ${u}\in \sbh ({\mathcal O})$  чаще всего будем обозначать как $\nu_{u}:=\frac{1}{2\pi}\Delta {u}\in \mathcal M ^+(\mathcal O)$ или $\mu_u$. 
Так, для  ${u}\neq \boldsymbol{-\infty}$ её мера Рисса  $\nu_{u}$ --- борелевская положительная мера \cite{Rans}--\cite{HK}. Для  функции же $\boldsymbol{-\infty}\in \sbh({\mathcal O})$ её мера Рисса  по определению равна $+\infty$ на любом подмножестве из ${\mathcal O}$.

\subsubsection{\underline{Нули голоморфных функций}}\label{zf}
Пусть $0\neq f\in \Hol({\mathcal O})$. Функция $f$ {\it обращается в нуль на последовательности точек\/ 
${\tt Z}=\{{\tt z}_k\}_{k=1,2,\dots}$,}
 лежащих в $\mathcal O$ (пишем $\tt Z\subset \mathcal O$), если кратность нуля, или корня,  функции $f$ в каждой точке $z \in \mathcal O$ не меньше числа повторений этой точки  в последовательности ${\tt Z}$ (пишем $f({\tt Z})=0$).
Каждой последовательности $\tt Z\subset \mathcal O$ без точек сгущения в $\Omega$ сопоставляем её {\it считающую меру\/}
\begin{equation}\label{df:nZS}
	n_{\tt Z}(S):=\sum_{{\tt z_k}\in S} 1 \quad \text{\it для  произвольных $S\subset D$}
\end{equation}
 --- число точек из $\tt Z$, попавших в $S$.
Последовательность нулей (корней) функции $f\in \Hol (D)\setminus \{0\}$, каким-либо образом перенумерованную с учетом кратности, обозначаем через $\Zero_f $. Так как  $\log |f|\in \sbh (\mathcal O)$,  взаимосвязь меры Рисса $\nu_{\log |f|}$   со считающей мерой  её нулей \eqref{df:nZS} задаётся равенством \cite[Theorem 3.7.8]{Rans}
\begin{equation}\label{nufZ}
	\nu_{\log |f|}=\frac1{2\pi} \Delta \log |f|\overset{\eqref{df:nZS}}{=}n_{\Zero_f}.
\end{equation}
В частности, из $f({\tt Z})=0$ следует $n_{\tt Z}\leq n_{\Zero_f}$ на $\mathcal O$, и наоборот.

\subsubsection{\underline{Некоторые соглашения}}\label{ns} Число $C$ и постоянную функцию, тождественно равную  $C$ не различаем. Через $\const_{a_1, a_2, \dots}$ обозначаем вещественные постоянные, вообще говоря, зависящие от $a_1, a_2, \dots$ и, если не оговорено противное, только от них; $\const_{a_1, a_2, \dots}^+\geq 0$.

Ссылка над знаками (не)равенства, включения, или, более общ\'о, бинарного отношения, означает, что при переходе к правой части этого отношения применялись, в частности, и отмеченная  формула, определение, обозначение или  утверждение. 

(Под)область в $\CC_{\infty}$ --- открытое связное подмножество в $\CC_{\infty}$.
\underline{Далее всюду } 
	
	\begin{equation*}
						\boxed{\varnothing \neq D_0\Subset D \neq \CC_{\infty},  \text{ {\sffamily{где $D_0,D$ --- подобласти в}} $\CC_{\infty}$.}}
	\eqno{\rm(D)} 
\end{equation*}
Основной объект --- область $D$, а роль подобласти $D_0\Subset D$ аналогична роли фиксированной точки области (множества) или  начала отсчёта, или координат, в пространстве. Она,  --- во всяком случае,  в начале формулировок основных утверждений, --- может выбираться произвольно. Как правило, чем б\'ольшей она выбрана, тем, формально, соответствующее утверждение  --- более общее.
  
\subsubsection{\underline{Тестовые субгармонические функции}}\label{tsf} Аналог основных (финитных) функций теории распределений Л.Шварца, или обобщённых функций,  в нашей работе  предоставляет
\begin{definition}\label{testv} Функцию $v\in \sbh^+ (D\setminus D_0)$ называем {\it тестовой\/} ({\it для области $D$ вне подобласти  $D_0$}), если 
			\begin{equation}\label{resv0}
	\lim_{D\ni z'\to z}v(z') =0 \quad \text{\it для любой точки $z\in \partial D$}.	
	\end{equation}
	{\it Класс всех тестовых функций\/} обозначаем через $\sbh_0^+(D\setminus D_0)$. Для каждого числа $b\geq 0$ определим его  {\it подкласс} 
\begin{equation}\label{sbh+}
	 \sbh_0^+(D\setminus D_0;\leq b):=\left\{v\in \sbh_0^+(D\setminus D_0)\colon 
	 \sup_{z\in \partial D_0}v(z)\leq b \right\},
\end{equation}
\end{definition}

\begin{remark}\label{rmp}  Из принципа максимума для субгармонических функций
\begin{equation}\label{sbh+0b}
	 \sbh_0^+(D\setminus D_0) =\bigcup_{b\geq 0} \sbh_0^+(D\setminus D_0;\leq b),
\end{equation}
и в определении \eqref{sbh+}, учитывая \eqref{resv0},  можно заменить $\sup_{z\in \partial D_0}$  на $\sup_{z\in D\setminus  D_0}$. 
\end{remark}
\begin{remark}\label{rmmv} Во всех утверждениях настоящей работы, где используются классы \eqref{sbh+0b} и $\sbh_0^+(D\setminus D_0)$, можно использовать, вообще говоря, и более широкие классы тестовых функций соответственно $\sbh_0^+\bigl(D\setminus \clos{D_0}; \leq b\bigr)$
\begin{subequations}\label{df:mv}
\begin{align} 
 :=\Bigl\{v\in \sbh^+(D\setminus \clos {D_0}\,)&\colon \lim_{D\ni z'\to z}v(z') \overset{\eqref{resv0}}{=}0, \;
	 \limsup_{\clos {D_0} \not{\ni} z\to  \partial D_0}v(z)\leq b \Bigr\} ,
\tag{\ref{df:mv}b}\label{df:mvb}\\ 
  \sbh_0^+\bigl(D\setminus \clos{D_0}\bigr) &:= \bigcup_{b\geq 0} \sbh_0^+\bigl(D\setminus \clos{D_0}; \leq b\bigr) 
\tag{\ref{df:mv}$\infty$}\label{df:mvi}
\end{align}
\end{subequations}
с соответствующими заменами интегралов по множеству $D\setminus D_0$ и сумм по точкам в таких множествах соответственно интегралами и суммами по множеству  $D\setminus \clos D_0$. Формально это несколько обобщает наши результаты, но по существу незначительно, поскольку небольшие <<раздутия>>    подобласти   $D_0$ с сохранением   рамочного условия (D) практически не ослабляют результаты, но уже переводят классы 
  \eqref{df:mv} в классы из Определения \ref{testv}. В то же время использование классов  \eqref{df:mv} требует, во всяком случае в наших  доказательствах,  большей детализации и затеняет основную их линию. 
	\end{remark}
	\begin{remark}\label{1_3}
	Часто также будет использоваться <<минимальный>> класс тестовых функций
		\begin{equation}\label{mint}
	\sbh_0^+\bigl(D\setminus \{z_0\}\bigr)	:=
	\Bigl\{v\in \sbh^+(D\setminus \{z_0\})\colon \lim_{D\ni z'\to z}v(z') \overset{\eqref{resv0}}{=}0\Bigr\},
	\end{equation}
	 который формально можно считать пересечением всех классов 
	\eqref{sbh+0b}, или \eqref{df:mvi}, по всем подобластям $D_0\ni z_0$ из рамочного соглашения (D) п.~\ref{ns} Введения.
\end{remark}
\subsection{Частные случаи Основной Теоремы}\label{CorMT}
Одно из следствий   нашей Основной Теоремы   из \S~\ref{S2}, содержащего как подготовительные, так и основные  результаты с полными доказательствами,    может быть сформулировано как
\begin{corollary}\label{cor:1} Пусть  $ M\in \sbh(D)\setminus \{\boldsymbol{-\infty}\}$ с мерой Рисса $\nu_M$,  
$u\in \sbh(D)$ с мерой Рисса $\nu_u$, а также  $u\leq M$ на $D\setminus D_0$.
Тогда  для любой  тестовой функции $v\in \sbh_0^+(D\setminus D_0)$  верна импликация
\begin{equation}\label{vnuum}	
	\left(\int_{D\setminus D_0}v\,d\nu_M<+\infty\right)\;\Longrightarrow \;
	\left(\int_{D\setminus D_0} v\, d \nu_u<+\infty\right).
\end{equation}
Так, если $0\neq f\in \Hol (D)$ и $f$ обращается  в нуль на\/ ${\tt Z}=\{{\tt z}_k\}_{k=1,2,\dots}\subset D$,  а $\log |f|\leq M$ на $D\setminus D_0$, то  
для любой $v\in \sbh_0^+(D\setminus D_0)$ справедлива импликация 
\begin{equation}\label{vnuumf}
\hspace{-4,5mm}\left(\int_{D\setminus D_0}v\,d\nu_M<+\infty\right)\Longrightarrow 
\left(	\sum_{{\tt z}_k\in D\setminus D_0} v({\tt z}_k)<+\infty\right).
\end{equation}
\end{corollary}

\begin{theorem}\label{th:1} 
В условиях Следствия\/ {\rm \ref{cor:1}} при  $z_0\in D_0\setminus (-\infty)_M$  	для любого числа  $b>0$  	найдутся  постоянные  
\begin{equation}\label{cz0C+}
C :=\const_{z_0,D, D_0,b}^+, \qquad           \overline{C}_M := \const_{z_0,D, D_0,M}^+,
\end{equation}
с которыми  для любой функции $v\overset{\eqref{sbh+}}{\in} \sbh_0^+(D\setminus D_0;\leq b)$ выполнено  неравенство
 \begin{equation}\label{mest+}
\int_{D\setminus D_0}  v \,d {\nu}_u 		 \leq 	\int_{D\setminus D_0}  v \,d {\nu}_M	+		C\,	\overline{C}_M -C\, u(z_0) .
\end{equation}
В частности, если $f\in \Hol (D)$ --- ненулевая функция, обращающаяся в нуль на последовательности\/ ${\tt Z}=\{{\tt z}_k\}_{k=1,2,\dots}\subset D$,  а $\log |f|\leq M$ на $D\setminus D_0$, то для любой тестовой функции $v\overset{\eqref{sbh+}}{\in} \sbh_0^+(D\setminus D_0;\leq b)$ справедливо неравенство 
\begin{equation}\label{in:fz0}
	\sum_{{\tt z}_k\in D\setminus D_0} v({\tt z}_k)\leq \int_{D\setminus D_0} v\,d\nu_M+	C \,\overline{C}_M -C \,\log \bigl|f(z_0)\bigr|.
\end{equation}
При этом постоянная  $\overline{C}_M$ из \eqref{cz0C+} положительно однородна  по $M$, т.\,е. 	 $\overline{C}_{aM}=a\overline{C}_M$ для любого  $a\in  \RR^+$, и  полуаддитивна сверху, т.\,е. 		$\overline{C}_{M_1+M_2}\leq \overline{C}_{M_1}+\overline{C}_{M_2}$ при любых  $M_1, M_2 \in \sbh (D)\setminus \{\boldsymbol{-\infty}\}$.

\end{theorem}  

\begin{remark}\label{addnu} Если в Следствии \ref{cor:1} или в Теореме \ref{th:1} известно, что $\nu_M\leq \nu^0\in \mathcal M^+(D)$ и $ \mathcal M^+(D)\ni \nu_0\leq \nu_u$ на $D$, то в заключениях \eqref{vnuum} и \eqref{mest+} можно заменить $\nu_u$ на $\nu_0$ и $\nu_M$ на $\nu^0$, 
поскольку тестовые функции положительны. При этом, конечно, постоянные из \eqref{cz0C+} становятся зависимыми от разности мер $\nu^0-\nu_M$, а $u(z_0)$ в правой части \eqref{mest+} надо заменить на некоторое другое число, но  в дополнительном изначальном  предположении, что $u(z_0)\neq -\infty$. 
\end{remark}

\section{Основные результаты}\label{S2}

\subsection{$\delta$-субгармонические функции}\label{2_3} Основным руководством по  $\delta$-суб\-г\-а\-р\-м\-о\-н\-и\-ч\-е\-с\-ким функциям  могут служить работы \cite{Ar_d}--\cite{Gr}.  Здесь используем наиболее подходящую  для рассматриваемой здесь проблематики   
трактовку разности субгармонических функций, близкую к    \cite[2]{Gr}. Стандартные алгебраические действия над элементами расширенной числовой прямой $[-\infty,+\infty]:=\{-\infty\}\cup \RR \cup \{+\infty\}$ естественным образом, как и порядковая структура, переносятся с $\RR$.
  Наряду с функцией $\boldsymbol{-\infty}$ используем и функцию $\boldsymbol{+\infty}$, тождественно равную $+\infty$.
	Функцию $M\colon D\to [-\infty,+\infty]$ в области $D\subset \CC_{\infty}$ 
называем {\it тривиальной $\delta$-суб\-г\-ар\-м\-о\-н\-и\-ч\-е\-с\-кой,\/} если $M=\boldsymbol{-\infty}$ или $M=\boldsymbol{+\infty}$ на $D$, и   
{\it нетривиальной $\delta$-суб\-г\-ар\-м\-о\-н\-и\-ч\-е\-с\-кой функцией с\/} (вещественной борелевской) {\it  мерой Рисса\/} $\nu_M\in \mathcal M(D)$, или {\it зарядом Рисса\/} $\nu_M$, если выполнены  следующие три условия-соглашения.
\begin{enumerate}
	\item\label{da} {\it Существуют две функции $ u_1, u_2\in \sbh (D)\setminus \{\boldsymbol{-\infty}\}$ с мерами Рисса соответственно $\nu_{u_1}, \nu_{u_2}\in \mathcal M^+(D)$, для которых $M(z):=u_1(z)-u_2(z)\in \RR$ для всех $z\notin (-\infty)_{u_1}\cup (-\infty)_{u_2}$.\/} При этом для заряда (меры)  Рисса $\nu_M:=\nu_{u_1}- \nu_{u_2}\in \mathcal M(D) $  однозначно определено разложение Хана--Жордана $\nu_M:=\nu_M^+-\nu_M^-$ на разность взаимно сингулярных положительной и отрицательной вариаций $\nu_M^+, \nu_M^-\in \mathcal M^+(D)$, $|\nu_M|:=\nu_M^++ \nu_M^-$ --- полная вариация меры $\nu_M$. Для полунепрерывной сверху   локально ограниченной сверху  функции $v\colon D\to \RR^+$ полагаем
		\begin{equation}\label{intvnuM}
	\int_D v \, d 	\nu_M:=\int_D v \, d 	\nu_M^+ -\int_D v \, d 	\nu_M^-
	\end{equation}
	при условии конечности хотя бы одного из  интегралов справа.
	\item\label{dad} {\it Определяющее множество ${\dom M}\subset D$ --- это множество точек $z\in D$, для каждой из которых при некотором  $r_z>0$ конечен интеграл}
				\begin{equation}\label{dMa}
		\hspace{-3mm}\left(\int\limits_0^{r_z}\frac{|\nu_M|(z,t)}{t} \, dt<+\infty\right)\Leftrightarrow
		\left(\int\limits_{D(z,r_z)} \log |z'-z| \, d |\nu_M|(z')>-\infty\right).
		\end{equation}
		{\it При этом   доопределяем функцию $M$ на всех  $z\in \dom M$ по правилу
				\begin{equation*}
		M(z)=\lim_{0<r\to 0} \frac{1}{2\pi} \int_0^{2\pi} M(z+re^{i\theta}) \, d \theta \in \RR	\quad \text{для всех $z\in {\dom M}$},
		\end{equation*}}
		что всегда согласуется с предварительно заданными значениями $M$ в предыдущем п.~\ref{da} на   $\CC_{\infty}\setminus 
		\bigl((-\infty)_{u_1}\cup (-\infty)_{u_2}\bigr)\subset \dom M $. В частности, известно \cite{Rans}--\cite{HK}, что {\it в случае субгармонической\/}  в $D$
		функции $u\neq \boldsymbol{-\infty}$ её определяющее множество $\dom u=D\setminus (-\infty)_u$.

		\item\label{dai} {$M(z)=+\infty$ при\footnote{В \cite[2]{Gr} положено $M(z)=0$ для $z\notin {\dom M}$, но для наших целей предпочтительнее  $=+\infty$, поскольку функция $M$ у нас играет роль	мажоранты.} $z\in D\setminus {\dom M}$.}
\end{enumerate}

Пункт \ref{da}, следуя \cite[Theorem 11]{Ar_d},    
можно заменить, избегая  явного упоминания  субгармонических функций, на 
				\begin{enumerate}
			\item[$1'.$]  {\it $M$ --- локально интегрируемая по мере  Лебега на $D$ функция, обладающая  в смысле теории распределений  
			свойством: для любой подобласти $D'\Subset D$ найдется постоянная $C'\in \RR^+$, с которой для любой финитной дважды непрерывно дифференцируемой  функции  			$f \colon D'\to \RR$ с носителем $\supp f \subset D'$ имеет место неравенство
			\begin{equation*}
			\Bigl|\int_{D'} M	\Delta f \, d\lambda \Bigr|\leq C' \max_{z\in D'} \bigl|f (z)\bigr|,\quad \text{$\lambda$ --- мера Лебега на $D$.}
			\end{equation*}
			При этом $\nu_M:=\frac{1}{2\pi} \Delta M$ в смысле теории распределений.}
		\end{enumerate}
		При таком подходе    п.~\ref{dad} заменяем  соответственно на 
				\begin{enumerate}
			\item[$2'.$]  {\it Для определяющего множества ${\dom M}\subset D$, определяемого, как и выше, условием конечности интегралов вида  \eqref{dMa}, полагаем 
			\begin{equation*}
			M(z)=\lim_{0<r\to 0} \frac{1}{\pi r^2} \int_{D(z,r)} M \, d \lambda 	\quad \text{для всех $z\in {\dom M}$},
			\end{equation*}
			 где $D\setminus {\dom M}$ --- нулевой ёмкости. В частности, $\lambda (D\setminus {\dom M})=0$.}
		\end{enumerate}
		Пункт~\ref{dai} сохраняем или считаем функцию $M$ неопределённой в $D\setminus {\dom M}$.

Класс $\delta$-субгармонических в $D$ функций обозначаем через $\dsbh (D)$.

\subsection{{\sc  Основная Теорема}}\label{ss:23} {\it Пусть\/  $\boldsymbol{\pm\infty} \neq M\in \dsbh (D)$ с зарядом Рисса 
$\nu_M:=\frac{1}{2\pi}\Delta M$ и   определяющим множеством\/   ${\dom M}\subset D$  	{\rm{\large(}см. \ref{2_3};  \ref{dad}, \eqref{dMa}{\large)}}. 
Тогда   для любых 
\begin{enumerate}[{\rm (i)}]
	\item\label{zbi+} точки $z_0\in D_0\cap {\dom M}$ и числа  $b>0$,
			\item\label{Diii+} регулярной\footnote{Простые и наглядные достаточные условия регулярности области геометрического характера см., например, в \cite[Теорема 2.11]{HK}, а другие --- в \cite[4.2]{Rans}.}
		области $\widetilde{D}\subset \CC_{\infty}$ с функцией Грина $g_{\widetilde{D}}(z, z_0)$ с полюсом в точке $z_0$ при  $D_0\Subset \widetilde{D}\subset D$  и $\CC_{\infty}\setminus \clos \widetilde{D}\neq \varnothing$ 
	\end{enumerate}
	 найдётся  число 
\begin{equation}\label{cz0C}
C:=\const_{z_0,D_0,\widetilde{D},b}^+:= \frac{b} {\inf\limits_{z\in \partial D_0}  g_{\widetilde{D}}(z, z_0)}>0,
\end{equation}
с которыми для любой функции $u\in \sbh (D)\setminus \{\boldsymbol{-\infty}\}$, удовлетворяющей неравенству $u\leq M$ на $D\setminus D_0$, 
а также 
\begin{enumerate}
		\item[{\rm (iii)}]\label{vii+}  для любой функции $v\overset{\eqref{sbh+}}{\in}  \sbh_0^+(D\setminus D_0;\leq b) $
		\end{enumerate}
выполнено неравенство
\begin{subequations}\label{mest}
\begin{align}C u(z_0) 	+\int\limits_{D\setminus D_0}  v \,d {\nu}_u 		&\leq	\int\limits_{D\setminus D_0}  v \,d {\nu}_M	+\int\limits_{\widetilde{D}\setminus D_0} v \,d {\nu}_M^-   +C\, \overline{C}_M,
\tag{\ref{mest}v}\label{{mest}v}
\\	
\text{где}\quad
\overline{C}_M:=	\int\limits_{\widetilde{D}} g_{\widetilde{D}}(\cdot, z_0)  \,d {\nu}_M  
		&+\int\limits_{\widetilde{D}\setminus D_0} g_{\widetilde{D}}(\cdot, z_0)  \,d {\nu}_M^-  +M(z_0). 
\tag{\ref{mest}C}\label{{mest}C}
\end{align}
\end{subequations}
   Для величины $\overline{C}_M$ из    \eqref{{mest}C} возможно значение $+\infty$, но при $\widetilde{D}\Subset D$ --- это  некоторая постоянная 
	$\overline{C}_M:=\const_{z_0,D_0, \widetilde{D},M}^+<+\infty$, положительно однородная  и  полуаддитивная сверху по $M$.}

\begin{remark}\label{remw} Основная Теорема верна   и для тестовых  функций из Замечания \ref{rmmv}, т.\,е. с функциями 
$v\overset{\eqref{df:mvi}}{\in} \sbh_0^+(D\setminus \clos D_0)$ в п.~(iii). 
\end{remark}

\subsection{Меры и потенциалы Йенсена}\label{i:mpJ}\label{mJ:Ex}  

\begin{definition}[\cite{Khab91}, \cite{Gamelin}--\cite{Khab03}]\label{df:J} Меру $\mu \in \mathcal M^+(\CC_{\infty})$ называем  {\it мерой Йенсена внутри области\/ $D$ в точке\/ $z_0\in D$,\/} если 
$\mu \in \mathcal M^+_{\comp}(D)$, , т.\,е. $\supp \mu \subset D$,  и 
\begin{equation}\label{df:mJ}
	u(z_0)\leq \int u\,d \mu \quad    \text{\it  для всех  $u\in \sbh(D)$}.
\end{equation}
 Класс всех таких мер Йенсена обозначаем через $J_{z_0}(D)$. 
\end{definition}
Очевидно, по Определению \ref{df:J} каждая мера  $\mu \in J_{z_0}(D)$ {\it вероятностная,\/} т.\,е. $\mu (D)=1$. Далее неоднократно используется то, что {\it для любого множества нулевой ёмкости\/} $E\subset D$, в частности, при  $E=(-\infty)_u$, $\boldsymbol{-\infty}\neq u\in \sbh (D)$, {\it для меры $\mu \in J_{z_0}(D)$ имеем $\mu \bigl(E\setminus \{z_0\}\bigr)=0$ \cite[Corollary 1.8]{CR}.}

 \begin{example}\label{ex:1hm} Пусть $\widetilde{D}$ --- область в $D$ с неполярной границей,  $\CC_{\infty}\setminus \clos \widetilde{D}\neq \varnothing$. 
{\it Гармоническая мера\/} $\omega_{\widetilde{D}}(z_0, \cdot)$ {\it для\/} (или относительно) $\widetilde{D}$ {\it в точке\/} $z_0\in \widetilde{D}$ 
{\large(}см. \cite[4.3]{Rans}, \cite[3.6]{HK}, \cite[5.7.4]{HK}{\large)}
при условии $\widetilde{D}\Subset D$   --- пример меры Йенсена из $J_{z_0}(D)$. Вырожденный случай   --- мера Дирака  $\delta_{z_0}$ в точке $\{z_0\}$, т.\,е. единичная масса, сосредоточенная в точке $z_0$. 
 \end{example}

\begin{definition}[{\large(}\cite{Khab91}, \cite{Khab01}, \cite{Khab99VB}--\cite{Khab99}, {\rm  ср. c {\cite[3.1, 3.3]{Gamelin}; ранние частные случаи  --- в \cite{Sar71}--\cite{And81}}}{\large)}]\label{df:PoJ}
 Функцию $V$ из кла\-с\-са $\sbh^+\bigl(\CC_{\infty} \setminus \{z_0\}\bigr)$ 
называем {\it потенциалом Йенсена  внутри  области $D$ с полюсом в точке  $z_0\in D$,} если выполнены два условия:
\begin{enumerate}[{\rm 1)}]
\item\label{V:f2} {\it найдётся  область $D_V\Subset D$, содержащая  точку $z_0\in D_V$, для которой $V (z)\equiv 0$ при всех   $z\in
\CC_{\infty} \setminus D_V$, т.\,е. $V\bigm|_{\CC_{\infty} \setminus D_V}=0$} ( финитность в $D$);  
\item\label{V:f3}  задана  {\it логарифмическая полунормировка в точке\/ $z_0$,} а именно: 
\begin{subequations}\label{nvz}
\begin{align}
\limsup\limits_{z_0\neq z \to z_0}\dfrac{V(z )}{l_{z_0}(z)}&\leq  1, 
 \tag{\ref{nvz}o}\label{nvzn}\\ 
\text{\it где}\quad l_{z_0}(z):=&\begin{cases}
\log \frac1{|z-z_0|} \quad&\text{\it при $z_0\neq \infty$},\\
\log |z| \quad&\text{\it при $z_0= \infty$}.
\end{cases}
\tag{\ref{nvz}l}\label{{nvz}INF}
\end{align}
\end{subequations}
Условия \eqref{nvz} можно записать в эвивалентной форме
\begin{subequations}
\begin{align*}
\sup\limits_{z\neq z_0}\left (V(z )-\log^+\frac1{|z-z_0|}\right)&<+\infty \quad \text{при $z_0\in \CC$}  ,
 \\ 
 \sup\limits_{z\neq \infty}\bigl (V(z )-\log^+ |z|\bigr)&<+\infty \quad \text{при $z_0= \infty$} .
\end{align*}
\end{subequations}

\end{enumerate}
Класс всех  таких потенциалов Йенсена обозначаем через  $PJ_{z_0} (D)$. 
\end{definition}
Имеет место очевидное (см. Определение \ref{testv})
\begin{propos}\label{PJ=test}
 Каждый потенциал Йенсена из $PJ_{z_0} (D)$ при $z_0\in D_0$ --- тестовая функция для $D$ вне $D_0$, 
т.\,е. $PJ_{z_0} (D) \subset \sbh_0^{+}(D\setminus D_0)$ для $z_0\in D_0$.  
 \end{propos}
\begin{example}\label{exp2}  Пусть $\widetilde{D}$ --- область в $\CC_{\infty}$ с неполярной границей, 
$\CC_{\infty}\setminus \clos \widetilde{D}\neq \varnothing$. {\it Функция Грина  $g_{\widetilde{D}}(\cdot ,z_0)$} {\it для} (или относительно) $\widetilde{D}${\; \it  с полюсом в точке $z_0\in \widetilde{D}$, продолженная на\/} $\CC_{\infty}$ по правилу
{\large(}см. \cite[4.4]{Rans}, \cite[3.7, 5.7]{HK}{\large)}
\begin{equation}\label{gD0r}
	g_{\widetilde{D}}(z, z_0):=\begin{cases}
	0 \quad &\text{при $z\in \CC_{\infty}\setminus \clos \widetilde{D}$},\\
	\limsup\limits_{\widetilde{D}\ni z'\to z} g_{\widetilde{D}}(z', z_0)\quad &\text{при $z\in \partial \widetilde{D}$}.
	\end{cases}
\end{equation}
   --- пример потенциала  Йенсена из $PJ_{z_0}(D)$  при условии $\widetilde{D}\Subset D$. Вырожденный вариант   --- 
функция, тождественно равная нулю на $\CC_{\infty}\setminus \{z_0\}$. 
 \end{example}
	
	{\it Логарифмический потенциал рода\/ $0$ вероятностной меры\/ $\mu \in \mathcal  M^+_{\comp}(\CC_{\infty})$  с полюсом в точке $z_0\in \CC_{\infty}$} определяем для всех $w\in \CC_{\infty}\setminus \{z_0\}$ как функцию
\begin{subequations}\label{df:Vmu}
	\begin{align}
	V_{\mu}(w ) 	&:= 		\int_{D} \log \Bigl|\frac{w-z}{w-z_0}\Bigr| \,d \mu (z)
	= \int_{D} \log \Bigl|1-\frac{z-z_0}{w-z_0}\Bigr| \,d \mu (z) \quad\text{\it при $z_0\neq \infty$},
	\tag{\ref{df:Vmu}o}\label{{df:Vmu}o}
	\\
	\intertext{где при $w=\infty$ подынтегральные выражения доопределены значением  $0$,} 
	V_{\mu}(w ) 	&:=
	\int_{D} \log \Bigl|\frac{w-z}{z}\Bigr| \,d \mu (z)=
	\int_{D} \log \Bigl|1-\frac{w}{z}\Bigr| \,d \mu (z) \quad\text{\it при $z_0= \infty$},
\tag{\ref{df:Vmu}$\infty$}\label{{df:Vmu}in}
		\end{align}
\end{subequations}
где при $z=\infty$ подынтегральные выражения доопределены значением  $0$.

Отметим основные взаимосвязи между мерами и потенциалами Йенсена. 
Первое --- следующее утверждение о двойственности:
\begin{propos}[{\cite[Предложение 1.4, Теорема двойственности]{Khab03}}]\label{pr:1} Отображение  $\mathcal P$, 
определяемое по правилу 
\begin{equation}\label{con:P}
	\mathcal P \colon J_{z_0}(D)\to PJ_{z_0} (D), \quad 
	\mathcal P (\mu)\overset{\eqref{df:Vmu}}{:=} V_{\mu}, \quad \mu \in  J_{z_0}(D) 
\end{equation}
--- аффинная\footnote{Здесь это означает, что $\mathcal P\bigl(t\mu_1+(1-t)\mu_2\bigr)=t\mathcal P (\mu_1)+(1-t)\mathcal P (\mu_2)$ для всех $t\in [0,1]$.} биекция, а 
\begin{equation}\label{eq:mu}
	{{\mathcal P}}^{-1}(V)\overset{\eqref{{nvz}INF}}{=}\frac1{2\pi} \Delta  V\Bigm|_{D\setminus \{z_0\}}+
	\Bigl(1-\limsup\limits_{z_0\neq z \to z_0}\dfrac{V(z )}{l_{z_0}(z)}\Bigr)\cdot
	{\delta}_{z_0}\, , \quad V\in PJ_{z_0}(D).
\end{equation}
В частности,  для регулярной  области $\widetilde{D}\Subset D$ при $z_0\in \widetilde{D}$ --- классическое
\begin{equation*}
	\mathcal P\bigl(\omega_{\widetilde{D}}( z_0, \cdot )\bigr)=g_{\widetilde{D}}(\cdot, z_0 ), \quad z_0\in \widetilde{D}\Subset D.
\end{equation*}
\end{propos}

Второе   ---  это {\it обобщённая формула Пуассона--Йенсена:\/}
\begin{propos}[{\cite[Предложение 1.2]{Khab03}}]\label{pr:2}
Пусть $\mu \in J_{z_0}(D)$. Тогда для любой функции $u\in \sbh (D)$ с мерой  Рисса
$\nu_u$ при $u(z_0)\neq -\infty$ имеем равенство
\begin{equation}\label{f:PJ}
	u(z_0) +\int_{D\setminus \{z_0\}} V_{\mu} \,d {\nu}_u=\int_{D} u \,d \mu  .
\end{equation}
В частности, для регулярной области $\widetilde{D}\Subset D$ --- это классическая формула
\begin{equation}\label{f:PJgo}
	u(z_0) +\int_{\widetilde{D}\setminus \{z_0\}} g_{\widetilde{D}}(\cdot, z_0) \,d {\nu}_u
	=\int_{\widetilde{D}} u\,d \omega_{\widetilde{D}}(z_0,\cdot)  .
\end{equation}
\end{propos}

\subsection{Тестовые функции и  потенциалы Йенсена}\label{ss:2_2} Отдельные элементы предлагаемых ниже конструкций в очень  частных случаях уже успешно использовались нами ранее, например, в \cite{Khab91}, \cite{Khab91s}, \cite{Khab99}  по различным поводам.

\begin{propos}\label{exttf} Для каждой тестовой функции $v\overset{\eqref{sbh+0b}}{\in} \sbh_0^+(D\setminus D_0)$ её, обозначаемое тем же символом $v$, продолжение 
\begin{equation}\label{d:exvv}
	v(z):=\begin{cases}
v(z) \quad &\text{для $z\in D\setminus D_0$},\\
0 \quad &\text{для $z\in \CC_{\infty}\setminus D$}		
	\end{cases}
\end{equation}
--- субгармоническая функция  в $\CC_{\infty}\setminus D_0$, равномерно стремящаяся  к нулю при приближении к\/ $\CC_{\infty}\setminus  D$ в том смысле, что для любого числа $\e >0$  найдётся подобласть $D_{\e} \Subset D$, для которой $D_0\subset D_{\e}$ и  
 \begin{equation}\label{0ev}
	 0\leq v(z)<\e \quad\text{во  всех точках  $z\in \CC_{\infty} \setminus D_{\e}$}
 \end{equation}
 {\rm (ср. с более жёсткой  финитностью из п.~{\rm \ref{V:f2}} Определения \ref{df:PoJ})}. Свойство, выражаемое равномерными соотношениями 
\eqref{0ev} будем записывать далее в виде
\begin{equation}\label{est:u0} 
\lim_{D\ni z\to \partial D}v(z)=0.
\end{equation}
 \end{propos}
\begin{proof} Для каждой точки $z\in \partial D$ по Определению \ref{testv} тестовых функций 
	$\lim_{D\ni z'\to z}v(z')\overset{\eqref{resv0}}{=}0$ и для продолженной функции \eqref{d:exvv} при любом  $\e >0$ найдётся число $r_{\e}(z)$, для которого $0\leq v(z')<\e$ для всех $z'\in D\bigl(z, r_\e (z)\bigr)$. В силу компактности $\partial D$ в $\CC_{\infty}$ можно выбрать конечное число кругов $D\bigl(z_j, r_\e (z_j)\bigr) $, $j=1,2,\dots n<+\infty$, с  которыми  при 
	$$D_{\e}:=D\setminus \bigcup_{j=1}^n D\bigl(z_j, r_\e (z_j)\bigr) $$
		выполнено 	\eqref{0ev}. В частности, функция  \eqref{d:exvv} непрерывная на $\CC_{\infty}\setminus D$, а значит полунепрерывная сверху на  некотором открытом множестве, включающем в себя  $\CC_{\infty}\setminus D_0$. Интегрируемость по кругам с центром  в точках $z\in \CC_{\infty}\setminus D$ и мажорирование средними по ним нулевых значений функции \eqref{d:exvv}  на $\CC_{\infty}\setminus D$ очевидны ввиду положительности этой функции.		
\end{proof}

\begin{remark}\label{rod} Предложение \ref{exttf} и его доказательство переносятся и на  тестовые  функции $v\overset{\eqref{df:mvi}}{\in} \sbh_0^+(D\setminus \clos D_0)$ из Замечания \ref{rmmv}    с той лишь разницей, что в первой строчке правой части   \eqref{d:exvv}  надо заменить $D_0$ на $\clos D_0$.
\end{remark}

\begin{definition}\label{df:ctf} {\it Тестовую функцию\/} $v\overset{\eqref{sbh+0b}}{\in} \sbh_0^+(D\setminus D_0)$, продолженную на $\CC_{\infty}$ по правилу  \eqref{d:exvv},  называем {\it продолженной\/} на $\CC_{\infty}\setminus D_0$ {\it тестовой функцией,\/} обозначая, как правило,  это продолжение той же буквой, или тем же символом,  что  и исходную тестовую функцию {\large(}ср. с \eqref{gD0r} из Примера \ref{exp2}{\large)}. 

Кроме того, если при этом $\mu_v\in \mathcal M^+(D\setminus D_0)$ --- мера Рисса исходной тестовой функции  $v\overset{\eqref{sbh+0b}}{\in} \sbh_0^+(D\setminus D_0)$,  то меру Рисса продолженной тестовой функции, сосредоточенную, вообще говоря, уже на $\clos D\setminus D_0\subset \CC_{\infty}\setminus D_0$, т.\,е. из класса   $\mathcal M^+(\clos D\setminus D_0)$,  также называем {\it продолженной мерой Рисса\/} продолженной тестовой функции $V$ и тоже обозначаем её как  $\mu_v$. 
\end{definition}
\begin{propos}\label{prprm} Пусть $z_0\in \CC_{\infty}$, $p\in \RR^+$ и  функция $V\in \sbh\bigl(\CC_{\infty}\setminus\{z_0\}\bigr)$ с мерой Рисса $\mu_V$  на $\CC_{\infty}\setminus\{z_0\}$ удовлетворяет условию 
\begin{equation*}
\limsup\limits_{z_0\neq z \to z_0}\dfrac{V(z )}{l_{z_0}(z)}\overset{\eqref{nvz}}{=} p, 
\quad \text{\it  где}
\quad l_{z_0}(z)\overset{\eqref{{nvz}INF}}{:=}\begin{cases}
\log \frac1{|z-z_0|} \quad&\text{\it при $z_0\neq \infty$},\\
\log |z| \quad&\text{\it при $z_0= \infty$}.
\end{cases}
\end{equation*}
Тогда 
	$\mu_V \bigl(\CC_{\infty}\setminus\{z_0\}\bigr)=p$.
\end{propos}
\begin{proof} При $z_0\neq \infty$ с помощью  инверсии $\star$, как во Введении в разделе \ref{s1.1}, но, если необходимо, со сдвигом
\begin{equation}\label{inv_z0}
	\star_{z_0} \colon z\longmapsto \frac{1}{\overline{z-z_0}}\,, \quad z\in \CC_{\infty}, \quad \star_{z_0}(z_0)=\infty,
\end{equation}
сохраняющими, как известно, при соответствующей замене переменных субгармоничность и полную меру,
можно всегда перейти к случаю $z_0=\infty$. 
Остаётся воспользоваться представлением Адамара для $V$ \cite[Теорема 4.2]{HK}, а затем  и  \cite[Theorem 6.32]{HII}, как в 
\cite[Доказательство Предложения 1.4, (1.13)--(1.14)]{Khab03}.
\end{proof}
\subsubsection{\underline{Продолжение тестовых функций в $D\setminus \{z_0\}$}}\label{spf:241}  {\it 

\begin{propos}\label{prop:V}  Для любых 
\begin{enumerate}[{\rm (i)}]
	\item\label{zbi} точки $z_0\in D_0$ 
	и числа  $b\in \RR^+$,
	\item\label{Diii} регулярной области
	 	$\widetilde{D}\subset \CC_{\infty}$ с\/ $D_0\Subset \widetilde{D}\subset D$  и $\CC_{\infty}\setminus \clos \widetilde{D}\neq \varnothing$,
\end{enumerate}
найдётся   число 
\begin{equation}\label{cz0}
\widetilde{c}:=\const_{z_0,D_0,\widetilde{D}, b}^+=\frac{1}{b} \inf_{z\in \partial D_0}  g_{\widetilde{D}}(z, z_0)  >0,
\end{equation}
для которого  по любой функции $v\overset{\eqref{sbh+}}{\in}  \sbh_0^+(D\setminus D_0;\leq b) $ с мерой Рисса $\mu_v$
при помощи продолженной функции Грина  $g_{\widetilde{D}}(\cdot, z_0)$ можно  построить  функцию 
\begin{equation}\label{cr:Vng}
\widetilde{V}:= 
\begin{cases}
g_{\widetilde{D}}(\cdot, z_0)\quad &\text{на $\clos D_0\setminus \{z_0\}$},\\
\max\bigl\{g_{\widetilde{D}}(\cdot, z_0) ,\; \widetilde{c}\cdot v\bigr\} \quad &\text{на $\widetilde{D}\setminus \clos D_0$},\\
\widetilde{c}\cdot v \quad &\text{на $D\setminus \widetilde{D}$},\\
0 \quad &\text{на $\CC_{\infty}\setminus {D}$},
\end{cases}
\end{equation}
обладающую следующими  свойствами:
\begin{subequations}\label{VVV}
\begin{align}
\hspace{-20mm}
\widetilde{V}\in \sbh^+\bigl(\CC_{\infty}\setminus \{z_0\}\bigr)\text{ с } &\text{мерой Рисса }
{\mu_{\widetilde{V}}}\in 
\mathcal M^+\bigl(\clos D\bigr), 
\tag{\ref{VVV}V}\label{cr:Vn+V}\\ 
\widetilde{V}\bigm|_{D_0\setminus \{z_0\}} &\in \Har \bigl(D_0\setminus \{z_0\}\bigr),
\tag{\ref{VVV}h}\label{VVVh}
\\
 \lim_{D\ni z\to \partial D} \widetilde{V}(z)&\overset{\eqref{est:u0}}{=}0,
\tag{\ref{VVV}o}\label{VVVho} 
\\
\lim_{z_0\neq z \to z_0}
\frac{\widetilde{V}(z )}{l_{z_0}(z)
}=  1\; &\text{    --- нормировка в точке $z_0$,}
\tag{\ref{VVV}n}\label{VVVo}
\\
\text{где}\quad l_{z_0}(z)&\overset{\eqref{{nvz}INF}}{:=}\begin{cases}
\log \frac1{|z-z_0|} \quad&\text{при $z_0\neq \infty$},\\
\log |z| \quad&\text{при $z_0= \infty$},
\end{cases}
\tag{\ref{VVV}l}\label{VVVl}\\
D_0\cap \supp {\mu_{\widetilde{V}}}=\varnothing, \quad {\mu_{\widetilde{V}}}\bigm|_{D\setminus \clos \widetilde{D}}&=\widetilde{c}\mu_v\bigm|_{D\setminus \clos \widetilde{D}}, 
\quad {\mu_{\widetilde{V}}} \bigl(\clos D\bigr)=1.
\tag{\ref{VVV}m}\label{cr:Vn+N}
\end{align}
\end{subequations}
\vskip 1mm
\end{propos}
}

\begin{proof}
Функция $v\overset{\eqref{sbh+}}{\in} \sbh_0^+(D\setminus D_0;\leq b)$ ввиду положительности и выполнения свойства  \eqref{est:u0}, т.\,е. 
$\lim\limits_{z\to \partial D}v(z)=0$, по Предложению \ref{exttf}  может быть  распространена  на открытое в $\CC_{\infty}$ дополнение $\CC_{\infty}\setminus \clos D_0$ значениями $0$ на $\CC_{\infty}\setminus D$ как субгармоническая. Таким образом, можем считать функцию $v$ определённой и субгармонической всюду на открытом множестве $\CC_{\infty}\setminus \clos D_0$. Кроме того,  в силу её полунепрерывности сверху  в  окрестности границы $\partial D_0\supset \partial (\CC_{\infty} \setminus \clos D_0)$, для неё имеет место  ограничение
\begin{equation}\label{ev0}
	\hspace{-1mm} 	\limsup_{(\CC_{\infty} \setminus \clos D_0) \ni z'\to z} v(z')\leq \sup_{z\in \partial D_0} v(z) \overset{\eqref{sbh+}}{\leq} b
		\quad \text{для всех $z\in \partial (\CC_{\infty} \setminus \clos D_0)$.}
\end{equation}

Для области $\widetilde{D}$ из условия \eqref{Diii} рассмотрим  продолженную, как в Примере \ref{exp2}, функцию Грина $g_{\widetilde{D}}(\cdot, z_0)$.
Из  свойств функции Грина \cite[Theorem 4.4.3]{Rans}
			\begin{equation}\label{dfa}
	a:=\inf_{z\in \partial D_0}  g_{\widetilde{D}}(z, z_0)>0.	
	\end{equation}
	Теперь перейдём к функции 		
\begin{subequations}\label{v0}
\begin{align}
	v_0&\overset{\eqref{dfa}}{:=}\frac{b}{a}\,  g_{\widetilde{D}}(\cdot , z_0)  \in \sbh^+\bigl(\CC_{\infty}\setminus\{z_0\}\bigr),
\tag{\ref{v0}v}\label{v0v}\\ 
\intertext{для которой по известным свойствам функций Грина 
\cite[Theorem 4.4.9]{Rans}}
 \lim_{D\ni z\to z_0}\frac{v_0(z)}{l_{z_0}(z)}
&\overset{\eqref{VVVl}}{=} \frac{b}{a}\,, \quad v_0 (z)\equiv 0 \quad\text{для всех $z\in D\setminus \widetilde{D}$}. 
\tag{\ref{v0}o}\label{v0o}
\\
v_0\bigm|_{\widetilde{D}\setminus \{z_0\}}&\in \Har \bigl(\widetilde{D}\setminus \{z_0\}\bigr).
\tag{\ref{v0}h}\label{v0h}
\end{align}
\end{subequations}
Кроме того, ввиду \eqref{dfa}, \eqref{v0v}, \eqref{v0o} и \eqref{ev0}
\begin{equation}\label{vvv}
	v_0\bigm|_{\partial (\CC_{\infty} \setminus \clos D_0)}\overset{\eqref{v0v}}{\geq} b
	\overset{\eqref{ev0}}{\geq} v\bigm|_{\partial (\CC_{\infty} \setminus \clos D_0)}, \quad v_0\bigm|_{D\setminus \widetilde{D}}=0\leq v\bigm|_{D\setminus \widetilde{D}} .
\end{equation}
Будет использована

\textsc{Теорема о склейке} \cite[Theorem 2.4.5]{Rans}. {\it Пусть $\mathcal O_0, \mathcal O_1$ --- открытые множества в $\CC$ и 
$\mathcal O_1 \subset \mathcal O_0$. Пусть $v_0\in \sbh (\mathcal O_0)$, $v\in \sbh (\mathcal O_1)$.  Если 
\begin{equation}\label{0vs}
	\limsup_{\mathcal O_1\ni  z'\to z} v(z')\leq v_0(z) \quad \text{для всех $z\in \mathcal O_0\cap \partial \mathcal O_1$},
\end{equation}
то функция
\begin{equation}\label{consv}
	\widetilde{v}:=\begin{cases}
\max\{v,v_0\} \quad &\text{на\/ $\mathcal O_1$},\\
v_0		\quad &\text{на\/ $\mathcal O_0\setminus \mathcal O_1$},
	\end{cases}
\end{equation}
субгармоническая на  $\mathcal O_0$.}
\vskip 1mm

 Применим Теорему о склейке при
\begin{equation*}
	\mathcal O_0:=\CC\setminus \{z_0\}, \quad \mathcal O_1:=\CC\setminus \clos D_0
\end{equation*}
к функциям $v_0$  из \eqref{v0v} и $v$, продолженной выше на  $\mathcal O_1$, удовлетворяющим, ввиду \eqref{vvv},  условию \eqref{0vs}. По построению
\eqref{consv} и ввиду \eqref{v0o} для построенной функции $\widetilde{v}\in \sbh^+\bigl(\CC\setminus \{z_0\}\bigr)$ её конструкцию можно описать через соответствующие сужения более детально:
\begin{subequations}\label{v000}
\begin{align}
	0\leq \widetilde{v}&=
	\begin{cases}
		v_0 \quad &\text{на $\clos D_0\setminus \{z_0\}$},\\ 
	\max\{v,v_0\}\quad &\text{на $\widetilde{D}\setminus \clos D_0$},\\
	v \quad &\text{на $D\setminus  \widetilde{D}$},\\
	0 \quad &\text{на $\CC_{\infty}\setminus  D$},
	\end{cases}
\tag{\ref{v000}v}\label{v000v}\\
\intertext{и при этом имеет место нормировка} 
	\limsup_{\widetilde{D}\ni w\to z_0}&\frac{\widetilde{v}(w)}{l_{z_0}(w)}\overset{\eqref{v0o}}{=} \frac{b}{a} \, ; \quad 
	\widetilde{v}\bigm|_{D_0\setminus \{z_0\}}\overset{\eqref{v0h}}{\in} \Har \bigl(D_0\setminus \{z_0\}\bigr).
\tag{\ref{v000}o}\label{v000o}
\end{align}
\end{subequations}
Таким образом, ввиду \eqref{v000o} для положительной субгармонической функции 
\begin{equation}\label{V000o}
\widetilde{V}\overset{\eqref{v000v}}{:=}\frac{a}{b}\, \widetilde{v}	\in \sbh^+\bigl(\CC_{\infty}\setminus \{z_0\}\bigr) 
\end{equation}
выполнено условие нормировки  \eqref{VVVo}. Положим 
\begin{equation*}
	\widetilde{c}:=\frac{a}{b}\overset{\eqref{dfa}}{=} \frac1b \inf_{z\in \partial D_0}  g_{\widetilde{D}}(z, z_0) 
\end{equation*}
в \eqref{cz0}. Тогда из    \eqref{v000v}, домноженного на  $\widetilde{c}$, получим в точности \eqref{cr:Vng}. Наконец, все перечисленные свойства \eqref{cr:Vn+V}--\eqref{VVVl}  ---  прямые следствия построения \eqref{v000}--\eqref{V000o} и известных свойств функции Грина $g_{\widetilde{D}}(\cdot , z_0)$, участвующей в нём, начиная с  \eqref{v0v}. 
Первое равенство в \eqref{cr:Vn+N} следует из  \eqref{VVVh} (гармоничность), второе --- из третьей  строчки в \eqref{cr:Vng} конструкции  функции $\widetilde{V}$. Мера ${\mu_{\widetilde{V}}}$ вероятностная по Предложению \ref{prprm} с $p=1$ ввиду  нормировки \eqref{VVVo}. 
\end{proof}
\begin{remark}\label{asPJ} Из перечня  свойств \eqref{VVV} функции $\widetilde{V}$ из \eqref{cr:Vng} по Определению \ref{df:PoJ} легко видеть, что эта функция --- потенциал Йенсена из $PJ_{z_0}(D')$ внутри любой подобласти $D'\ni z_0$ в $\subset \CC_{\infty}$ при  $D\Subset D'$. Соответственно по Предложению \ref{pr:1} из \eqref{eq:mu} следует, что мера Рисса ${\mu_{\widetilde{V}}}=\frac{1}{2\pi} \widetilde{V}$ --- мера Йенсена 
внутри таких же областей $D'\subset \CC_{\infty}$ в точке $z_0$. Более того, ${\mu_{\widetilde{V}}}$ --- мера Йенсена в точке $z_0$ внутри  любой области, содержащей  {\it оболочку\/} 
множества $\{z_0\}\cup \supp \mu_{\widetilde{V}}$, 
полученную  объединением  этого множества со всеми относительно компактными компонентами в $D'$ из дополнения $\{z_0\}\cup \supp \mu_{\widetilde{V}}$ до $D'$
\cite[Предложение 1.1]{Khab03}.
\end{remark}
\begin{remark}\label{mur} Если меру $\mu_v$ рассматривать в продолженном смысле как меру Рисса продолженной тестовой  функции 
$v$  вида \eqref{d:exvv} из Определения \ref{df:ctf}, то по построению $\widetilde{V}$ промежуточное равенство в \eqref{cr:Vn+N} уточняется до  исчерпывающего равенства
${\mu_{\widetilde{V}}}\bigm|_{\clos D\setminus \clos \widetilde{D}}=\widetilde{c}\,\mu_v\bigm|_{\clos D\setminus \clos \widetilde{D}}$\,.
\end{remark}

\begin{remark}\label{rem} Доказательство без изменений проходит и для тестовых  функций из Замечания \ref{rmmv}, т.\,е. функций 
$v\overset{\eqref{df:mvb}}{\in} \sbh_0^+(D\setminus \clos D_0; \leq b)$. При этом 
в \eqref{ev0} даже можно пропустить промежуточное выражение $ \sup_{z\in \partial D_0} v(z) $ в неравенствах  --- 
см. условие \eqref{0vs}  Теоремы о склейке.
\end{remark}
\begin{remark}\label{25}
Продолжение тестовой функции $v$ после  домножения на $\widetilde{c}$ в $D\setminus \{z_0\}$  означает, что каждая тестовая функция $v$ может  рассматриваться как тестовая функция из минимального  класса $\sbh^+_0\bigl(D\setminus \{z_0\}\bigr)$ {\large(}см. \eqref{mint}, Замечание \ref{1_3}{\large)} и продолжается на $\CC_{\infty}\setminus \{z_0\}$ с нулевыми значениями на $\CC_{\infty}\setminus D$.
\end{remark}

\subsubsection{\underline{Продолженная функция --- предел потенциалов Йенсена}}\label{sss242}
\begin{propos}\label{pr:Vn}
 {\it Пусть $\widetilde{V}$ --- функция из \eqref{cr:Vng}, построенная в\/
{\rm п.~\ref{spf:241}.} Тогда последовательность функций
\begin{equation}\label{Vninc+}
{V}_n:=\Bigl(\widetilde{V}-\frac1n\Bigr)^+:=\max\Bigl\{0, \widetilde{V}-\frac1n\Bigr\}\in 	\sbh^+\bigl(\CC_{\infty}\setminus\{z_0\}\bigr), \quad n\in \NN,
\end{equation}
представляет собой последовательность потенциалов  Йенсена $V_n\in PJ_{z_0}(D)$,   для которой поточечно 
\begin{subequations}\label{cr:Vn}
 \begin{align} 
\lim_{n\to+\infty} V_n=\widetilde{V}, \quad 
V_n&\leq V_{n+1} \quad \text{на $D\setminus \{z_0\}$ при всех  $n\in \NN$},
\tag{\ref{cr:Vn}l}\label{cr:Vnm}
\\
\intertext{для некоторого числа $r_0>0$ с вложением $D(z_0,r_0)\Subset D_0$ имеем}
V_n\bigm|_{D(z_0,r_0)} &\in \Har \bigl(D(z_0,r_0)\bigr)\quad \text{при больших $n\in \NN$},
\tag{\ref{cr:Vn}h}\label{cr:Vnh}
\\
\lim_{(D\setminus \{z_0\})\ni z \to z_0}
&\frac{V_n(z )}{l_{z_0}(z)}= 1  \quad  \text{  --- нормировка в $z_0$ при всех $n\in \NN$}.
\tag{\ref{cr:Vn}o}\label{cr:Vno}
\end{align}
\end{subequations}}
\end{propos}

\begin{proof}
При любом $n\in \NN$ для каждой из функций 
\begin{equation}\label{Vninc}
{V}_n:=\Bigl(\widetilde{V}-\frac1n\Bigr)^+:=\max\Bigl\{0, \widetilde{V}-\frac1n\Bigr\}\in 	\sbh^+\bigl(\CC_{\infty}\setminus\{z_0\}\bigr),
\end{equation}
её субгармоничность и  положительность в $\CC_{\infty}\setminus \{z_0\}$ --- следствие из  \eqref{cr:Vn+V}, 
нормировка   \eqref{cr:Vno} --- из \eqref{VVVho},  гармоничность \eqref{cr:Vnh} --- \eqref{VVVh}, а требуемая в Определении \ref{df:PoJ} потенциалов Йенсена финитность в $D$ вытекает из построения \eqref{Vninc} и свойства \eqref{VVVo}. Таким образом,  ${V}_n \in PJ_{z_0}(D)$ --- потенциалы Йенсена внутри $D$ с полюсом $z_0\in D$. Согласно построению \eqref{Vninc}
последовательность $(V_n)_{n\in \NN}$ возрастающая и стремится поточечно к $\widetilde{V}$ в смысле \eqref{cr:Vnm}.
\end{proof}

\subsubsection{\underline{Одна гипотеза о последовательности выметаний мер}}\label{baly}
Остановимся на  природе {\it вероятностных мер Йенсена\/} $\mu_n$ потенциалов  Йенсена $V_n$ с носителем $\supp \mu_n\subset D$  из 
\eqref{Vninc+} и законе их построения по вероятностной мере Рисса  $\mu_{\widetilde{V}} \in \mathcal M^+(\clos D)\subset \mathcal M^+(\CC_{\infty})$.
Удобно считать, что $z_0=\infty$, поскольку к этому случаю всегда можно перейти с помощью инверсии со сдвигом  вида \eqref{inv_z0}.

Каждому $n\in \NN$ сопоставим открытое, в силу полунепрерывности сверху функции $\widetilde{V}$,  множество 
$\mathcal O_n :=\bigl\{z\in \CC
 \colon  \widetilde{V}(z) <  1/n \bigr\} $
и рассмотрим его дополнение $K_n:=\CC\setminus \mathcal O_n$  --- компакт в $\CC$ ввиду условия \eqref{cr:Vno} логарифмической нормировки.
Дословно повторив  классическую процедуру выметания меры $\mu_{\widetilde{V}}$ на компакт $K_n$, 
основанную на некоторой  аппроксимационной технике, детально описанной в \cite[гл.~IV, \S\S~1,2,4]{L}, можно убедиться, что {\it меры Рисса\/
$\mu_n\in J_{\infty}(D)\subset \mathcal M^+_{\comp}(D)$ субгармонических функций $V_n$--- это  выметания   вероятностной меры\/ $\mu_{\widetilde{V}} \in \mathcal M^+(\clos D)\subset \mathcal M^+(\CC)$ из\/ $\mathcal O_n $  на\/  $K_n $, а при $\NN\ni n\to +\infty$  последовательность вероятностных  мер $\mu_n\in \mathcal M^+(\CC)$\/  $*$-слабо сходится на $\CC$ к вероятностной мере $\mu_{\widetilde{V}}\in 
\mathcal M^+(\CC)$, т.\,е. для любой финитной\/ {\rm на  $\CC$} непрерывной  функции $f\colon \CC\to \RR$ существует предел $\lim\limits_{n\to +\infty}\int f \, d \mu_n=\int f \, d \mu_{\widetilde{V}}$.}  

По обобщенной формуле  Пуассона--Йенсена, прописанной ниже в виде \eqref{Pju}  и возрастания функций $V_n$ следует, что {\it для любой\/}  функции $u\in \sbh (D)$ имеет место цепочка неравенств 
\begin{equation*}
	\dots \leq \int_D u\,d\mu_n\leq \int_D u \,d \mu_{n+1}\leq \dots .
\end{equation*}
Это означает,  что каждая следующая вероятностная мера $\mu_{n+1}$ является выметанием предшествующей вероятностной меры $\mu_n$ относительно конуса $\sbh (D)$  \cite{Khab01} ---  в записи 
 \begin{equation}\label{precmu}
	\dots \prec \mu_n\prec \mu_{n+1}\prec \cdots\,
 \end{equation}
	\begin{Conj}\label{g:1}  Пусть задана последовательность выметаний \eqref{precmu} относительно $\sbh (D)$ вероятностных мер Йенсена  
	$\mu_n\in  J_{z_0}(D)$,  $*$-слабо сходящаяся на $\CC_{\infty}$ к вероятностной мере $\mu \in J_{z_0}(\CC_{\infty})$, сосредоточенной 
внутри  области $D$, т.\,е. $\mu(\CC_{\infty}\setminus D)=0$. Тогда для  любой интегрируемой по мере $\mu$ функции $u\in \sbh(D)$ имеем
\begin{equation*}
 	\lim_{n\to +\infty} \int_D u\,d\mu_n =\int_{D} u\, d\mu \quad \text{ или \;  $\leq \int_{D} u\, d\mu$}. 
\end{equation*}
 
	\end{Conj}

Для некоторых  специальных областей и/или функций $u$ удается подтвердить эту гипотезу, но это ещё весьма далеко от общей ситуации. Например, Гипотеза \ref{g:1} со знаком равенства верна, если функцию $u$ можно представить в виде {поточечно возрастающей к $u$ в $D$ последовательности} функций $u_k \in \sbh (D_k)$, $k\in \NN$,  где последовательность областей  $\dots \Supset D_{k+1}\Supset D_{k}\Supset \cdots$ в пересечении даёт область $D$. Обнадёживающим фактором в пользу справедливости Гипотезы \ref{g:1} может служить подобный известный  факт 
для потенциалов в пространствах выметаний \cite[VI, 10.2--5]{BH}. 
Справедливость Гипотезы  \ref{g:1} дало бы дальнейшее развитие обобщенной формулы Пуассона--Йенсена из Предложения \ref{pr:2} вплоть до применения к тестовым функциям, продолженным на $\CC_{\infty}\setminus \{z_0\}$ в рамках Замечания \ref{25} (см. и  первое из заключительных замечаний в п.~\ref{persp} ниже). 
 
\subsection{\sc Доказательство Основной Теоремы}\label{ss:+} Будет использована построенная в  п.~\ref{sss242}
 возрастающая к функции $\widetilde{V}$ из \eqref{cr:Vng} последовательность потенциалов Йенсена $(V_n)_{n\in \NN}$ со свойствами \eqref{cr:Vn}. Каждому потенциалу Йенсена $V_n$ по Предложению \ref{pr:1} с отображением $\mathcal P$ из \eqref{con:P} соответствует 
мера Йенсена $\mu_n$ внутри области  $D$ в точке $z_0\in D\setminus {\dom M}$: 
\begin{equation}\label{cr:Vnh00}
\mu_n:= \mathcal P^{-1} (V_n)\overset{\eqref{eq:mu}, \eqref{cr:Vno}}{=}\frac1{2\pi} \Delta V_n\in PJ_{z_0}(D), \quad
\supp \mu_n\overset{\eqref{cr:Vnh}}{\subset} D\setminus D_0.  
\end{equation}
Пусть $M=u_1-u_2$, где $u_1,u_2\in \sbh (D)$ с мерами Рисса  $\nu_M^+, \nu_M^-\in \mathcal M^+(D)$ соответственно, причём по определению $\delta$-субгармонической функции ввиду  $z_0\in {\dom M}$ имеем $u_1(z_0)\neq -\infty$ и $u_2(z_0)\neq -\infty$. 
При этом,  если $u(z_0)=-\infty$ или \underline{не} выполнены условия 
	\begin{equation}\label{Mv}
				\int_{D\setminus D_0}v \, d \nu_M^-\overset{\eqref{intvnuM}}{<}+\infty,  \quad 
					\int_{\widetilde{D}\setminus D_0}g_{\widetilde{D}} (\cdot, z_0)\, d \nu_M^-
			{<}+\infty,  	
				\end{equation}
то неравенство  \eqref{mest} тривиально. Поэтому можем рассмотреть только случай, когда  $u(z_0)\neq -\infty$ и одновременно выполнено \eqref{Mv}. 

По обобщённой формуле Пуассона--Йенсена из Предложения \ref{pr:2}, применённой к субгармоническим функциям $u, u_1,u_2$,   получаем
\begin{subequations}\label{PJ}
\begin{align}
	u(z_0) +\int_{D\setminus \{z_0\}} V_{n} \,d {\nu}_u &\overset{\eqref{f:PJ}, \eqref{cr:Vnh00}}{=}\int_{D\setminus D_0} u \,d \mu_n  ,
\tag{\ref{PJ}$u$}\label{Pju}\\ 
u_1(z_0) +\int_{D\setminus \{z_0\}} V_{n} \,d {\nu}_M^+ &\overset{\eqref{f:PJ},\eqref{cr:Vnh00}}{=}\int_{D\setminus D_0} u_1 \,d \mu_n  ,
\tag{\ref{PJ}$u_1$}\label{Pju1} \\
u_2(z_0) +\int_{D\setminus \{z_0\}} V_{n} \,d {\nu}_M^- &\overset{\eqref{f:PJ}, \eqref{cr:Vnh00}}{=}\int_{D\setminus D_0} u_2 \,d \mu_n  .
\tag{\ref{PJ}$u_2$}\label{Pju2}
\end{align}
\end{subequations}
Из условия $u\leq M=u_1-u_2$ на $D\setminus D_0$ для правых частей равенств \eqref{PJ} получаем
\begin{equation*}
	\int_{D\setminus D_0} u \,d \mu_n \leq \int_{D\setminus D_0} M \, d \mu_n=
	\int_{D\setminus D_0} u_1 \,d \mu_n- \int_{D\setminus D_0} u_2 \,d \mu_n,
\end{equation*}
 откуда по обобщённым формулам Пуассона--Йенсена \eqref{PJ}
\begin{equation}\label{uu12}
	u(z_0) +\int_{D\setminus \{z_0\}} V_{n} \,d {\nu}_u 	+\int_{D\setminus \{z_0\}} V_{n} \,d {\nu}_M^-
	\leq M(z_0) +\int_{D\setminus \{z_0\}} V_{n} \,d {\nu}_M^+.
\end{equation}
Напомним, что последовательность функций $V_n$, возрастает и стремится к $\widetilde{V}$ на $D$ поточечно. Тогда, очевидно, из \eqref{uu12} следует
\begin{equation}\label{uu122}
	u(z_0) +\int_{D\setminus \{z_0\}} V_{n} \,d {\nu}_u 	+\int_{D\setminus \{z_0\}} V_{n} \,d {\nu}_M^-
	\leq M(z_0) +\int_{D\setminus \{z_0\}} \widetilde{V} \,d {\nu}_M^+,
\end{equation}
где для правой части допускается и значение $+\infty$. Если этот интеграл действительно равен $+\infty$, то ввиду конечности 
интеграла \eqref{dMa} для $z=z_0\in {\dom M}$ 
\begin{equation*}
	+\infty=\int_{D\setminus \widetilde{D}} \widetilde{V} \,d {\nu}_M^+\overset{\eqref{cr:Vng}}{=}
	\int_{D\setminus \widetilde{D}} \widetilde{c}\cdot  v \,d {\nu}_M^+,
\end{equation*}
откуда, ввиду конечности первого интеграла в \eqref{Mv}, следует, что первый  интеграл в правой части   \eqref{{mest}v} также равен $+\infty$ и доказывать нечего. Поэтому далее предполагаем, что интеграл в правой части  \eqref{uu122} конечен. Применяя Теорему 
 Бепп\'о Л\'еви о монотонной сходимости интегралов  к левой части \eqref{uu122}, получаем 
\begin{equation}\label{itots}
	u(z_0) +\int_{D\setminus \{z_0\}} \widetilde{V} \,d {\nu}_u 	+\int_{D\setminus \{z_0\}} \widetilde{V}\,d {\nu}_M^-
	\leq M(z_0) +\int_{D\setminus \{z_0\}} \widetilde{V} \,d {\nu}_M^+.
\end{equation}
Здесь  в силу \eqref{cr:Vng}
\begin{equation*}
		\widetilde{V}= 
	\begin{cases}
	g_{\widetilde{D}}(\cdot, z_0)\quad &\text{на $D_0$},\\
	\max\bigl\{g_{\widetilde{D}}(\cdot, z_0),\;\widetilde{c}\cdot v \bigr\}\quad &\text{на $\widetilde{D}\setminus D_0$},\\
	\widetilde{c}\cdot v \quad &\text{на $D\setminus \widetilde{D}$}.
	\end{cases}
	\end{equation*}
		Из этих равенств применительно к \eqref{itots}, учитывая \eqref{intvnuM}, получаем
		\begin{multline*}
		u(z_0) +\int_{D_0} g_{\widetilde{D}}(\cdot, z_0)  \,d {\nu}_u  
		+\int_{D\setminus D_0} \widetilde{c}\cdot v \,d {\nu}_u 	\\		\leq 
		M(z_0) +		\int_{D_0} g_{\widetilde{D}}(\cdot, z_0)  \,d {\nu}_M  
		+\int_{\widetilde{D}\setminus D_0} \max\bigl\{g_{\widetilde{D}}(\cdot, z_0),\;\widetilde{c}\cdot v \bigr\}  \,d {\nu}_M  
		+\int_{D\setminus \widetilde{D}} \widetilde{c}\cdot v \,d {\nu}_M,
				\end{multline*}
	откуда, после отбрасывания  второго положительного слагаемого-интеграла в левой части и  деления  обеих частей на постоянную $\widetilde{c}$ из  \eqref{cz0}, с постоянной 
		\begin{equation*}
		C :=\frac{1}{\widetilde{c}}\overset{\eqref{cz0}}{=}
		\frac{b} {\inf_{z\in \partial D_0}  g_{\widetilde{D}}(z, z_0)}  >0
\end{equation*}
		вида  \eqref{cz0C} следует 
\begin{multline}\label{ittCMg}
C u(z_0) 	+\int_{D\setminus D_0}  v \,d {\nu}_u 	\leq 	\int_{D\setminus \widetilde{D}}  v \,d {\nu}_M	
\\	+		C\int_{D_0} g_{\widetilde{D}}(\cdot, z_0)  \,d {\nu}_M  
		+\int_{\widetilde{D}\setminus D_0} \max\bigl\{Cg_{\widetilde{D}}(\cdot, z_0),\; v \bigr\}  \,d {\nu}_M^+  +CM(z_0).
\end{multline}	
	В силу очевидного для положительных функций  неравенства
\begin{equation}\label{Cginv}
	\max\bigl\{Cg_{\widetilde{D}}(\cdot, z_0),\; v \bigr\}\leq Cg_{\widetilde{D}}(\cdot, z_0)+ v \quad \text{на $D\setminus D_0$}
\end{equation}
 из \eqref{ittCMg} получаем 
\begin{multline*} 
C u(z_0) 	+\int_{D\setminus D_0}  v \,d {\nu}_u 	\leq 	\int_{D\setminus D_0}  v \,d {\nu}_M	+\int_{\widetilde{D}\setminus D_0}
v \,d {\nu}_M^-
\\	+		C\int_{\widetilde{D}} g_{\widetilde{D}}(\cdot, z_0)  \,d {\nu}_M  
		+C\int_{\widetilde{D}\setminus D_0} g_{\widetilde{D}}(\cdot, z_0)  \,d {\nu}_M^-  +CM(z_0),
		 \end{multline*}	
		что и доказывает требуемое \eqref{mest}. 	Наконец, если $\widetilde{D}\Subset D$, то как второй интеграл в правой части \eqref{{mest}v} ввиду ограниченности $v$ на $\widetilde{D}\setminus D_0$, так и интегралы из   \eqref{{mest}C} ввиду \eqref{dMa} для $z_0\in \dom M$, участвующие в определении величины $\overline{C}_M$, конечны. Свойства положительной однородности и полуаддитивности для  $\overline{C}_M$ по $M$ --- следствие явного представления её в \eqref{{mest}C} как суммы интегралов по заряду
		$\nu_M$, мере  $\nu_M^-$ 				и значения $M(z_0)$, зависящим от $M$ подобным же образом.  
	\begin{remark}\label{remw+} Доказательство Основной Теоремы проходит без каких-либо существенных усложнений 
		и для тестовых  функций из Замечания \ref{rmmv}, т.\,е. с функциями  $v\overset{\eqref{df:mvb}}{\in} \sbh_0^+(D\setminus \clos D_0; \leq b)$ в п.~ (iii).
	\end{remark}
	
	\subsection{Субгармоническая мажоранта $M$} Приведем  здесь 
 \begin{proof}[ Теоремы \ref{th:1}]  Если $u(z_0)=-\infty$, то доказывать нечего. То же самое, если интегралы в правых частях 
\eqref{mest+}--\eqref{in:fz0} равны $+\infty$. Поэтому предполагаем, что $u(z_0)>-\infty$ и конечность этих интегралов.  Для частного случая  $\boldsymbol{-\infty}\neq M \in \sbh (D)$   Основной Теоремы, очевидно, 
$\nu_M^-=0$  и в правой части неравенства \eqref{mest} из \eqref{Mv} интегралы по мере $\nu_M^-$  равны нулю. По условию также $M(z_0)\neq -\infty$. Всегда можно подобрать  
регулярную  область $\widetilde{D}\Subset D$, к примеру, ограниченную  
конечным числом гладких жордановых дуг \cite[гл.~V, \S~4]{St}, для которой $D_0\Subset \widetilde{D}$. Таким образом, 
для любой тестовой  функции $v\overset{\eqref{sbh+}}{\in} \sbh_0^+(D\setminus D_0;\leq b)$ из неравенства \eqref{mest} получаем
\begin{equation}\label{ittCMg+}
\int_{D\setminus D_0}  v \,d {\nu}_u 	\leq 	\int_{D\setminus D_0}  v \,d {\nu}_M	 - C u(z_0) 	
+		C\int_{\widetilde{D}} g_{\widetilde{D}}(\cdot, z_0)  \,d {\nu}_M  		 +CM(z_0).
		 \end{equation}	
При этом выбор области $\widetilde{D}$ полностью обусловлен лишь взаимным расположением областей $D_0$ и $D$, т.\,е. в выборе постоянной $C$ в \eqref{cz0C} влияние области $\widetilde{D}$ можно заменить на зависимость от областей  $D_0$ и $D$, как для постоянной $C$ из \eqref{cz0C+}.
Возможность выбора постоянной  
\begin{equation}\label{CMg0}
	\overline{C}_M:= \int_{\widetilde{D}} g_{\widetilde{D}}(\cdot, z_0)  \,d {\nu}_M   +M(z_0)
\end{equation}
 как в  \eqref{cz0C+}  положительно однородной и полуаддитивной сверху по $M$ следует из  её явного вида \eqref{CMg0}. 
 Это и \eqref{ittCMg+} даёт \eqref{mest+} с постоянными $C, \overline{C}_M$ из \eqref{cz0C+}.
Из \eqref{mest+} для функции $u=\log |f|$ с функцией $f\in \Hol (D)\setminus \{0\}$, удовлетворяющей неравенству $\log |f|\leq M$ на $D\setminus D_0$ согласно \eqref{df:nZS}--\eqref{nufZ} ввиду $n_{\tt Z}\leq n_{\Zero_f}$ при $f({\tt Z})=0$ следует  \eqref{in:fz0}.  
\end{proof}

 \begin{proof}[ Следствия \ref{cor:1}] Всегда можно выбрать точку $z_0\in D_0$ так, что $u(z_0)\neq -\infty$. Из \eqref{est:u0} и  условия $v\in \sbh^+ (D\setminus D_0)$, по принципу максимума 
$b:=\sup_{z\in \partial D_0} v(z)<+\infty$ и, следовательно, $v\in \sbh_0^+(D\setminus D_0;\leq b)$. Тогда по Теореме \ref{th:1}
из конечности  левой части импликации \eqref{vnuum} для $u\leq M$ на $D\setminus D_0$ из  неравенства \eqref{mest+} получаем требуемую конечность правой части \eqref{vnuum}. Соответственно,
для $\log |f|\leq M$ на $D\setminus D_0$ из  \eqref{in:fz0} следует \eqref{vnuumf}.
\end{proof}

\section{Обратные теоремы}\label{ConverRes}

Результаты раздела дают некоторые формы  обращения Основной Теоремы и Теоремы \ref{th:1} на определённых  специальных подклассах
тестовых функций. 

\subsection{Обращение с финитными тестовыми функциями} 
Используем
\begin{definition}\label{df:fi} 
Тестовую фу\-н\-к\-ц\-ию $v \overset{\eqref{sbh+}}{\in}  \text{\rm sbh}_0^+(D\setminus D_0)$ из Определения  \ref{testv} называем {\it финитной,\/}  если для неё  справедливо усиление условия \eqref{resv0}: {\it существует подобласть $D_v\Subset D$, для  которой  $v\bigm|_{D\setminus D_v}=0$,}  т.\,е. вне $D_v\supset D_0$ функция $v$ тождественно равна нулю.
\end{definition}

\begin{theorem}\label{th:inver}  
Пусть $D$ --- область с неполярной границей,    $M\in \dsbh (D)$ --- нетривиальная функция с зарядом Рисса $\nu_M$,  $\nu\geq 0$ --- борелевская мера на  $D$,  $b>0$. 
Если существует постоянная $C\in \RR$, с которой   для всех финитных тестовых  функций $v{\in}  \text{\rm sbh}_0^+(D\setminus D_0;\leq b)$ выполнено неравенство  
\begin{equation}\label{estv+}
\int\limits_{D\setminus D_0}  v \,d {\nu} 		\leq	\int\limits_{D\setminus D_0}  v \,d {\nu}_M	+C,
\end{equation}
то  для каждой непрерывной функции\/  $r \colon D\to \RR$, удовлетворяющей условиям\/ 
\begin{subequations}\label{rd:}
\begin{align}
	0<r(z)< 	&\dist (z, \partial D) \quad\text{для всех $z\in D$,} 
	\tag{\ref{rd:}d}\label{rdD}\\
\Bigl\{z\in D \colon D\bigl(z,r(z)\bigr)\cap D^0\neq \varnothing \Bigr\}&\Subset D \quad\text{для любой подобласти $D^0\Subset D$}, 
\tag{\ref{rd:}c}\label{rdD1ad}
 \end{align}
\end{subequations}
 найдётся функция  $u\in \text{\rm sbh}(D)\setminus \{\boldsymbol{-\infty}\}$ с мерой Рисса $\nu_u\geq \nu$, для которой
\begin{equation}\label{in:hM}
u(z)\leq  M^{*r}(z):=\frac1{\lambda\bigl(D(r(z))\bigr)} \int\limits_{D(z, r(z))} M \,d\lambda \quad\text{для всех  $ z\in D$},
\end{equation}
где  $M^{*r}(z)$ ---  переменное усреднение функции  $M$ по кругам $D\bigl(z,r(z)\bigr)$. 

В частности, если функция  $M$ ещё и непрерывна, то можно подобрать  функцию 
  $u\in \text{\rm sbh}(D)\setminus \{\boldsymbol{-\infty}\}$ с мерой Рисса $\nu_u\geq \nu$   так, что   $u\leq M$  на $D$.
\end{theorem}
\begin{remark}\label{rrr} Условие  \eqref{rdD} на функцию $r$ естественно, поскольку только оно гарантирует существование 
усреднений $M^{*r}$ всюду на $D$, а ограничение  \eqref{rdD1ad}  довольно слабое. Например,  \eqref{rdD1ad} выполнено, если 
$\lim\limits_{D\ni z'\to z} r(z')=0$ для любой точки $z\in \partial D$, поскольку в этом случае  выполнено условие 
 $\lim\limits_{D\ni z\to \partial D} r(z)=0$ в смысле \eqref{0ev}--\eqref{est:u0}. Именно   из  ограничения \eqref{rdD1ad}  на $r$ легко следует, что для любой подобласти $D^0\Subset D$ найдётся подобласть $D_1 \Subset D$ cо свойствами $D^0\Subset D_1$ и $D\bigl(z, r(z)\bigr)\cap D^0=\varnothing$ для всех $z\in D\setminus D_1$. 
\end{remark}
\begin{proof} Для  меры $\nu$ на $D\supset D_0$ всегда можно  подобрать точку $z_0\in D_0\cap \dom M$, в которой для некоторого числа $r_0>0$ имеем соотношения
 \begin{equation}\label{dMa+}
		\left(\int_0^{r_0}\frac{\nu (z_0,t)}{t} \, dt<+\infty\right)\; \overset{\eqref{dMa}}{\Longleftrightarrow}	\;
		\left(\int_{D(z,r_0)} \log |z'-z_0| \, d  \nu (z')>-\infty\right).
		\end{equation}
Это обеспечивает существование   субгармонической функции $u_0\neq \boldsymbol{-\infty}$ с мерой Рисса в точности $\nu$
 и свойством $u_0(z_0)\neq -\infty$ --- см. п.~2 в подразделе \ref{2_3}. 

Далее нам временно потребуется ограниченность функции $M$ в окрестности точки $z_0$. Для этого пока преобразуем  её локально с сохранением условия 
\eqref{estv+}.  Выберем  число $r_0>0$ столь малым, что $D(z_0,3r_0)\Subset D_0$. Из представления  $M=u_+-u_-$ в  виде разности субгармонических функций $u_+, u_- \in \sbh(D)$ можно локально изменить значения функции $M$ в  $D(z_0,2r_0)\Subset D_0$:
		гармонически  продолжить интегралом Пуассона    функции $u_+$ и $u_-$ внутрь  круга $D(z_0,2r_0)$ --- обозначаем их
		как соответственно $u_+^o$ и $u_-^o$. Тогда функция 		$M^o:=u_+^o-u_- ^o\in \dsbh (D)$ уже ограничена в окрестности замкнутого круга 
		$\overline{D}(z_0,r_0)$, а \eqref{estv+} по-прежнему выполнено для   всех финитных  функций $v{\in}  \text{\rm sbh}_0^+(D\setminus D_0;\leq b)$.
Пока будем обозначать   функцию $M^o$ прежним символом $M$.

 Через $L_{\loc}^1(S)$ обозначаем множество всех локально интегрируемых по мере Лебега $\lambda$ функций на $S$ со значениями в $[-\infty, +\infty]$.  
Будет использована
\vskip 1mm
	{\sc Теорема C} {\large(}частный случай \cite[Теорема 6]{Kh07}{\large)}.
{\it Пусть\/ $M\in L^1_{\loc} (D)$, точка   $z_0\in D$,   $u_0 \in \sbh(D)$    с $u_0(z_0)\neq -\infty$.
Если  функция $M$ ограничена в открытой окрестности замыкания\/ 
$\clos D_1$ какой-нибудь подобласти\/ $D_1\Subset D$, содержащей точку $z_0$ и  существует постоянная\/ $C_0\in \RR$, с которой  
\begin{equation}\label{in:arsV}
	\int_{D} u_0 \,d \mu \leq \int_{D} M \,d \mu +C_0 
\quad\text{для любой  меры\/ $ \mu \in  J_{z_0}(D)$,}	
\end{equation}
 то  для каждой непрерывной функции\/  $r \colon D\to \RR$, удовлетворяющей условию\/ \eqref{rdD},
	 найдется субгармоническая на  $D$ функция\/ $w\neq \boldsymbol{-\infty}$, для которой
\begin{equation}\label{in:hMsh}
u_0+w\leq  M^{*r} \quad\text{на  $ D$}.
\end{equation}
}
В нашем случае роль области $D_1$ будет играть круг $D(z_0,r_0)$.

Кроме того, потребуется 
\begin{lemma}\label{pr:PJP} Пусть функция $M\in \dsbh(D)$ нетривиальная с зарядом Рисса $\nu_M$,  $z_0\in\dom M$,  функция $u_0\in \sbh (D)$ с мерой Рисса $\nu$,
 $u_0(z_0)\neq -\infty$,  а $V\in PJ_{z_0}(D)$ --- потенциал Йенсена  и $C_1\in \RR$. Если 
\begin{equation}\label{inmuM+}
	\int_{D\setminus \{z_0\}} V\, d \nu \leq \int_{D\setminus \{z_0\}} V \, d \nu_M +C_1, 
\end{equation}
то для меры Йенсена  $\mu\overset{\eqref{eq:mu}}{=}\mathcal P^{-1} (V)\in J_{z_0}(D)$ справедливо неравенство
\begin{equation}\label{inmuM+q}
	\int u_0\, d \mu \leq \int M \, d \mu +C_0, \quad \text{где $C_0=C_1-M(z_0)+u(z_0)$}.
\end{equation}
\end{lemma}
Действительно, при условии $z_0\in \dom M$ функция $M$ представима в виде разности $M=u_+-u_-$ субгармонических функций $u_+, u_- \in \sbh(D)$ 
 с мерами Рисса соответственно $\nu_M^+,\nu_M^- \in \mathcal M^+(D)$, для которых ввиду \eqref{dMa}  имеем 
$u_+(z_0)\neq -\infty$ и $u_-(z_0)\neq -\infty$. К каждой из функций $u_+, u_- $ применима обобщённая формула Пуассона--Йенсена 
Предложения \ref{pr:2}, следовательно, она применима и к функции $M$. Тогда  для  меры Йенсена  $\mu\overset{\eqref{eq:mu}}{:=}\mathcal P^{-1}V$ 
имеем
 \begin{multline*}
	\int_{D} u_0 \,d \mu   \overset{\eqref{f:PJ}}{=}\int_{D\setminus \{z_0\}} V\,d {\nu} +u_0(z_0) \\
	\overset{\eqref{inmuM+}}{\leq} \int_{D\setminus \{z_0\}} V \, d \nu_M +C_1+u_0(z_0)
	\overset{\eqref{f:PJ}}{=} \int M \, d \mu -M(z_0)+C_1 +u_0(z_0),
\end{multline*}
что и требуется в \eqref{inmuM+q}.

Вернёмся непосредственно к доказательству Теоремы \ref{th:inver}.

Для области $D$ с неполярной границей  $\partial D\subset \CC_{\infty}$ всегда существует функция Грина $g_{D}(\cdot ,z_0)$ 
{\; \it  с полюсом в точке $z_0$.} Далее всюду  в нашем доказательстве 
\begin{equation*}
g:=g_D(\cdot, z_0) \quad \text{\it --- функция Грина для $D$ с полюсом $z_0\in D_0$.}
\end{equation*}
Здесь для нас важны только  следующие её свойства {\large(}см. \cite[4.4]{Rans}, \cite[3.7, 5.7]{HK}{\large)}:
\begin{enumerate}[{(g1)}]
	\item\label{g1} $\lim\limits_{z_0\neq z\to z_0} \dfrac{g(z )}{l_{z_0}(z)}\overset{\eqref{{nvz}INF}}{=}  1$ --- нормировка в точке $z_0$ \eqref{nvzn};
	\item\label{g2} $g\bigm|_{D\setminus \{z_0\}}\in \Har^+\bigl(D\setminus \{z_0\}\bigr)$ --- гармоничность и положительность в $D\setminus \{z_0\}$. 
	\end{enumerate}
В частности, из принципа максимума-минимума,  ввиду $z_0\in D_0\Subset D$, 
\begin{equation}\label{B0}
	0<\const_{z_0,D_0,D}:=B_0:=\sup_{z\in \partial D_0} g(z)<+\infty.
\end{equation}
Пусть $V\in PJ_{z_0}(D)$ --- {\it произвольный потенциал Йенсена.\/} Тогда ввиду (g\ref{g1})--(g\ref{g2}) 
  	\begin{equation*}
	\limsup_{D\ni z\to z_0} \frac{(V-g)(z)}{l_{z_0}(z)} \overset{\rm(g\ref{g1})}{\leq} 0,  \quad 
	(V-g)\bigm|_{D\setminus \{z_0\}} \overset{\rm(g\ref{g2})}{\in} {\sbh} \bigl(D\setminus \{z_0\}\bigr).	
	\end{equation*}
	Отсюда точка $z_0$ --- устранимая особенность для функции $V-g$ и, поскольку
		\begin{equation*}
	\limsup_{D\ni z'\to z} (V-g)(z')\leq \limsup_{D\ni z'\to z} V(z')  = 0 \quad	\text{для всех $z\in \partial D$,}
	\end{equation*}
		для функции $V-g\in \sbh(D)$ по принципу максимума $V-g\leq 0$ на $D$. Отсюда
				\begin{equation}\label{es:g}
		V\leq g \quad \text{на $D$}, \qquad 
			V\overset{\eqref{B0}}{\leq} B_0 \quad \text{на $\partial D_0$}. 	
		\end{equation}
		Следовательно, для рассматриваемой в открытой окрестности $D\setminus D_0$  функции 
		\begin{equation*}
		v:=\frac{b}{B_0} \, V	\in \sbh_0^+\bigl(D\setminus D_0;\leq b\bigr)
		\end{equation*}
		справедливо неравенство \eqref{estv+}. Умножая обе его части на $B_0/b$,  получаем 
		\begin{equation*}
		\int_{D\setminus D_0} V \, d\nu  	\leq \int_{D\setminus D_0} V \, d \nu_M +\frac{B_0}{b} \,C 
				\quad\text{\it для всех $V\in PJ_{z_0}(D)$.}
		\end{equation*}
		Это неравенство можно переписать в виде
				\begin{multline}\label{bQ}
			\int_{D\setminus \{z_0\}} V \, d\nu  \leq \int_{D\setminus \{z_0\}} V \, d \nu_M \\
						+\frac{B_0}{b} \,C  + \int_{D_0\setminus \{z_0\}} V \, d\nu  +\int_{D_0\setminus \{z_0\}} V \, d \nu^-_M 
												\overset{\eqref{es:g}}{\leq}  \int_{D\setminus \{z_0\}} V \, d \nu_M \\
						+\left(\frac{B_0}{b} \,C  + \int_{D_0\setminus \{z_0\}} g \, d\nu  +\int_{D_0\setminus \{z_0\}} g \, d \nu^-_M\right) 
					\quad\text{\it для всех $V\in PJ_{z_0}(D)$.}
\end{multline}
			Два последних интеграла здесь конечны ввиду \eqref{dMa+} и 	$z_0\in D_0\cap \dom M$, обеспечивающем выполнение \eqref{dMa}.  Кроме того, они не зависят от $V\in PJ_{z_0}(D)$. Таким образом, с постоянной $C_1$, равной  значению <<большой>> скобки в конце \eqref{bQ},  выполнено  
		\eqref{inmuM+}	 для любого потенциала $V\in PJ_{z_0}(D)$. Отсюда по Лемме  \ref{pr:PJP} имеет место \eqref{inmuM+q} для любой меры Йенсена 
		$\mu\in J_{z_0}(D)$. Следовательно, выполнено  условие   \eqref{in:arsV} Теоремы C и 		найдётся  функция  $w\in \sbh (D)$, для которой имеем
		\eqref{in:hMsh}. При этом c мерой Рисса $\nu_w$ функции $w$, очевидно, выполнено  неравенство $\nu+\nu_w\geq \nu$ на $D$. Следовательно, функция $u^o:=u_0+w$ с мерой Рисса $\nu_{u^o}=\nu+\nu_w$ --- требуемая в \eqref{in:hM}, но пока для функции $M=M^o$, отличающейся от $M$ в круге 
			$D(z_0,2r_0)$. Вернёмся к прежним обозначениям $M\leq M^{o}$.   Для {\it непрерывной\/} функции  $r$ функции  $M^{*r}$, $(M^o)^{*r}$ также непрерывны в $D$, поскольку обе они из  класса  $L_{\loc}^1(D)$. В то же время субгармоническая функция $u^o\neq \boldsymbol{-\infty}$ ограничена сверху в $D(z_0,3r_0)\Subset D$. Следовательно,  можно выбрать достаточно большую  постоянную $C_3\geq 0$ так, что 
			$u_0:=u^o-C_3\leqslant (M^o)^{*r}$ на $D$  с мерой Рисса $\nu_{u_0}=\nu_{u^o}\geq \nu$. По условию  \eqref{rdD1ad} и Замечанию \ref{rrr}  найдётся подобласть $D_1\Subset D$, включающая в себя $D(z_0,r_0)$,  для которой по построению $M^o$ и определению усреднения выполнены равенства $(M^o)^{*r}=M^{*r}$, а значит и неравенство  $u_0\leqslant M^{*r}$ на $D\setminus D_1$. Поскольку функция $M^{*r}$ непрерывна на $D$, а $u_0$ ограничена сверху на $D_1$, можно выбрать достаточно большую  постоянную $C_4\geq 0$ так, что $u:=u_0-C_4\leq M^{*r}$ на $D$ с мерой Рисса $\nu_u=\nu_{u_0}\geq \nu$, что и требовалось.
			Непрерывная функция  $M$ изначально локально ограничена снизу, что позволяет избежать в доказательстве промежуточного использования функции $M^o$. Кроме того, функция $M$ локально равномерно непрерывна, что позволяет  выбрать непрерывную функцию $r$, удовлетворяющую \eqref{rd:}, с которой $M^{*r}\leq M+1$ на $D$,  и  заменить  правую часть в \eqref{in:hM} на $M$.
			
\end{proof}

\begin{remark}\label{cify} Проверку равномерного по $v$ условия, требующего выполнения  неравенства \eqref{estv+}  в Теореме \ref{th:inver},  можно ослабить в направлении  сужения класса тестовых функций. Конкретнее, для  произвольного множества  $F$ функций  на открытых  множествах из $\CC_{\infty}$ со значениями в $\RR$
используем обозначение $C^{\infty}\cap F$ для всех бесконечно дифференцируемых функций из $F$. 
\begin{addition}[{\rm к Теореме \ref{th:inver}}]\label{add:1} В Теореме\/ {\rm \ref{th:inver}} достаточно требовать, чтобы  с некоторой постоянной $C\in \RR$ 
неравенство  \eqref{estv+}  выполнялось для всех тестовых финитных  функций $v\in C^{\infty} \cap\text{\rm sbh}_0^+(D\setminus D_0;\leq b)$. 
\end{addition}
\begin{proof} Пусть   $v \in  \text{\rm sbh}_0^+(D\setminus D_0;\leq b)$ --- {\it произвольная финитная\/} тестовая функция. По Определениям
\ref{testv}, \ref{df:fi}  найдутся открытые множества  ${\mathcal O}_0, {\mathcal O}$ в $\CC_{\infty}$, обладающие свойствами:  
 $D_0\Subset {\mathcal O}_0 \Subset  {\mathcal O}\Subset D$, функция  $v$ субгармоническая в $D\setminus {\mathcal O}_0$ и 
тождественно равна нулю в $D\setminus \mathcal O$.  В частности, положим 
\begin{equation*}
	\e_0:=\dist ({\mathcal O}_0, \CC_{\infty}\setminus D_0 )>0, \quad
	\e_o:= \dist ({\mathcal O}, \CC_{\infty}\setminus D )>0, \quad \e^o:= \min \{\e_0,\e_o \}>0.
\end{equation*}
Рассмотрим  функцию $a\colon (0,+\infty)\to [0,+\infty)$ класса $C^{\infty}$ с $\supp a\Subset (0,1)$  и нормировкой $\int_0^{+\infty} a(x)\,dx =1$, а также 
 меры $\alpha_{\e}$, определяемые плотностями
\begin{equation*}
	\,d\alpha_{\e} (z):=\frac{1}{\e^2}\,a\bigl(|z|/\e\bigr)\,d\lambda (z), \quad 0<\e< \e^o.
\end{equation*}
Как известно \cite[2.7]{Rans}, \cite[3.4.1]{HK},  для убывающей последовательности строго положительных чисел $\e_n\to 0$ 
последовательность субгармонических бесконечно дифференцируемых функций-свёрток  
$	v_n:=v*\alpha_{\e_n}$, убывая по $n$, поточечно стремиться к функции $v$. В частности, $v\leq v_n$ на $D\setminus D_0$. При этом, поскольку $\partial D_0 $ --- компакт, величины $\sup_{z\in \partial D_0}v_n=:b_n$ стремятся к $b$ при $n\to +\infty$.  
Пусть зафиксировано произвольное число $b'>b$. По построению, начиная с некоторого $n$, функции $v_n\in  {\sbh}_0^+(D\setminus D_0;\leq b')$ {\it тестовые финитные.\/}
  По условию Дополнения найдётся число $C'$, с которым 
\begin{equation*}
\int\limits_{D\setminus D_0}  v_n \,d {\nu} 		\leq	\int\limits_{D\setminus D_0}  v_n \,d {\nu}_M	+C'
\end{equation*}	
  для всех построенных функций $v_n$ при всех  $n\geq n_0:=\const_{\e^o,b',v}$. 
	Отсюда по  разложению Хана--Жордана	для заряда Рисса $\nu_M	=\nu_M^+-\nu_M^-$
	\begin{equation*}
\int\limits_{D\setminus D_0}  v_n \,d ({\nu}+\nu_M^-) 		\leq	\int\limits_{D\setminus D_0}  v_n \,d {\nu}_M^+	+C'
\quad\text{для всех $n\geq n_0$}.
\end{equation*}
	Следовательно, ввиду $v\leq v_n$ на $D\setminus D_0$, получаем 
		\begin{equation*}
\int\limits_{D\setminus D_0}  v \,d {\nu} 		\leq	\int\limits_{D\setminus D_0}  v_n \,d {\nu}_M^+ - 
\int\limits_{D\setminus D_0}  v \,d \nu_M^-	+C'\quad\text{для всех $n\geq n_0$}.
\end{equation*}	
Устремляя  в первом интеграле справа  $n$ к $+\infty$, 	по Теореме  Бепп\'о Л\'еви  имеем 
	\begin{equation*}
\int\limits_{D\setminus D_0}  v \,d {\nu} 		\leq	\int\limits_{D\setminus D_0}  v \,d {\nu}_M^+ - 
\int\limits_{D\setminus D_0}  v \,d \nu_M^-	+C'=
\int\limits_{D\setminus D_0}  v \,d {\nu}_M+C'.
\end{equation*}	
В силу произвола в выборе  финитной функции  $v{\in}  \text{\rm sbh}_0^+(D\setminus D_0;\leq b)$ 
неравенство \eqref{estv+}  с постоянной $C:=C'$ выполнено для всех таких $v$, что и нужно. 
\end{proof}
\begin{addition}[{\rm к Теореме \ref{th:inver}}]\label{add:2}  В случае субгармонической  функции  
 $M\in\sbh(D)\setminus \{\boldsymbol{-\infty}\}$ наряду с Дополнением\/ {\rm \ref{add:1}} в Теореме\/ {\rm \ref{th:inver}} достаточно требовать, чтобы  функция $r$ с условиями\/  \eqref{rd:} была лишь локально отделённой от нуля снизу в  том смысле, что для любого $z\in D$ найдутся числа $t_z,c_z >0$, для которых 	$r(z')\geq c_z$  	для всех  $z'\in D(z,t_z)\Subset D$.

\end{addition}
\begin{proof} Элементарно, из соображений, использующих компактность, например, с использованием исчерпания области $D$ относительно компактными подобластями, устанавливается
\begin{lemma}\label{l:dz}
Для любой отделённой от нуля снизу функции $r$ на $D$, удовлетворяющей условиям\/ \eqref{rd:}, найдется непрерывная, даже бесконечно дифференцируемая,
функция $r'\leq r$, по-прежнему удовлетворяющая условиям    \/ \eqref{rd:}.
\end{lemma}
Применяя Теорему \ref{th:inver} с непрерывной  функцией $r'$ вместо $r$ строим требуемую $u\leq {M}^{*r'}\leq M^{*r}$, где при последнем переходе использовано возрастание  усреднений в круге субгармонической функции $M$ по радиусу   круга. 
\end{proof}
\end{remark}

\begin{corollary}\label{holD} Пусть в условиях Теоремы\/ {\rm \ref{th:inver}} функция $M$ непрерывна, а мера
 $\nu\overset{\eqref{df:nZS}}{:=}n_{\tt Z}$ --- считающая мера некоторой последовательности точек\/  ${\tt Z}=\{{\tt z}_k\}_{k=1,2,\dots}\subset D$ без точек сгущения в $D$, т.\,е. условие \eqref{estv+}
заменяется на эквивалентное  неравенство  
\begin{equation*}\label{estv++}
\sum_{{\tt z}_k\in D\setminus D_0}  v \bigl({\tt z}_k\bigr) 		\leq	\int\limits_{D\setminus D_0}  v \,d {\nu}_M	+C\quad
\text{для всех финитных   $v\in   C^{\infty}\cap \text{\rm sbh}_0^+(D\setminus D_0;\leq b)$.}
\end{equation*}
Для единообразия  формулировки  предполагаем, что $D\subset \CC$, т.\,е.~$\infty \notin D$.  

Тогда  для любого числа $\e>0$ найдётся ненулевая функция $f\in \Hol (D)$, для которой  $ f({\tt Z})=0$ и выполнены неравенства
\begin{equation}\label{efM}
	\log \bigl|f(z)\bigr|\leq  \frac{1}{\pi r^2} \int_{D(r)}M(z+z') \,d\lambda(z') +
\log \frac{1}{r}+(1+\varepsilon )\log\bigl(1+|z|+r\bigr) 
\end{equation}
 	в каждой  точке $z\in D$ с любым числом $r$, удовлетворяющим условию 
\begin{equation}\label{0r1}
	0<r<\dist (z, \CC_{\infty}\setminus D). 
\end{equation}
\end{corollary}
\begin{proof} По Теореме  \ref{th:inver} и Дополнениям  \ref{add:1}--\ref{add:2} существует функция $u\in \sbh(D)\setminus \{\boldsymbol{-\infty}\}$ с мерой Рисса $\nu_u\geq n_{\tt Z}$, удовлетворяющая неравенству $u\leq M$ на $D$. Пусть $f_{\tt Z}\in \Hol (D)$ --- некоторая функция с последовательностью нулей $\Zero_{f_{\tt Z}}=\tt Z$, которая всегда существует по классической Теореме  Вейерштрасса. Тогда для  нетривиальной функции $w:=u-\log |f_{\tt Z}|\in \dsbh (D)$  её заряд Рисса 
$\nu_w=\nu_u- n_{\tt Z} \geq 0$ положителен, т.\,е.  $\nu_w$ --- мера Рисса уже субгармонической функции $w\in \sbh(D)\setminus \{\boldsymbol{-\infty}\}$.
При этом по построению 
\begin{equation}\label{infZD}
\log |f_{\tt Z}|+w\leq M \quad \text{на $D$}.
\end{equation}

	Усредним по кругам $D(z,r)$ по мере Лебега $\lambda$ обе части \eqref{infZD}:
	\begin{equation*}
		\frac{1}{\lambda \bigl(D(r)\bigr)}\int_{D(z,r)}\log |f_{\tt Z}| \, d\lambda
		+\frac{1}{\lambda \bigl(D(r)\bigr)}\int_{D(z,r)} w \, d\lambda
		\leq \frac{1}{\lambda \bigl(D(r)\bigr)}\int_{D(z,r)}M \, d\lambda
	\end{equation*}
	 для всех $z\in D$ и $0<r<\dist (z, \CC_{\infty}\setminus D)$.   
	Поскольку субгармонические функции не превышают их усреднений по кругам, последнее соотношение  даёт
	\begin{equation}\label{estf+}
		\log \bigl|f_{\tt Z}(z)\bigr| 		+\frac{1}{\lambda \bigl(D(r)\bigr)}\int_{D(z,r)} w \, d\lambda
		\leq \frac{1}{\lambda \bigl(D(r)\bigr)}\int_{D(z,r)}M \, d\lambda \quad\text{для всех $z\in D$.}
	\end{equation}
\begin{lemma}[{\cite[Theorem 1]{KhB15}, \cite[Теорема 1]{BaiKha16}, \cite[Теорема 3]{BaiKha16_mz}}]\label{l:hmin} Для любой функции $w\in \sbh(D)\setminus \{\boldsymbol{-\infty}\}$ и  любого числа $\e>0$ найдётся ненулевая функция $f_{\e}\in \Hol (D)$, для которой 
\begin{equation}\label{logfr}
	\log \bigl|f_{\e}(z)\bigr|\leq \frac{1}{\lambda \bigl(D(r)\bigr)}\int_{D(z,r)} w \, d\lambda
+
 \log \frac{1}{r}+(1+\varepsilon )\log\bigl(1+|z|+r\bigr) 
	\end{equation}
	во всех точках $z\in D\subset \CC$ при любых $r$, удовлетворяющих условиям  \eqref{0r1}.
\end{lemma}
	
По Лемме \ref{l:hmin}  для всех $z\in D$ и всех   $r$ из  \eqref{0r1} согласно \eqref{estf+}  имеем
\begin{equation*}\label{estfZ} 
		\log \bigl|f_{\tt Z}(z)\bigr| 	 +\log \bigl|f_{\e}(z)\bigr|\leq 
		\frac{1}{\lambda \bigl(D(r)\bigr)}\int_{D(z,r)}M \, d\lambda 	+ 	
 \log \frac{1}{r}+(1+\varepsilon )\log\bigl(1+|z|+r\bigr) ,
	\end{equation*}
	и функция $f:=f_{\tt Z}f_{\e}\neq 0$ --- требуемая, поскольку  выполнено  \eqref{efM} и $f_{\tt Z}({\tt Z})=0$.
\end{proof}
\begin{remark}\label{r:tim} Следствие \ref{holD} остаётся в силе и при $\infty \in D\subset \CC_{\infty}$, если 
в последнем слагаемом в правой части \eqref{efM} в подлогарифменном выражении заменить  величину $|z|$ на $\frac{1}{|z-z_0|}$,  где $z_0\notin D$ --- произвольная точка в $\CC\setminus D$. 

Иногда переменные $r$ и $|z|$ в последних двух слагаемых в правой части \eqref{logfr} бывает удобно разделить. Один из вариантов
даётся  неравенством
\begin{equation*}
	\log \frac{1}{r}+(1+\varepsilon )\log\bigl(1+|z|+r\bigr)\leq \log \frac{(1+r)^{1+\e}}{r} +	 (1+\varepsilon )\log\bigl(1+|z|\bigr).
\end{equation*}
\end{remark}

\subsection{Обращения с более узкими классами тестовых функций}\label{an_cklf} 
\subsubsection{\underline{Обращение с функциями Грина}}\label{Grf} 

В теореме обращения этого пункта используются лишь продолженные функции Грина специальной системы 
относительно компактных регулярных подобластей области\footnote{Здесь уже, в отличие от Теоремы \ref{th:inver}, область $D\subset \CC_{\infty}$ произвольная.} $D$, содержащих $D_0$ с фиксированным полюсом $z_0\in D_0$.
 Отметим, что каждая такая {\it функция Грина --- тестовая финитная функция для области\/ $D$ вне подобласти\/ $D_0$ }
по Определениям \ref{testv}, \ref{df:fi}  согласно Предложению \ref{PJ=test} и его реализации в Примере \ref{exp2}.

\begin{definition}[{\rm {\large(}см. \cite[Определение 1]{Kh07}, \cite[Определение 1]{KhS09}{\large)}}]\label{opexd} 
Систему {\it регулярных\/} областей 
			$\mathcal U_{D_0} (D)\subset \{D'\Subset D \colon 	\; D_0\subset D' \}$
		 называем {\it  регулярной оптимально исчерпывающей в\/} $D$ с центром $D_0$, если для любых двух областей $D_1$ и $D_2$ 
при $D_0\subset D_1\Subset D_2\subset D$ выполнены два    условия:

\begin{enumerate}[{1)}]
	\item
 {\it  можно подобрать область\/ $D'\in \mathcal U_{D_0}(D )$ так, что\/  $D_1\Subset D' \Subset D_2$ и каждая непустая ограниченная компонента связности дополнительного множества\/ ${\mathbb C}_{\infty} \setminus D'$ пересекает дополнение\/ ${\mathbb C}_{\infty} \setminus D_2$};

\item
 {\it для любой области $D'\in  \mathcal U_{D_0}(D )$
найдется область $D'' \in  \mathcal U_{D_0}(D )$, с которыми 
$D_1\Subset D'' \Subset D_2$  и объединение 
$D'' \cup D'$ также принадлежит\/ $\mathcal U_{D_0}(D )$};
\end{enumerate}
и, кроме того,  эта система {\it условно инвариантна относительно сдвига в\/} $D$, т.\,е. из $D' \in \mathcal U_{D_0} (D )$, 
$z\in {\mathbb C}$ и $D_0\subset D'+z \Subset \Omega$ следует, что $D'+z \in \mathcal U_{D_0} (D )$. 
\end{definition}

Простым  примером  регулярной оптимально исчерпывающей системы областей 
 может служить {\it специальная система $\mathcal U_{D_0}^d(D)$ всевозможных связных объединений $D'\supset D_0$ 
конечного числа кругов $D(z,t)\Subset D$, исключая те области $D'$, в дополнении ${\mathbb C}_{\infty} \setminus D'$ которых есть изолированные  точки}.   С такими же исключениями круги в этом примере можно заменить на относительно компактные в $D$ всевозможные, вообще говоря,  с несвязной границей,  $n$-угольники или, более общ\'о, 
односвязные подобласти \cite[Theorems 4.2.1--4.2.2]{Rans}, \cite[{2.6.3}]{HK} какого-либо специального типа.  

\begin{theorem}\label{th3} Пусть   $M=M_+-M_-\in \dsbh (D)$  с зарядом Рисса $\nu_M$, где $M_+\in \sbh (D)\cap C(D)$ и 
$M_-\in \sbh(D)$, а также $z_0\in D\cap \dom M \Subset D$.  

Пусть $\nu\geq 0$ --- борелевская мера на  $D$\/ и в точке $z_0$ для меры $\nu$ выполнено условие типа \eqref{dMa}, т.\,е.  
при некотором  $r_0>0$, для которого   $D(z_0,r_0)\Subset D$,  конечен один, или каждый,  из интегралов в  \eqref{dMa+}.
Пусть\/ $\mathcal U_{D_0} (D)$ --- 	  {\it  регулярная оптимально исчерпывающая  система областей в\/} $D$ с центром $D_0\subset D$ и $z_0\in D_0$. 
Пусть  с некоторой постоянной $C\in \RR$  выполнены неравенства
\begin{equation}\label{estv+g}
\int\limits_{ D \setminus \{z_0\}}  g_{D'}(\cdot ,z_0) \,d {\nu} 		\leq	\int\limits_{D\setminus \{z_0\}} g_{D'}(\cdot ,z_0) \,d {\nu}_M	+C
\text{ для всех $D'\in \mathcal U_{D_0} (D)$}.
 \end{equation}
Тогда  найдётся функция  $u\in \text{\rm sbh}(D)\setminus \{\boldsymbol{-\infty}\}$ с мерой Рисса $\nu_u\geq \nu$, удовлетворяющая ограничению   $u\leq M$ на $D$.
\end{theorem}
\begin{proof}  Пусть  $\nu_{M_+}$ и $\nu_{M_-}$ --- меры Рисса соотв. функций ${M_+}$ и ${M_-}$. Равномерную по $C$ серию неравенств  \eqref{estv+g}  тогда можно записать в обозначении $\nu_1:=\nu +\nu_{M_-}$ в виде
\begin{equation}\label{estv+g+}
\int\limits_{ D \setminus \{z_0\}}  g_{D'}(\cdot ,z_0) \,d {\nu}_1 		\leq	\int\limits_{D\setminus \{z_0\}} g_{D'}(\cdot ,z_0) \,d \nu_{M_+}	+C
\text{ для всех $D'\in \mathcal U_{D_0} (D)$},
 \end{equation}
где $\nu_1, \nu_{M_+}\in \mathcal M^+(D)$ --- уже {\it положительные меры.\/}  
Очевидно, существует какая-нибудь субгармоническая в $D$ функция $u_1\neq \boldsymbol{-\infty}$ с мерой Рисса $\nu_1$.  Для её меры Рисса $\nu_1$,  
ввиду  условия $z_0\in \dom M $, т.\,е. условия  \eqref{dMa}, а также  условия \eqref{dMa+} на $z_0$, выполнено условие  \eqref{dMa+} с  заменой $\nu$ на $\nu_1$. Следовательно,  $M_-(z_0)\neq -\infty$ и  обязательно $u_1(z_0)\neq -\infty$.  Далее потребуется 
 вариации утверждений из\footnote{К сожалению,  в формулировке Основной теоремы из нашей работы \cite{Kh07}, на промежуточном этапе доказательства которой и основано  \cite[Теорема (основная)]{KhS09},  
	допущена досадная опечатка в знаках $\pm$. Так, в её формулировке  (2.11) из п.~(h1) 
		должно выглядеть в точности как \eqref{est:g}.  Дальнейший комментарий --- в сноске к \cite[Теорема (основная)]{KhS09}.} \cite[Основная Теорема, Теорема 6]{Kh07}:
\vskip 1mm
	{\sc Теорема D} {\large(}частный случай  \cite[Теорема (основная)]{KhS09}{\large)}.
{\it Пусть функция\/  $M \in \sbh(D)$  с мерой Рисса $\nu_{M}$ ограничена снизу в некоторой окрестности замыкания   $\clos D_0$, 
 $u\in \sbh(D)$ --- субгармоническая функция с мерой Рисса $\nu$ на $D$ и $u(z_0)\neq -\infty$,  система областей\/ $\mathcal U_{D_0} (D)$ --- 	регулярная оптимально исчерпывающая для $D$ с центром $D_0\ni z_0$. Если имеет место соотношение 
\begin{equation}\label{est:g}
-\infty < \inf_{D' \in \mathcal U_{D_0}(D )}
\left(-\int_{D \setminus \{z_0\}} g_{D'}(\cdot  , z_0) \, d \nu_u+
\int_{D\setminus \{z_0\}} g_{D'} (\cdot , z_0 )\, d \nu_M\right),   
\end{equation}  
то  для любой непрерывной функции\/  $r \colon D\to \RR$, удовлетворяющей условию\/ 
\eqref{rdD},  найдётся субгармоническая в\/ $D$ функция\/ $v\neq \boldsymbol{-\infty}$, гармоническая в некоторой окрестности точки $z_0$, с которой   $u +v\overset{\eqref{in:hM}}{\leqslant} M^{*r}$ на $D$. При этом  если функция $M$ ещё и  непрерывная в $D$, то переменное усреднение $M^{*r}$ в правой части последнего неравенства  можно заменить на саму функцию $M$.
}

\begin{addition}[{\rm к Теореме D}]\label{add:D} Условие Теоремы\/ {\rm D}  об ограниченности снизу функции $M\in \sbh (D)\setminus \{\boldsymbol{-\infty}\}$  в  окрестности замыкания $\clos D_0\Subset D$  при дополнительном ограничении \eqref{rdD1ad} на функцию $r$
можно заменить на более слабое условие $M(z_0)\neq -\infty$.
\end{addition}
\begin{proof}[Дополнения \ref{add:D}]  Пусть для произвольной субгармонической функции $M$ с $M(z_0)\neq -\infty$ выполнено \eqref{est:g}. В эквивалентной формулировке это означает, что найдётся число  $C\in \RR$, с  которым 
\begin{equation}\label{estv+g+ad}
\int\limits_{ D \setminus \{z_0\}}  g_{D'}(\cdot ,z_0) \,d {\nu} 		\leq	\int\limits_{D\setminus \{z_0\}} g_{D'}(\cdot ,z_0) \,d {\nu}_{M}	+C
\text{ для всех $D'\in \mathcal U_{D_0} (D)$},
 \end{equation}  
Всегда можно выбрать некоторую регулярную подобласть $D^0 \Subset D$, включающую в себя $D_0\Subset D^0$. Рассмотрим   новую функцию $M_0$, 
совпадающую с $M$ на $D\setminus D^0$ и  заданную как гармоническое продолжение функции $M\bigm|_{\partial D^0}$ внутрь области $D^0$.
Тогда \cite[Теорема 2.18]{HK} $M_0\in \sbh (D)$, $M\leq M_0$ на $D$ и, очевидно, функция $M_0$ ограничена снизу в окрестности $D^0$ замыкания $\clos D_0$. 
Из  классической формулы Пуассона--Йенсена \eqref{f:PJgo} применительно  к $M_0$ и $M$
\begin{multline*}
\int_{D\setminus \{z_0\}} g_{D'}(\cdot ,z_0) \,d {\nu}_{M} \overset{\eqref{f:PJgo}}{=}
	\int_{D'} M\,d \omega_{D'}(z_0,\cdot) -M(z_0) \\
	\leqslant \int_{D'} M_0\,d \omega_{D'}(z_0,\cdot) -M(z_0)
	\overset{\eqref{f:PJgo}}{=} 
	\int_{D\setminus \{z_0\}} g_{D'}(\cdot ,z_0) \,d {\nu}_{M_0}+M_0(z_0)-M(z_0)
\end{multline*}
для любой области  $D'\in \mathcal U_{D_0} (D)$, где $\nu_{M_0}$ --- мера Рисса функции $M_0$.  Из \eqref{estv+g+ad}
\begin{equation*}
\int\limits_{ D \setminus \{z_0\}}  g_{D'}(\cdot ,z_0) \,d {\nu} 		\leq	\int\limits_{D\setminus \{z_0\}} g_{D'}(\cdot ,z_0) \,d {\nu}_{M_0}	+C_0
\text{ для всех $D'\in \mathcal U_{D_0} (D)$},
 \end{equation*} 
где $C_0:=C+M_0(z_0)-M(z_0)\in \RR$. Следовательно, по Теореме D существует функция $v\in\sbh(D)$, с которой $u+v\leq M_0^{*r}$ на $D$, где $M_0^{*r}$ --- непрерывная функция на $D$, как, впрочем,  и функция $M^{*r}$.  При этом  из  ограничения \eqref{rdD1ad}  на $r$ следует, что найдётся подобласть $D_1 \Subset D$ cо свойствами $D^0\Subset D_1$ и $D\bigl(z, r(z)\bigr)\cap D^0=\varnothing$ для всех $z\in D\setminus D_1$. 
По построению $M_0$ имеем $M_0^{*r}=M^{*r}$ на $D\setminus D_1$ и, в силу непрерывности этих функций на $D$, существует постоянная  $C_0\geq 0$, для которой $M_0^{*r}\leq M^{*r}+C_0$ уже всюду на $D$. Таким образом, $u+(v-C_0)\leq M^{*r}$  на $D$. Тем самым Дополнение \ref{add:D} доказано. 
\end{proof}

Продолжим доказательство Теоремы \ref{add:D}. По Теореме D с Дополнением \ref{add:D}, применённым к функциям $u_1$ и {\it непрерывной\/} функции 
$M_+$ вместо соотв. $u$ и $M$, ввиду \eqref{estv+g+} найдётся функция $v \in \sbh (D)$, гармоническая в окрестности точки $z_0$, с которой $u_1+v\leq M_+$ на $D$. По построению   $u_1\in \sbh (D)$ с  мерой Рисса  $\nu_1:=\nu +\nu_{M_-}$. Следовательно,  мера Рисса функции $u_0:=u_1-M_-$ --- это мера $\nu$, т.\,е. существует функция $u:=u_0+v \in \sbh(D)$ с мерой Рисса $\nu_u\geq \nu$, для которой $u\leq M_+-M_-=M$ на $D$, что и требовалось.
\end{proof}

\begin{corollary}\label{holDg} Пусть в условиях Теоремы\/ {\rm \ref{th3}}  мера
 $\nu\overset{\eqref{df:nZS}}{:=}n_{\tt Z}$ --- считающая мера некоторой последовательности точек\/  ${\tt Z}=\{{\tt z}_k\}_{k=1,2,\dots}\subset D$ без точек сгущения в $D$, т.\,е. $z_0\notin {\tt Z}$, а  условие \eqref{estv+g}
заменяется на  неравенства  
\begin{equation*}
\sum\limits_{ {\tt z}_k\in D' }  g_{D'}({\tt z}_k ,z_0)  		\leq	\int\limits_{D\setminus \{z_0\}} g_{D'}(\cdot ,z_0) \,d {\nu}_M	+C
\text{ для всех $D'\in \mathcal U_{D_0} (D)$}.
 \end{equation*}
Тогда  в предположении, что $D\subset \CC$, т.\,е.~$\infty \notin D$,  для любого числа $\e>0$ найдётся ненулевая функция $f\in \Hol (D)$, для которой  $ f({\tt Z})=0$ и выполнены неравенства \eqref{efM} в каждой  точке $z\in D$ с любым числом $r$, удовлетворяющим условию \eqref{0r1}.
При $\infty \in D$  изменения как  в  Замечании\/ {\rm  \ref{r:tim}}.
\end{corollary}
Выводится из Теоремы  \ref{th3} так же, как  Следствие \ref{holD}  из  Теоремы \ref{th:inver}.  

\begin{remark} Регулярную  оптимально исчерпывающую   систему областей $\mathcal U_{D_0} (D)$ с центром $D_0\subset D$ в Теореме \ref{th3}
и Следствии \ref{holDg} на основе анализа тонких совместных результатов В.~Хансена и И.~Нетуки \cite{HN} об аппроксимации мер Йенсена гармоническими мерами можно заменить на систему областей $D'\Subset D$, включающих в себя  область $D_0\Subset D$ и полученных из 
 исчерпывающей область $D$ последовательности областей $D_n\Subset D$, $n\in \NN$  с аналитической или кусочно линейной или иной <<хорошей>> границей удалением из $D_n$ произвольного конечного набора  попарно непересекающихся замкнутых кругов.
\end{remark}
\subsubsection{\underline{Обращение с аналитическими и полиномиальными дисками}}\label{apnd}

Важный подкласс в классе $J_{z_0}(D)$ мер Йенсена порождают аналитические диски в $D$ с центром $z_0$.
{\it Аналитическим диском  в области\/ $D$ с центром в точке\/}  $z_0\in D$\/ называется функция $g \colon  \clos \DD \to D$, непрерывная на $\clos \DD$ с  голоморфным сужением в $\DD$, для которой $f(0)=z_0$ {\large(}см. \cite{Poletsky}, \cite{Po99}, \cite{CR+}{\large)}\footnote{ Если используются различные виды аналитических дисков, то рассматриваемые здесь аналитические диски у Е.~Полецкого выделяются  как {\it замкнутые\/} \cite{Po99}.}.  Для любого такого аналитического диска $g$ легко показать, что функция 
\begin{equation*}
	z\mapsto \frac{1}{2\pi}\int_0^{2\pi} \log \left|1-\frac{g(e^{i\theta})}{z}\right| \, d \theta , \quad z\in D\setminus \{z_0\}, 
\end{equation*}
--- потенциал Йенсена внутри области $D$ с полюсом в точке $z_0$, т.\,е.  это финитная тестовая субгармоническая функция для $D$ вне $D_0$ для любой подобласти $D_0\Subset D$, содержащей точку $z_0\in D$.

Если аналитический диск  $g$ в $D$ с центром $z_0\in D$ --- {\it многочлен\/} комплексной переменной,  то естественно называть его {\it полиномиальным диском в $D$ с центром в точке $z_0\in D$\/}. 
\begin{theorem}\label{th:4} Пусть функция  $M\in \dsbh (D) $ с зарядом Рисса $\nu_M$,  точка $z_0$ и мера $\nu\geq 0$ такие же, как в Теореме\/ {\rm \ref{th3}}. Если существует постоянная  $C\in \RR$, с которой неравенство 
\begin{equation*}
	\int\limits_{g(\clos \DD)} \int_0^{2\pi} \log \left|1-\frac{g(e^{i\theta})}{z}\right| \, d \theta \, d\nu 	
	 \leq  	\int\limits_{g(\clos \DD)}  \int_0^{2\pi} \log \left|1-\frac{g(e^{i\theta})}{z}\right| \, d \theta \, d\nu_M+C 	
\end{equation*}
выполнено для всех аналитических или только полиномиальных  дисков $g$ в $D$ с центром $z_0$, то  
найдётся функция  $u\in \text{\rm sbh}(D)\setminus \{\boldsymbol{-\infty}\}$ с мерой Рисса $\nu_u\geq \nu$, удовлетворяющая ограничению 
 $u\leq M$ на $D$.
\end{theorem}
В случае субгармонической  функции $M$ обсуждение схемы доказательства Теоремы  \ref{th:4} 
содержится  в  \cite[1.2.1--1.2.2, Дополнения 1.2.3, 1.2.4]{Khsur}. Это одна из причин, по которой  мы опускаем здесь 
доказательство  Теоремы  \ref{th:4}. Другая в том, что многомерный вариант Теоремы \ref{th:4} в $\CC^n$ более естественен и будет рассмотрен с применениями в ином месте. Как и Следствия \ref{holD}, \ref{holDg} доказывается 
\begin{corollary}\label{holDgad} В условиях Теоремы\/ {\rm \ref{th:4}}  вместо меры
рассмотрим последовательность точек\/  ${\tt Z}=\{{\tt z}_k\}_{k=1,2,\dots}\subset D$ без точек сгущения в $D$. 
Если существует постоянная  $C\in \RR$, с которой неравенство 
\begin{equation*}
	  \int_0^{2\pi} \sum\limits_{ {\tt z}_k\in g(\clos \DD) } \log \left|1-\frac{g(e^{i\theta})}{{\tt z}_k}\right| \, d \theta  	
	 \leq  	\int\limits_{g(\clos \DD)}  \int_0^{2\pi} \log \left|1-\frac{g(e^{i\theta})}{z}\right| \, d \theta \, d\nu_M(z)+C 	
\end{equation*}
выполнено для всех аналитических или только полиномиальных  дисков $g$ в $D$ с центром $z_0$, то в предположении, что $\infty \notin D$,  для любого числа $\e>0$ найдётся ненулевая функция $f\in \Hol (D)$, для которой  $ f({\tt Z})=0$ и выполнены неравенства \eqref{efM} в каждой  точке $z\in D$ с любым числом $r$, удовлетворяющим условию \eqref{0r1}.
При $\infty \in D$  изменения как  в  Замечании\/ {\rm  \ref{r:tim}}.
\end{corollary}

\section{Радиальные  мажоранты $M$}\label{S3}

Функция $f$ на подмножестве  $S\subset \CC_{\infty}$, удовлетворяющем условию $e^{i\theta}S:=
\{e^{i\theta}z\colon z\in S\}=S$ для всех $\theta \in \RR$, {\it радиальная}, если $f(ze^{i\theta})=f(z)$ для любых $z\in S$ и $\theta \in \RR$. Очевидно, такая функция однозначно определяется своим сужением на $S\cap [0,+\infty]$. Для такой  функции $f$  и $r\in S\cap [0,+\infty]$ через $f'_{\leftd}(r)$ и $f'_{\rightd}(r)$ обозначаем соответственно  левую и правую производную сужения функции $f\bigm|_{S\cap \RR^+}$ в тех   точках $r\in S\cap [0,+\infty]$, в которых такая производная имеет смысл  и существует.  При этом для $r=+\infty$ используем инверсию $\star$ из раздела \ref{s1.1}  Введения, или
$\star_{z_0}$ в обозначении  \eqref{inv_z0}.

\subsection{{Радиальные субгармонические функции}}\label{ssCr} 
Для   $0\leq r_1<r_2\leq +\infty$ полагаем $A(r_1,r_2):=\{z\in \CC \colon r_1<|z|<r_2\}\subset \CC$ --- {\it открытое кольцо.}
\begin{propos}\label{prsbr} Пусть $M\colon A(r_1,r_2)\to \RR$ --- радиальная функция  с сужением $m=M\bigm|_{(r_1,r_2)}$. Следующие  пять  утверждений попарно эквивалентны.
\begin{enumerate}
	\item\label{1s}   Функция $M$ --- субгармоническая в  кольце $A(r_1,r_2)$, $M\neq \boldsymbol{-\infty}$. 	
	\item\label{2s} Функция $m\colon (r_1, r_2)\to \RR$ --- выпуклая относительно $\log$, т.\,е. суперпозиция $m\circ \exp$ --- выпуклая функция на $(\log r_1, \log r_2)\subset \RR$.
	\item\label{3s} Функция $m$ непрерывна, во  всех точках  $r\in (r_1, r_2)$ существуют непрерывная слева  левая производная $m'_{\leftd}(r)$ и/или непрерывная справа правая производная $m'_{\rightd}(r)$, 	а  функции $r\mapsto rm'_{\leftd}(r)$ и/или $r\mapsto rm'_{\rightd}(r)$ возрастающие. При этом	$ m'_{\leftd}(r)\leq m'_{\rightd}(r)$  во всех точках $r\in (r_1,r_2)$ и непрерывны и совпадают всюду на $(r_1,r_2)$ за исключением счётного множества значений $r$.
\item\label{4s} 	Для некоторого (любого) $r_0\in (r_1,r_2)$ и некоторой  возрастающей непрерывной вне счётного подмножества из $(r_1,r_2)$ функция $n_0\colon (r_1,r_2)\to \RR$ справедливо представление  с интегралом Римана
\begin{equation}\label{mnrep}
	m(r)=m(r_0)+\int_{r_0}^r \frac{n_0(t)}{t}\, dt, \quad  r\in (r_1,r_2). 
\end{equation}
Здесь  можно выбрать функцию $n_0$ по правилу
\begin{equation}\label{n0rep}
n_0(r):=	rm'_{\leftd}(r) \quad \text{ или } \quad  n_0(r):= rm'_{\rightd}(r), \quad   r\in (r_1,r_2).
\end{equation}
\item\label{5s} Функция $m$ полунепрерывна сверху, локально интегрируема по мере Лебега на открытом интервале $(r_1,r_2)$ и  
$r\mapsto \bigl(rm'(r)\bigr)'$ (здесь уже дифференцирование распределения) ---   положительное распределение, или обобщённая функция,  на  классе финитных гладких функций на $(r_1,r_2)$, и даже    положительная мера Радона, т.\,е. борелевская мера на $(r_1,r_2)$.
 \end{enumerate}
\end{propos}
\begin{proof} Все эти эквивалентности  известны, но разбросаны по различным источникам. Так, эквивалентность 
$\ref{1s}\Leftrightarrow \ref{2s}$ --- в \cite[1.II.10]{Doob}.  Импликация $\ref{2s}\Rightarrow \ref{3s}$ сразу следует из соответствующих свойств выпуклой функции $m\circ \exp$ --- \cite[Theorem 1.1.7]{Ho}, \cite[1.XIV.4]{Doob}. Представление \eqref{mnrep} при переходе $\ref{3s}\Rightarrow \ref{4s}$ получаем сразу при выборе $n_0$ как в  \eqref{n0rep}. Импликация $\ref{4s}\Rightarrow \ref{5s}$ --- результат  дифференцирования по $r$ тождества \eqref{mnrep} и возрастания функции $n_0$, почти всюду совпадающей с функциями из \eqref{n0rep} \cite[гл.~VIII, \S~2]{Nat}.  Наконец, если справедливо 
 \ref{5s}, то  в полярных координатах $re^{i\theta}=z\in A(r_1,r_2)$, $r\in (r_1,r_2)$, оператор Лапласа $\Delta$ от полунепрерывной сверху локально интегрируемой по мере Лебега $\lambda$ функции  $M(re^{i\theta}):=m(r)$, $r\in (r_1,r_2)$, $\theta \in \RR$, записывается через плотности мер и (тензорное) произведение положительных мер \cite[гл.~IV, \S~8]{Schwartz}
\begin{equation}\label{df:del}
	\frac{1}{2\pi}\,(\Delta M) (re^{i\theta})\,d \lambda (re^{i\theta})=d\bigl(rm_{\leftd}'(r)\bigr)\otimes \frac{1}{2\pi} \,d\theta\geq 0.
\end{equation}
Таким образом, $M\in \sbh \bigl(A(r_1,r_2)\bigr)$ и цикл замыкается на  утверждении \ref{1s}.
\end{proof}
\begin{remark}\label{r11}  Мера Рисса радиальной функции $M$ описана  в \eqref{df:del}, где левую производную $m_{\leftd}'(r)$ можно заменить  на 
правую $m_{\rightd}'(r)$, $r\in (r_1,r_2)$.
\end{remark}

Предложение \ref{prsbr} вместе с \cite[Theorem 2.6.6]{Rans}, \cite[Предложение 5.1]{Kond}  сразу даёт

\begin{propos}\label{prsdii} При $0<R\leq +\infty$ радиальная функция $M\colon D(R)\to \RR$   субгармоническая тогда и только тогда, когда её сужение  $m:=M\bigm|_{(0,R)}$ обладает одним (любым)  из попарно эквивалентных свойств\/ {\rm \ref{2s}--\ref{5s}} Предложения\/ {\rm \ref{prsbr}} при  $r_1=0$ и $r_2=R$,  сужение $m_0:=M\bigm|_{[0,R)}$ --- возрастающая непрерывная в нуле функция, а представление \eqref{mnrep} в силу непрерывности $M$ в нуле можно записать в виде
\begin{equation}\label{repM0}
	M(r)=M(0)+\int_0^{r} \frac{n(t)}{t} \, dt, \quad r\in [0,R),
\end{equation}
где возрастающую  функцию $n\colon [0,R)\to \RR$ можно задать как непрерывную справа равенствами $n(r):= r(m_0)_{\rightd}'(r)\geq 0$ во всех точках  $r\in [0,R)$. 
\end{propos}
Из Предложения \ref{prsdii}, используя   инверсию $\star$, как  в разделе \ref{s1.1}, получаем
\begin{propos}\label{prsdiii} Для\/ $0<R<+\infty$  радиальная функция\/ $v\colon \CC_{\infty}\setminus \overline{D}(R)\to \RR$ субгармоническая тогда и только тогда, когда её сужение  $m:=v\bigm|_{(R,+\infty)}$ обладает одним (любым)  из попарно эквивалентных свойств\/ {\rm \ref{2s}--\ref{5s}} при  $r_1=R$ и $r_2=+\infty$, сужение $m_0:=v\bigm|_{(R,+\infty]}$ --- убывающая непрерывная в точке $+\infty$ функция, а представление вида \eqref{repM0} можно записать как
\begin{equation}\label{repv0}
	v(r)=v(\infty)-\int_r^{+\infty} \frac{n(t)}{t} \, dt, \quad r\in (R,+\infty],
\end{equation}
где возрастающую  функцию $n\colon (R,+\infty]\to \RR$ можно задать  как непрерывную слева равенствами $n(r):= r(m_0)_{\leftd}'(r)\leq 0$ во всех точках  $r\in (R,+\infty]$. 
\end{propos}

\subsection{{Случай $D=\CC$}}\label{ssCrC}    Пусть $0< {R_0}<+\infty$.  
Для произвольного  $b\in \RR^+$ введем класс $\decr^+[{R_0},+\infty;\leq b)$ всех {\it положительных убывающих 
} функций $d$  на $[R_d,+\infty]$ при некотором  $R_d\in (0,{R_0})$  {\it  с ограничением\/} 
\begin{equation}\label{dfdd}
 \int_{R_0}^{+\infty} \frac{d(t)}{t}\, dt\leq b .
\end{equation}
Полагаем 
\begin{equation}\label{dinfty+}
	\decr^+[{R_0},+\infty;<+\infty):=\bigcup_{b>0}\decr^+[{R_0},+\infty;\leq b) 
\end{equation}
--- класс всех положительных убывающих функций в открытых окрестностях луча-отрезка $[R_0,+\infty]\subset [0,+\infty]$ с ограничением
\begin{equation}\label{dfdd1}
 \int_{R_0}^{+\infty} \frac{d(t)}{t}\, dt< +\infty .
\end{equation} 
Из \eqref{dfdd1} для таких функций $d$, в частности, следует $\lim_{t\to+\infty}d(t)=0$.

\begin{propos}\label{prsds} Если $d\in \decr^+[{R_0},+\infty;\leq b)$, $0<R_d<{R_0}$, то функция 
\begin{equation}\label{dfvdc}
	v(z):=\int_{|z|}^{+\infty}\frac{d(t)}{t} \, dt , \quad  z\in \CC_{\infty}\setminus \overline{D}(R_d) ,
\end{equation}
принадлежит  классу   $\sbh_0^+\bigl(\CC\setminus D({R_0});\leq b\bigr)$, определённому в\/ \eqref{sbh+}.
\end{propos}
\begin{proof}  Из определения \eqref{dfdd}--\eqref{dfvdc} с возрастающей функцией  $n:=-d$ имеет место представление \eqref{repv0} Предложения 
\ref{prsdiii} с требуемыми в нём свойствами для функции $n$ и ограничениями $v\overset{ \eqref{dfdd}}{\leq}b $, $\lim_{z\to \infty}v(z)=0$.
\end{proof}

\begin{corollary}\label{coref}
Пусть\/ $m\colon [{R_0},+\infty)\to \RR$ --- возрастающая непрерывная функция, у которой в каждой точке $r\in [{R_0},+\infty)$ существует непрерывная справа правая производная $m_{\rightd}'(r)$ и функция $r\mapsto rm'_{\rightd}(r)$ возрастает. Пусть $f\neq 0$ --- целая функция, обращающаяся в нуль на последовательности ${\tt Z}=\{{\tt z}_k\}_{k=1,2,\dots}$ и удовлетворяющая неравенству  $\log \bigl|f(z)\bigr|\leq m\bigl(|z|\bigr)$ для всех $z\in \CC$ при $|z|\geq {R_0}$. Тогда  для любого числа $b\geq 0$ найдутся 
\begin{equation}\label{dCCMR}
	C:=\const_{{R_0},b}^+, \quad \overline{C}_m:=\const_{{R_0},b,m}^+,
\end{equation}
с которыми при любых $d\in \decr^+[{R_0}, +\infty;\leq b)$ справедливо неравенство
\begin{equation}\label{sumbyIm}
		\sum_{|{\tt z}_k|\geq {R_0}} \int_{|{\tt z}_k|}^{+\infty} \frac{d(t)}{t} \, d t\leq  
		\int_{{R_0}}^{+\infty} d(t)m'_{\rightd}(t) \, dt+C\overline{C}_m -C\log \bigl|f(0)\bigr|,
\end{equation}
а постоянная  	$\overline{C}_m$ положительно однородная  и  полуаддитивная сверху по $m$.
\end{corollary}

\begin{proof} Если не конечен интеграл 
\begin{equation}\label{intRdm}
	\int_{{R_0}}^{+\infty} d(t)m'_{\rightd}(t) \, dt \geq 0,
\end{equation}
то доказывать нечего. Поэтому предполагаем интеграл  \eqref{intRdm} конечным. 

Для  функции $m\colon [{R_0},+\infty)\to \RR$  с  её  продолжением  
\begin{equation*}
	m_{\leq {R_0}}(r):=\begin{cases}
m(r)\quad &\text{при $r\geq {R_0} $},\\
 m({R_0})\quad &\text{при $0\leq r< {R_0} $},	
	\end{cases}
	\quad r\in \RR^+,
\end{equation*}
на $[0,{R_0}]$ по  
Предложению \ref{prsdii}  в варианте импликации \ref{3s}$\Rightarrow$\ref{1s} Предложения \ref{prsbr} функция $M(z):=m_{\leq {R_0}}\bigl(|z|\bigr)$, $z\in \CC$, --- субгармоническая радиальная с мерой Рисса, определяемой, согласно  \eqref{df:del} и Замечанию \ref{r11}, равенствами
\begin{equation}\label{df:del-}
	\frac{1}{2\pi}\,(\Delta M) (re^{i\theta})\,d \lambda (re^{i\theta})=
	\begin{cases}
		d\bigl(rm_{\rightd}'(r)\bigr)\otimes \frac{1}{2\pi} \,d\theta\geq 0  &\text{на  $\CC \setminus \overline{D}({R_0})$},\\
		{R_0}\,m_{\rightd}'({R_0}) \, \frac{1}{2\pi} \, d\theta &\text{на окружности   $\partial D({R_0})$}, \\
		0  &\text{на круге $D({R_0})$}.
	\end{cases}
\end{equation}

 По Предложению \ref{prsds} функция $v$, определённая в \eqref{dfvdc}, принадлежит классу $\sbh_0^+\bigl(\CC\setminus D({R_0});\leq b\bigr)$. Отсюда по Теореме \ref{th:1} в варианте $D:=\CC$, $D_0:=D({R_0})$ в силу 
\eqref{in:fz0} получаем 
\begin{equation}\label{in:fz0+}
		\sum_{|{\tt z}_k|\geq {R_0}} \int_{|{\tt z}_k|}^{+\infty} \frac{d(t)}{t} \, d t	
	\leq \int_{{R_0}}^{+\infty} \Bigl(\,	\int_{r}^{+\infty} \frac{d(t)}{t} \, d t \Bigr)	\,d \bigl(rm_{\rightd}'(r)\bigr)
	+	C \,\overline{C}_m-C \,\log \bigl|f(z_0)\bigr|
\end{equation}
с постоянными  $C$ и $\overline{C}_m:=\overline{C}_M$ из  \eqref{cz0C+} вида \eqref{dCCMR} с соответствующим требуемыми свойствами.  
При конечности интеграла  \eqref{intRdm} ввиду возрастания   функции $r\mapsto rm'_{\rightd}(r)$  из соотношений
\begin{equation*}
 rm'_{\rightd}(r)  \int_{r}^{+\infty} \frac{d(t)}{t} \, d t\leq \int_{r}^{+\infty} \frac{d(t)}{t}\, t m'_{\rightd}(t) \, dt
=\int_{r}^{+\infty} d(t)m'_{\rightd}(t) \, dt \underset{r\to  +\infty}{\longrightarrow} 0
\end{equation*}
следует возможность интегрирования по частям внешнего интеграла в правой части \eqref{in:fz0+}. Это позволяет заменить двойной интеграл 
в \eqref{in:fz0+} на интеграл из  \eqref{sumbyIm} с допустимыми в рамках \eqref{dCCMR} изменениями постоянных $C$ и $\overline{C}_m$.
\end{proof}

\begin{corollary}\label{corgl}  Пусть выполнены условия Следствия\/ {\rm \ref{coref}}. Тогда  для произвольной функции   
$d\overset{\eqref{dinfty+}}{\in} \decr^+[{R_0},+\infty;<+\infty)$ 
справедлива импликация 
\begin{equation*}
\left(\,  \int_{{R_0}}^{+\infty} d(t)m'_{\rightd}(t) \, dt<+\infty\,\right)
\Longrightarrow
\left(\,	\sum\limits_{|{\tt z}_k |\geq {R_0}} \int_{|{\tt z}_k|}^{+\infty} \frac{d(t)}{t} \, dt <+\infty \,\right).
\end{equation*}
\end{corollary}

Отметим  альтернативные  варианты  Следствий \ref{coref} и \ref{corgl}. 
Эквивалентность \ref{1s}$\Leftrightarrow$\ref{2s} Предложения \ref{prsbr} через замену переменных $x=\log r$, $r\in (r_1,r_2)$ дополняет
\begin{propos}\label{prgf}  Пусть $0\leq r_1<r_2\leq +\infty$, $q\colon (\log r_1, \log r_2)\to \RR$. Суперпозиция  $q \circ \log \colon (r_1,r_2)\to \RR$  выпукла относительно $\log$ тогда и только тогда, когда $q$ --- выпуклая функция.
\end{propos}
Отсюда получаем
\begin{corollary}\label{cor22} Пусть $x_0\in \RR$ и  $q\colon [x_0,+\infty)\to \RR$ --- выпуклая возрастающая функция.  Пусть 
$f\neq 0$ --- целая функция, обращающаяся в нуль на последовательности ${\tt Z}=\{{\tt z}_k\}_{k=1,2,\dots}$ и удовлетворяющая неравенству 
 $\log \bigl|f(z)\bigr|\leq q\bigl( \log |z|\bigr)$ для всех $z\in \CC$ при $|z|\geq {e^{x_0}}$. Тогда  для любого числа $b\geq 0$ найдутся постоянные 
	$C:=\const_{{x_0},b}^+$, $\overline{C}_q:=\const_{{x_0},b,q}^+$,
с которыми  для любой выпуклой убывающей  функции 
$v\colon [x_0, +\infty)\to \RR$ с ограничениями  $\lim_{x\to +\infty}v(x)=0$ и $v(x_0)\leq b$
 справедливо неравенство
\begin{equation*}\label{sumbyImv}
		\sum_{|{\tt z}_k|\geq {e^{x_0}}} v\bigl(\log {|{\tt z}_k|} \bigr)\leq  
		-\int_{{x_0}}^{+\infty} v_{\leftd}'(x)q'_{\rightd}(x) \,dx+C\overline{C}_q -C\log \bigl|f(0)\bigr|,
\end{equation*}
а постоянная  	$\overline{C}_q$ положительно однородная  и  полуаддитивная сверху по $q$.
\end{corollary}

\begin{corollary}\label{corglAL} В условиях Следствия\/ {\rm \ref{cor22}} для любой выпуклой убывающей  функции 
$v\colon [x_0, +\infty) \to \RR$ при   $\lim_{x\to +\infty}v(x)=0$   справедлива импликация
\begin{equation*}
\left(\,	 \int_{{x_0}}^{+\infty} v_{\leftd}'(x)q'_{\rightd}(x) \, dx>-\infty\right)\Longrightarrow
	\left(\,\sum\limits_{|{\tt z}_k |\geq {e^{x_0}}} v\bigl({\log |{\tt z}_k|}\bigr) <+\infty\,\right).
\end{equation*}
\end{corollary}

\begin{remark}\label{remL}  Поскольку монотонная функция на интервале вещественной прямой почти всюду по мере Лебега на этом интервале имеет конечную производную и локально интегрируема \cite[гл.~VIII, \S~2]{Nat}, во всех утверждениях этого раздела, переходя от интегралов Римана к интегралам Лебега, односторонние производные в интегралах  можно заменить на обычные производные.
\end{remark}

\begin{remark}\label{remLA}  Любое из Следствий \ref{coref}--\ref{corglAL} содержит в себе Теорему A с условиями Адамара 
\eqref{BcnuWA}, \eqref{BcZWA} из Введения.
\end{remark}

\begin{remark}\label{remLAC+}    Инверсия   	$\star_{z_0}\colon \CC_{\infty}\setminus \{z_0\}\to \CC$ из \eqref{inv_z0}, как и её обращение,  сохраняют субгармоничность и выпуклость относительно $\log$ на соответствующих множествах и преобразует возрастание  относительно порядка 
 в возрастание относительно обратного (инверсного) порядка и наоборот. Таким образом, она без труда, через замену переменных, переносит результаты этого раздела \ref{ssCrC}    на <<проколотую>> в точке $z_0$ расширенную плоскость $\CC_{\infty}\setminus\{z_0\}$.
\end{remark}

\subsection{{Случай $D=\DD$}}\label{caseDD}   
Пусть $0<r_0<1$ всюду в этом разделе \ref{caseDD}. Некоторая модификация  Предложения \ref{prsdii} --- 

\begin{propos}\label{prsdiiD} Пусть функция $m\colon [r_0, 1)\to \RR$ --- возрастающая. 
Радиальная функция 
\begin{equation}\label{df:Mm0}
	M(z):=\begin{cases}
		m\bigl(|z|\bigr)\quad &\text{при $r_0\leq	|z|<1$},\\ 
	m\bigl(r_0\bigr)\quad &\text{при $0\leq	|z|<r_0$,}
	\end{cases}
	\quad z\in \DD,
\end{equation}
субгармоническая в $\DD$  тогда и только тогда, когда  сужение функции $m$ на $(r_0,1)$, обозначаемое также через $m$,   обладает одним (любым)  из попарно эквивалентных свойств\/ {\rm \ref{2s}--\ref{5s}} Предложения {\rm \ref{prsbr}}  при  $r_1=r_0$ и $r_2=1$ и, кроме того,   исходная функция $m$  непрерывна справа в точке $r_0$. 

При   выполнении этих условий  существуют $m_{\rightd}'(r)$ при всех $r\in [r_0,1)$, а представления  вида \eqref{mnrep}, \eqref{repM0}  могут быть  записаны в форме
\begin{equation}\label{repM0+}
	M(r)=m(r_0)+\int_{r_0}^{\max\{r_0,r\}} \frac{n(t)}{t} \, dt, \quad r\in [0,1),
\end{equation}
где возрастающую функцию $n\colon [r_0,1)\to \RR$ можно задать и как  непрерывную справа   равенствами $n(r):= r\,m_{\rightd}'(r)\geq 0$ во всех точках  $r\in [r_0,1)$. 

Плотность  меры Рисса  субгармонической функции \eqref{df:Mm0} можно  задать, по  Замечанию\/ {\rm \ref{r11}},  
равенствами 
\begin{equation}\label{df:del+}
	\frac{1}{2\pi}\,(\Delta M) (re^{i\theta})\,d \lambda (re^{i\theta})=
	\begin{cases}
		d\bigl(rm_{\rightd}'(r)\bigr)\otimes \frac{1}{2\pi} \,d\theta\geq 0  &\text{на кольце $A(r_0,1)$},\\
		r_0m_{\rightd}'(r_0) \, \frac{1}{2\pi} \, d\theta &\text{на окружности   $\partial D(r_0)$}, \\
		0  &\text{  на круге $D(r_0)$}.
	\end{cases}
\end{equation}
\end{propos}

Для   $b\in \RR^+$ введем класс $\decr^+[r_0,1;\leq b)$ {\it положительных убывающих\/} функций 
$d\colon [r_d,1)\to \RR$ с некоторым $r_d\in (0, r_0)$ и  {\it с ограничением\/} 
\begin{equation}\label{ir01d}
 \int_{r_0}^{1} \frac{d(t)}{t}\, dt \leq b .
\end{equation} 

\begin{propos}\label{dvtr}  Если $d\in \decr^+[r_0,1;\leq b)$ с $0<r_d<r_0$, то функция 
\begin{equation}\label{dfvdcd+}
	v(z):=\int_{|z|}^{1}\frac{d(t)}{t} \, dt , \quad  z\in \DD \setminus \overline{D}(r_d) ,
\end{equation}
принадлежит  классу   $\sbh_0^+\bigl(\DD\setminus D(r_0);\leq b\bigr)$, определённому в\/ \eqref{sbh+}.
\end{propos}
\begin{proof} Продолжим функцию  $d$ нулевыми значениями на весь отрезок-луч $[1,+\infty]$, сохраняя для новой функции на $[r_d,+\infty]$ то же обозначение  $d\colon [r_d,+\infty]\to [0,+\infty)$. Ввиду \eqref{ir01d}, очевидно, справедливо ограничение \eqref{dfdd} с ${R_0}=r_0$. Следовательно, 
продолженная функция $d$ из класса $\decr^+[r_0,+\infty;\leq b)$. Тогда по Предложению  \ref{prsds} функция $v$, определённая как в \eqref{dfvdc} с $R_d=r_d$ принадлежит классу $\sbh_0^+\bigl(\CC \setminus D(r_0);\leq b\bigr)$, а её сужение на  $ \DD \setminus D(r_0)$ определяется как в  
\eqref{dfvdcd+} и $\lim_{1>|z|\to 1} v(z)=0$, что и требовалось.
\end{proof}

\begin{corollary}\label{coref+}
Пусть\/ $m\colon [r_0,1)\to \RR$ --- возрастающая непрерывная функция, у которой в каждой точке $r\in [r_0,1)$ существует непрерывная справа правая производная $m_{\rightd}'(r)$ и функция $r\mapsto rm'_{\rightd}(r)$ --- возрастающая. Пусть $f\in \Hol (\DD)$ --- ненулевая  функция, обращающаяся в нуль на последовательности ${\tt Z}=\{{\tt z}_k\}_{k=1,2,\dots}$,  и  неравенство  $\log \bigl|f(z)\bigr|\leq m\bigl(|z|\bigr)$ 
выполнено для всех $z\in \DD$ при $|z|\geq r_0$. Тогда  для любого числа $b\geq 0$ найдутся  постоянные 
\begin{equation}\label{dCCMR+}
	C:=\const_{r_0,b}^+, \quad \overline{C}_m:=\const_{r_0,b,m}^+,
\end{equation}
с которыми для  любой функции $d\in \decr^+[r_0, 1;\leq b)$ справедливо неравенство
\begin{equation}\label{sumbyIm+}
		\sum_{ r_0 \leq	|{\tt z}_k|<1} \int_{|{\tt z}_k|}^{1} \frac{d(t)}{t} \, d t\leq  
		\int_{r_0}^{1} d(t)m'_{\rightd}(t) \, dt+C\overline{C}_m -C\log \bigl|f(0)\bigr|,
\end{equation}
а постоянная  	$\overline{C}_m$ положительно однородная  и  полуаддитивная сверху по $m$.
\end{corollary}
\begin{proof}
Если \underline{не} конечен интеграл 
\begin{equation}\label{intRdm+}
	\int_{r_0}^{1} d(t)m'_{\rightd}(t) \, dt\geq 0,
\end{equation}
то доказывать нечего. Поэтому предполагаем интеграл  \eqref{intRdm+} конечным. 

Функция $M$, определённая в Предложении \ref{prsdiiD} в виде \eqref{df:Mm0} субгармоническая радиальная на $\DD$ с мерой Рисса \eqref{df:del+}. 

 По Предложению \ref{dvtr} функция $v$, определённая в \eqref{dfvdcd+}, принадлежит классу $\sbh_0^+\bigl(\DD\setminus D(r_0);\leq b\bigr)$. Отсюда по Теореме \ref{th:1} в варианте $D:=\DD$, $D_0:=D(r_0)$ в силу 
\eqref{in:fz0} получаем 
\begin{equation}\label{in:fz0++}
		\sum_{|{\tt z}_k|\geq r_0} \int_{|{\tt z}_k|}^{1} \frac{d(t)}{t} \, d t	
	\leq \int_{r_0}^{1} \Bigl(\,	\int_{r}^{1} \frac{d(t)}{t} \, d t \Bigr)	\,d \bigl(rm_{\rightd}'(r)\bigr)
	+	C \,\overline{C}_m-C \,\log \bigl|f(z_0)\bigr|
\end{equation}
с постоянными  $C$ и $\overline{C}_m:=\overline{C}_M$ из  \eqref{cz0C+} вида \eqref{dCCMR}.  
При конечности интеграла  \eqref{intRdm+} ввиду возрастания   функции $r\mapsto rm'_{\rightd}(r)$  из соотношений
\begin{equation*}
 rm'_{\rightd}(r)  \int_{r}^{1} \frac{d(t)}{t} \, d t\leq \int_{r}^{1} \frac{d(t)}{t}\, t m'_{\rightd}(t) \, dt
=\int_{r}^{1} d(t)m'_{\rightd}(t) \, dt \underset{1>r\to  1}{\longrightarrow} 0
\end{equation*}
следует возможность интегрирования по частям внешнего интеграла в правой части \eqref{in:fz0+}. Это позволяет заменить двойной интеграл 
в \eqref{in:fz0++} на интеграл из  \eqref{sumbyIm+} с допустимыми в рамках \eqref{dCCMR+} изменениями постоянных $C$ и $\overline{C}_m$.
\end{proof}
\begin{corollary}\label{corgl+} В условиях Следствия\/ {\rm \ref{coref+}} для любой положительной убывающей  функции 
$d$ в открытой окрестности  интервала $[r_0,1)\subset (0,+\infty)$ справедлива импликация 
\begin{equation*}
\left(\,	  \int_{r_0}^{1} d(t)m'_{\leftd}(t) \, dt<+\infty \, \right) \Longrightarrow
	\left(\,\sum\limits_{r_0 \leq	|{\tt z}_k|<1} \int_{|{\tt z}_k|}^{1} \frac{d(t)}{t} \, dt <+\infty \, \right).
\end{equation*}
\end{corollary}

Отметим  также и  альтернативные  варианты  Следствий \ref{coref+} и \ref{corgl+}, основанные на эквивалентности \ref{1s}$\Leftrightarrow$\ref{2s} Предложения \ref{prsbr}, замене  переменных $x=\log r$ при $r\in (r_0,1)$, $x\in  (\log r_0,0)$ и Предложении \ref{prgf}. 

\begin{corollary}\label{cor22+} Пусть $0>x_0\in \RR$ и  $q\colon [x_0,0)\to \RR$ --- выпуклая возрастающая функция.  Пусть 
$0\neq	 f\in \Hol (\DD)$, $f$ обращается  в нуль на ${\tt Z}=\{{\tt z}_k\}_{k=1,2,\dots}$ и удовлетворяет неравенству 
 $\log \bigl|f(z)\bigr|\leq q\bigl( \log |z|\bigr)$ для всех $z\in \DD$ при ${e^{x_0}} \leq |z|<1 $. Тогда  для любого числа $b\geq 0$ найдутся постоянные  	$C:=\const_{{x_0},b}^+$, $\overline{C}_q:=\const_{{x_0},b,q}^+$,

с которыми  для любой выпуклой убывающей  функции 
$v\colon [x_0, 0)\to \RR$ с ограничениями  $\lim\limits_{0>x\to 0}v(x)=0$ и $v(x_0)\leq b$
 справедливо неравенство
\begin{equation*}
		\sum_{ {e^{x_0}} \leq |{\tt z}_k|<1} v\bigl(\log {|{\tt z}_k|} \bigr)\leq  
		-\int_{{x_0}}^{1} v_{\leftd}'(x)q'_{\rightd}(x) \,dx+C\overline{C}_q -C\log \bigl|f(0)\bigr|,
\end{equation*}
а постоянная  	$\overline{C}_q$ положительно однородная  и  полуаддитивная сверху по $q$.
\end{corollary}

\begin{corollary}\label{corglAL+} В условиях Следствия\/ {\rm \ref{cor22+}} для любой выпуклой убывающей  функции 
$v\colon [x_0, 0) \to \RR$ с   $\lim_{0>x\to 0}v(x)=0$  истинна импликация
\begin{equation*}
\left(\,	  \int_{{x_0}}^{0} v_{\leftd}'(x)q'_{\rightd}(t) \, dt>-\infty\,\right)\Longrightarrow
	\left(\,\sum\limits_{ {e^{x_0}} \leq |{\tt z}_k|<1} v\bigl({\log |{\tt z}_k|}\bigr) <+\infty\,\right).
\end{equation*}
\end{corollary}

\begin{remark}\label{remL+}  Совпадает с Замечанием \ref{remL}.
\end{remark}

\begin{remark}\label{remLA+}  Любое из Следствий \ref{coref+}--\ref{corglAL+} содержит в себе Теорему B с условиями Бляшке \eqref{Bcnu}, \eqref{BcZ} из Введения.
\end{remark}

\begin{remark}\label{remLA++}    Гомотетия с центром в нуле и инверсия $\star$ из Введения сохраняют субгармоничность и выпуклость относительно $\log$. Эти отображения для любого $R\in (0,+\infty)$ без труда переносят результаты этого раздела \ref{caseDD}   на каждый  из случаев  $D=D(R)$ и  $\CC_{\infty}\setminus D(R)$.
\end{remark}

\subsection{{Случай $D=A(r_1,r_2)$ --- кольцо}}\label{Sann}  Всюду в этом разделе 
\begin{equation}\label{dff:r12}
	0\leq	r_1<r'<r'' <r_2\leq +\infty .
\end{equation}
Утверждения  этого раздела легко получаются  простой интеграцией  результатов разделов  
\ref{ssCrC}--\ref{caseDD}  с использованием, когда необходимо, замен переменных,  в частности инверсии. Поэтому доказательства и обоснования опускаем.

\begin{propos}\label{prann} Пусть функция $m\colon (r_1,r']\cup [r'',r_2)\to \RR$ 
\begin{enumerate}[{\rm 1)}] 
	\item 	убывающая на $(r_1,r']$, обладает одним (любым)  из попарно эквивалентных свойств\/ {\rm \ref{2s}--\ref{5s}} Предложения {\rm \ref{prsbr}}  при  $r_2=r'$, непрерывна слева в точке $r'$,
	\item возрастающая  на $[r'', r_2)$, обладает одним (любым)  из попарно эквивалентных свойств\/ {\rm \ref{2s}--\ref{5s}} Предложения {\rm \ref{prsbr}}  c  $r_1=r''$, непрерывна справа в  $r''$. 
\end{enumerate}
Предположим, что существует число\footnote{Число $m_0$, удовлетворяющее условию \eqref{m0m},  всегда существует, если $\lim\limits_{r_1<r\to r_1}m(r)=\lim\limits_{r_2>r\to r_2}m(r)=+\infty$. В противном случае этого можно добиться,   добавив число к одной из <<ветвей>> функции $m$ на $(r_1,r']$ или на $[r'',r_2)$. Здесь  такое изменение функции $m$, р\'авно как и уменьшение $r'$ и увеличение $r''$ в рамках условия \eqref{dff:r12}, не играет никакой  роли.}
\begin{equation}\label{m0m}
	m_0\in m\bigl( (r_1,r'] \bigr)\cap m\bigl( [r'',r_2) \bigr)\neq \varnothing.
\end{equation}
Тогда, уменьшая  $r'$ или/и увеличивая $r''$ в рамках условия \eqref{dff:r12}, легко  добиться того, что $m(r')=m(r'')=m_0$,  для исходной функции $m$, доопределённой  значениями $m$ на $(r',r'')$,   существуют $m_{\leftd}'(r)$ и $ m_{\rightd}'(r)$ при всех $r\in (r_1,r_2)$. При этом  радиальная функция
\begin{equation}\label{forannM}
	M(z):=\begin{cases}
		m\bigl(|z|\bigr) \quad&\text{при $|z|\in (r_1,r']\cup [r'',r_2)$}, \\
		m_0\quad&\text{при $r' <|z|<r'$},
	\end{cases}
	\quad x\in A(r_1,r_2),
\end{equation}
--- субгармоническая, представление   \eqref{mnrep} для любого $r_0\in (r',r'')$  
имеет вид 
\begin{equation*}
	M(r)=m_0+\int_{r_0}^r \frac{n_0(t)}{t}\, dt, \quad  r\in (r_1,r_2),
\end{equation*}
где возрастающую функцию $n_0\colon (r_1,r_2)\to \RR$ можно задать и как  непрерывную справа (или слева)  
равенствами $n_0(r):= r\,m_{\rightd}'(r)$  {\rm \large(}соответственно $n_0(r):= r\,m_{\rightd}'(r)${\rm \large)} 
во всех точках  $r\in (r_1,r_2)$. 

Плотность  меры Рисса  субгармонической функции \eqref{forannM} можно задать, по  Замечанию\/ {\rm \ref{r11}},  
равенствами 
\begin{equation*}\label{df:del++}
	\frac{1}{2\pi}\,(\Delta M) (re^{i\theta})\,d \lambda (re^{i\theta})=
	\begin{cases}
		d\bigl(rm_{\leftd}'(r)\bigr)\otimes \frac{1}{2\pi} \,d\theta\geq 0  &\text{на кольце $A(r_1,r')$},\\
		r'm_{\leftd}'(r') \, \frac{1}{2\pi} \, d\theta &\text{на окружности   $\partial D(r')$}, \\
		0  &\text{на кольце  $A(r',r'')$},\\
		r''m_{\rightd}'(r'') \, \frac{1}{2\pi} \, d\theta &\text{на окружности   $\partial D(r'')$}, \\
		d\bigl(rm_{\rightd}'(r)\bigr)\otimes \frac{1}{2\pi} \,d\theta\geq 0  &\text{на кольце $A(r'',r_2)$}.
	\end{cases}
\end{equation*}
\end{propos}

Для   пары чисел $b_1,b_2\in \RR^+$ введем класс $\monot^+(r_1,r',r'',r_2;\leq b)$ {\it положительных\/} функций 
$d\colon (r_1,r_d']\cup [r_d'',r_2)\to \RR$ с  $r'<r_d'\leq r_d''<r''$, {\it возрастающих на\/}  $(r_1,r_d']$ и {\it убывающих на\/} $ [r_d'',r_2)$,  {\it с ограничениями\/} 
\begin{equation*}\label{ir01d+}
 \int_{r_1}^{r'} \frac{d(t)}{t}\, dt \leq b_1 , \quad \int_{r''}^{r_2} \frac{d(t)}{t}\, dt \leq b_2.
\end{equation*} 

\begin{propos}\label{dvtr+}  Если $d\in \monot^+(r_1,r',r'',r_2;\leq b)$, то функция 
\begin{equation}\label{dfvdcd++}
	v(z):=
	\begin{cases}
		\int\limits_{r_1}^{|z|}d(t) \, \dfrac{dt}{t} \quad&\text{при  $r_1<|z|<r_d'$},\\
		\\
		\int\limits_{|z|}^{r_2}d(t) \, \dfrac{dt}{t} \quad&\text{при  $r_d''<|z|<r_2$},\\
	\end{cases}
	\qquad z\in A(r_1,r_2)\setminus A(r',r''), 
\end{equation}
принадлежит  классу   $\sbh_0^+\bigl(A(r_1,r_2)\setminus A(r',r'');\leq b\bigr)$, определённому в\/ \eqref{sbh+}.
\end{propos}
\begin{corollary}\label{corAnn} Пусть $m$ --- функция из Предложения {\rm \ref{prann}}, а ненулевая функция   $f\in \Hol \bigl(A(r_1,r_2)\bigr)$ обращается в нуль на последовательности точек ${\tt Z}=\{{\tt z}_k\}_{k=1,2,\dots}$ и  $\log \bigl|f(z)\bigr|\leq m\bigl(|z|\bigr)$ 
 при всех $z\in  A(r_1,r_2)\setminus A(r',r'')$. Тогда  для любых  пар чисел $b_1,b_2\geq 0$  и числа $z_0\in A(r',r'')$ найдутся  постоянные\footnote{Зависимость этих постоянных от величин из \eqref{dff:r12} не помечена.} 
\begin{equation}\label{dCCMR++}
	C:=\const_{z_0,b_1,b_2}^+, \quad \overline{C}_m:=\const_{z_0,b_1,b_2,m}^+,
\end{equation}
с которыми для  любой функции $d\in \monot^+(r_1,r',r'',r_2;\leq b)$ имеем
\begin{multline*}
				\left(\,\sum_{ r_1 <	|{\tt z}_k|\leq r'} \int_{r_1}^{|{\tt z}_k|} + 
				\sum_{ r'' \leq	|{\tt z}_k|<r_2} \int_{|{\tt z}_k|}^{r_2} \,\right)\frac{d(t)}{t} \, d t\\
			\leq  		-\int_{r_1}^{r'} d(t)m'_{\leftd}(t) \, dt
			+\int_{r''}^{r_2} d(t)m'_{\rightd}(t) \, dt 			+C\overline{C}_m -C\log \bigl|f(z_0)\bigr|,
	\end{multline*}
а постоянная  	$\overline{C}_m$ положительно однородная  и  полуаддитивная сверху по $m$.
\end{corollary}

\begin{corollary}\label{corgl++} В условиях Следствия\/ {\rm \ref{corAnn}} для  любого $b\in \RR^+$ и  любой функции 
$d\in \monot^+(r_1,r',r'',r_2;\leq b)$  справедлива импликация 
\begin{multline*}
\left(\,	\int_{r_1}^{r'} d(t)m'_{\leftd}(t) \, dt +\int_{r''}^{r_2} d(t)m'_{\rightd}(t) \, dt <+\infty	\, \right) 
\\
\Longrightarrow
	\left(\,\sum_{ r_1 <	|{\tt z}_k|\leq r'} \int_{r_1}^{|{\tt z}_k|} \frac{d(t)}{t} \, d t+
			\sum_{ r'' \leq	|{\tt z}_k|<r_2} \int_{|{\tt z}_k|}^{r_2} \frac{d(t)}{t} \, d t  <+\infty \, \right).
\end{multline*}
\end{corollary}

Отметим  также и  альтернативные  варианты  Следствий \ref{corAnn} и \ref{corgl++}, основанные на эквивалентности \ref{1s}$\Leftrightarrow$\ref{2s} Предложения \ref{prsbr}, замене  переменных $x=\log r$ при $r\in (r_1,r']\cup [r'',r_2)$ и $x\in  (x_1,x']\cup [x'',x_2)$ с 
\begin{equation}\label{x12}
	-\infty \leq \,x_1:=\log r_1\, <\, x':=\log r' \, <\, x'':=\log r'' \, <\, x_2:=\log r_2 \, \leq +\infty 
\end{equation}
и Предложении \ref{prgf}. В обозначениях  \eqref{x12} имеет место

\begin{corollary}\label{cor22++} Пусть   $q\colon (x_1,x'] \cup [x'', x_2) \to \RR$ --- выпуклая функция, убывающая на $(x_1,x']$ и возрастающая  на $[x'', x_2)$,  $q\bigl((x_1,x']\bigr) \cap q\bigl([x'', x_2)\bigr)\neq \varnothing$.  Пусть ненулевая функция   $f\in \Hol \bigl(A(e^{x_1},e^{x_2})\bigr)$   обращается  в нуль на ${\tt Z}=\{{\tt z}_k\}_{k=1,2,\dots}$ и   $\log \bigl|f(z)\bigr|\leq q\bigl( \log |z|\bigr)$ при  всех $z\in A(e^{x_1}, e^{x_2})\setminus 
A(e^{x'}, e^{x''})$. Тогда  для любой пары чисел $b_1,b_2\geq 0$ и числа $z_0\in A(e^{x'}, e^{x''})$ найдутся постоянные 
\eqref{dCCMR++} с $q$ и  $\overline{C}_q$ вместо $m$ и $\overline{C}_m$ , зависящие  от чисел \eqref{x12}, с которыми  для любой выпуклой   функции  $v\colon (x_1,x'] \cup [x'', x_2) \to \RR$, возрастающей на $(x_1,x']$ и убывающей на $[x'', x_2)$ с ограничениями  $\lim_{x_1< x\to x_1}v(x)=\lim_{x_2>x\to x_2}v(x) =0$ и $v(x')\leq b_1$, $v(x'')\leq b_2$
 справедливо неравенство
\begin{multline*}
				\left(\, \sum_{ e^{x_1} <	|{\tt z}_k|\leq e^{x'}} +
			\sum_{e^{x''} \leq	|{\tt z}_k|<e^{x_2}} \right)
			v\bigl(\log|{\tt z}_k|\bigr)
			\\
			\leq  	
				-\int_{{x_1}}^{x'} v_{\rightd}'(x)q'_{\leftd}(x) \,dx 				-\int_{x''}^{x_2} v_{\leftd}'(x)q'_{\rightd}(x) \,dx 						+C\overline{C}_m -C\log \bigl|f(z_0)\bigr|,
	\end{multline*}
а постоянная  	$\overline{C}_q$ положительно однородная  и  полуаддитивная сверху по $q$.
\end{corollary}

\begin{corollary}\label{corglAL++} В условиях Следствия\/ {\rm \ref{cor22++}} для любой выпуклой   функции  $v\colon (x_1,x'] \cup [x'', x_2) \to \RR$, возрастающей на $(x_1,x']$ и убывающей на $[x'', x_2)$ с ограничениями  $\lim_{x_1< x\to x_1}v(x)=\lim_{x_2>x\to x_2}v(x) =0$ и $v(x')\leq b_1$, $v(x'')\leq b_2$  истинна импликация
\begin{multline*}
\left(\,	 	\int_{{x_1}}^{x'} v_{\rightd}'(x)q'_{\leftd}(x) \,dx 				+\int_{x''}^{x_2} v_{\leftd}'(x)q'_{\rightd}(x) \,dx >-\infty\,\right)
\\
\Longrightarrow
	\left(\,\Biggl(\, \sum_{ e^{x_1} <	|{\tt z}_k|\leq e^{x'}} + 			\sum_{e^{x''} \leq	|{\tt z}_k|<e^{x_2}} \,\Biggr) 	v\bigl(\log|{\tt z}_k|\bigr) <+\infty\,\right).
\end{multline*}
\end{corollary}

\begin{remark}\label{remL+++}  Совпадает с Замечанием \ref{remL}.
\end{remark}
\begin{remark} Основная  Теорема, Теорема  \ref{th:1} и Следствие \ref{cor:1} позволяют использовать радиальные тестовые функции 
\eqref{dfvdc}, \eqref{dfvdcd+},  \eqref{dfvdcd++} соответственно для плоскости $\CC$, круга $\DD$ и кольца, а также их альтернативные версии из Следствий  \ref{cor22}, \ref{cor22+}, \ref{cor22++} вида $v\circ \log$ с выпуклой функцией $v$  не только для радиальных мажорант $M$, но и для произвольных ($\delta$-)субгармонических мажорант $M$ в таких областях. 
\end{remark}

\section{Тестовые функции}\label{S6tf}

Уже по Теореме   B из п.~\ref{111} Введения видно, что как самую элементарную тестовую функцию можно рассматривать функцию Грина 
$g_{D'} (\cdot, z_0)$ для произвольной подобласти $D'\subset D$. 
Аналогично,  в исследованиях Коренблюма--Сейпа распределения нулей 
в равномерных пространствах Бергмана в круге $\DD$ роль тестовых функций 
играли    функции Грина (с полюсом в нуле) специальных конечных объединений открытых кругов из $\DD$, касающихся окружности $\partial \DD$ изнутри в конечном числе точек и содержащих в себе некоторый фиксированный круг 
$D(r_0)\Subset \DD$ {\large(}см. \cite[4]{HKZh}, \cite[\S~5]{KuKh09}{\large)}.  

Здесь мы укажем способы конструирования широких классов тестовых функций, далёких от радиальных.  
Очень полезной для этого  будет установленная Дж.~Гардинером с дополнениями М.~Климека 
{\large(}см. также Ю.~Риихентаус \cite{JR}{\large)}
\vskip 1mm
	
{\sc Теорема G--K} {\cite[Theorem 2.6.6]{Klimek}}. {\it 
Пусть $\mathcal O\subset \CC_{\infty}$ --- открытое множество,  $g\colon {\mathcal O}\to (0,+\infty)$.  
Если выполнено одно из трёх условий
\begin{enumerate}[{\rm(i)}]
	\item\label{ih} $s\in \Har ({\mathcal O}), g\in \Har({\mathcal O})$, $f\colon \RR\to\RR$ --- выпуклая функция;      
	\item\label{iih} $s\in \sbh ({\mathcal O}), g\in \Har({\mathcal O})$, $f\colon \RR\to\RR$ --- выпуклая возрастающая функция с интерпретацией $f(-\infty):=\lim\limits_{x\to -\infty}f(x)$; 
	\item\label{iiih} $s\in \sbh^+({\mathcal O})$, $-g\in \sbh({\mathcal O})$, $f\colon \RR^+\to \RR^+$ --- выпуклая функция с $f(0)=0$; 
\end{enumerate}
 то $g\,f(s/g)\in \sbh ({\mathcal O})$.}

\noindent
Ниже в подразделе \ref{OLsf} будет дано новое доказательство этого результата --- более сложное, чем оригинальное, но позволяющее явно выписывать меры Рисса субгармонических функций   $g\,f(s/g)\in \sbh ({\mathcal O})$ из заключения  Теоремы G--K.   

Отметим также очень давно известные  частные случаи Теоремы G--K, соответствующие случаю $g=1$ --- тождественная единица.
\vskip 1mm
	
{\sc Теорема C} {\cite[I.II.9]{Doob}}. {\it 
Пусть $\mathcal O\subset \CC_{\infty}$  открытое, $-\infty\leq a<b\leq +\infty$, $F\colon (a,b)\to \RR$ --- выпуклая функция. Если выполнено одно из двух условий
\begin{enumerate}[{\rm(i)}]
	\item\label{ihF} функция $s\colon \mathcal O \to (a,b)$ из пространства $\Har ({\mathcal O})$;      
	\item\label{iihF} $s\colon \mathcal O \to [a,b)$, $s\in \sbh ({\mathcal O})$ и  $F$ --- возрастающая функция, доопределенная  в точке $a$ по непрерывности $F(a):=\lim\limits_{a<x\to a}F(x)$; 
	\end{enumerate}
то суперпозиция $F\circ s \in \sbh (\mathcal O)$.} 
\subsection{Построение тестовых функций}\label{sstf}  Пользуемся различными версиями Теоремы G--K. При этом, как правило, роль функций $g$, а часто и $s$, будут играть именно функции Грина  с одним и тем же фиксированным полюсом $z_0\in D_0\Subset D$ различных областей, включающих в себя $D_0$. Поэтому  точку $z_0$ как полюс в обозначениях функций Грина зачастую  не указываем. 

\subsubsection{\underline{На основе гармонических функций}}\label{tfi} 
\begin{propos}\label{prgg}
Пусть  $D$ --- регулярная область с функцией Грина $g_D$, $s\in \Har \bigl(D\setminus D_0)$, 
а $f\colon \RR \to \RR^+$ ---   выпуклая функция, ограниченная на образе   $ (s/g_D) (D\setminus D_0)$. Тогда  $g_D\, f(s/g_D)\overset{\eqref{mint}}\in \sbh_0^+\bigl(D\setminus D_0\bigr)$ --- тестовая функция.  В качестве функции $s$ можно выбирать, например, $s=1$ или функцию Грина $g_{\widehat{D}}$  произвольной  области $\widehat{D}$ с неполярной границей, включающей в себя область $D$.
В частности, в последнем  случае по принципу подчинения \cite[Corollary 4.4.5]{Rans} для функций Грина (с одинаковыми полюсами) при $D\subset \widehat{D}$ всегда 
\begin{equation}\label{est11}
\frac{g_{\widehat{D}}}{g_D}\geq 1 \quad \text{на $D$}.
\end{equation}
\end{propos}  
\begin{proof} По п.~\eqref{ih} Теоремы G--K и в силу положительности $f$ имеем $g_D\, f(s/g_D)\in \sbh^+\bigl(D\setminus D_0\bigr)$. При этом ввиду ограниченности функции $f$ на образе $(s/g_D)(D\setminus D_0)$ и стремления к нулю функции Грина $g_D$ при приближении к границе $\partial D$ видим, что $g_D\, f(s/g_D)$ --- тестовая функция.  
\end{proof}
\begin{example}\label{exg} Рассмотрим регулярную область $\widehat{D}\supset D$, число $p>0$, 
выпуклую убывающую положительную на $\RR$ функцию 
\begin{equation*}
	f_p(x):=\begin{cases}
		x^{-p} &\text{ при $x\geq 1$},\\
	-p(x-1)+1 &\text{ при $x\leq   1$},
	\end{cases}
	\end{equation*}
    а также выпуклую убывающую функцию $x\mapsto e^{-px}$, $x\in \RR$. 
Тогда по Предложению \ref{prgg}  согласно \eqref{est11} две функции 
\begin{equation}\label{gdgex}
	g_D\left(\frac{g_{\widehat{D}}}{g_D}\right)^{-p}=
	g_D \cdot g_D^p \cdot g_{\widehat{D}}^{-p} \overset{\eqref{est11}}{\leq} g_D, \quad 
	g_D\exp\Bigl(-p(g_{\widehat{D}} / g_{D})\Bigr)\leq g_D
\end{equation}
--- примеры  нетривиальных тестовых функций  класса 
$\sbh_0^+\bigl(D\setminus \{z_0\}\bigr)$.
 
\end{example}
\subsubsection{\underline{На основе гармонической и субгармонической функций}}\label{tfii}
\begin{propos}\label{prggs}
Пусть функция  $g\in \Har (D\setminus D_0)$ строго положительна и ограничена на $D\setminus D_0$, $f\colon \RR \to \RR^+$ --- выпуклая возрастающая функция  с   интерпретацией $f(-\infty):=\lim\limits_{x\to -\infty}f(x)$.
Пусть   $s\in \sbh (D\setminus D_0)$ и выполнено одно из условий:
	
	\begin{enumerate}[{\rm (a)}]
		\item\label{ia} для некоторого $x_0\in [-\infty,+\infty)$  имеем $f(x_0)=0$ и  существует предел
	\begin{equation}\label{sgff}
	\lim_{z\to \partial D} \frac{s(z)}{g(z)}\overset{\eqref{est:u0}}{=}x_0 
\end{equation}
\item\label{ib} образ $(s/g)(D\setminus D_0)$ ограничен и $\lim\limits_{{z\to \partial D}} g(z)=0$.
\end{enumerate}
 Тогда $g\,f(s/g)\in \sbh_0^+(D\setminus D_0)$.
\end{propos}
Доказательство опускаем, поскольку оно легко следует из п.~\eqref{iih}
Теоремы G--K по аналогии с доказательством Предложения \ref{prgg}  
с той лишь разницей, что стремление к нулю для $g\,f(s/g)$ при приближении границы $\partial D$ достигается в случае \eqref{ia} за счёт условия 
\eqref{sgff} в сочетании с  $f(x_0)=0$, а в случае \eqref{ib} за счет стремления $g(z)$ к нулю при $z\to \partial D$ и ограниченности образа
$\bigl(f(s/g)\bigr)(D\setminus D_0)$.
\begin{example}[\rm(скорее частный случай Предложения \ref{prggs})]\label{ex22}  Пусть $\widehat{D}$ --- область и 
$D\Subset \widehat{D}$, $g\in \Har \bigl(\widehat{D}\setminus D_0\bigr)$ и $g>0$ на $D\setminus D_0$, к примеру, $g=g_{\widehat{D}}$ --- функция Грина области $\widehat{D}$. В условиях Предложения \ref{prggs}\eqref{ia} замена условия \eqref{sgff} на $f(0)=0$ и $\lim\limits_{z\to \partial D}s(z)=0$   даёт тот же результат.

Другой вариант --- выбор в п.~\eqref{ib}  в качестве функции $s$ продолженной функции Грина $g_{D'}(\cdot, z_0)$ для подобласти  $D'\subset D$ с неполярной границей, а в роли $g$ --- функции Грина $g_D(\cdot, z_0)$. В этом случае  $0\leq g_{D'}/g_D \leq 1$ на $D\setminus \{z_0\}$.
\end{example}
\subsubsection{\underline{На основе супер- и субгармонической функций}}\label{tfiii}
\begin{propos}\label{prggscr}
Пусть  функция $g>0$ на $D\setminus D_0$, $g\in -\sbh (D\setminus D_0)$, т.\,е. $g$ супергармоническая,   $s\in \sbh^+ (D\setminus D_0)$, $f\colon \RR^+ \to \RR^+$ --- выпуклая возрастающая функция  с  условием $f(0)=0$ и выполнено одно из трех условий
\begin{enumerate}[{\rm (a)}]
\item\label{ha} $\lim\limits_{z\to \partial D} {g(z)}\overset{\eqref{est:u0}}{=}0$ и образ  $ (s/g)(D\setminus D_0)$ ограничен;
\item\label{hb} выполнено соотношение \eqref{sgff} и функция $g$ ограничена на $D\setminus D_0$;
\item\label{hc} $s\in \sbh_0^+(D\setminus D_0)$, $\liminf\limits_{z\to \partial D} g(z)>0$ и $g$ ограничена сверху на $D\setminus D_0$.  
\end{enumerate}
 Тогда $g\,f(s/g)\in \sbh_0^+(D\setminus D_0)$ --- тестовая функция.
\end{propos}
Доказательство также опускаем, поскольку рассуждения для варианта \eqref{ha} идентичны доказательству Предложения \ref{prgg}, а для варианта \eqref{hb} --- доказательству  Предложения \ref{prggs}.
Вариант же \eqref{hc} --- частный случай варианта \eqref{hb}.
\begin{example}\label{exgDs} Пусть  подобласть $D'\subset D$ регулярная, $z_0\in D_0\Subset D$,  $s:=g_{D'}:=g_{D'}(z,z_0)$ --- продолженная нулём на $\CC_{\infty} \setminus D'$ функция Грина для $D'$ с полюсом $z_0$, $g=1$. Тогда для  $f$ из  Предложения  \ref{prggscr}\eqref{hc} функция $f\circ g_{D'}$ тестовая из  класса  $\sbh_0^+(D\setminus D_0; \leq b)$, где 
$b:=\sup\limits_{\partial D_0}\,(f\circ g_{D'})$.
\end{example}

\begin{example}\label{exp1gd} Пусть $D$ --- регулярная область в $\CC_{\infty}$ с функцией Грина $g_D=g_D(\cdot,z_0)$. Тогда  
для положительной функции $g:=\log (1+g_D)$, очевидно, выполнено условие \eqref{ha} Предложения  \ref{prggscr} и вне точки $z_0\in D_0$ 
функция $-\log (1+g_D) \in \sbh\bigl(D\setminus \{z_0\}\bigr)$ по Предложению C\eqref{ihF} как суперпозиция выпуклой функции $-\log$ и гармонической в $D\setminus \{z_0\}$ функции $1+g_D$, 
т.\,е. \textit{для любой функции Грина\/ $g_D (\cdot,z_0)$ функция\/ $\log (1+g_D)$ супергармонична}. Следовательно,  из Предложения \ref{prggscr} {\it для любой функции $s\in \sbh^+(D\setminus D_0)$ при ограниченности  образа $\bigl(s/\log (1+g_D)\bigr)(D\setminus D_0)$ для выпуклой возрастающей  функции  $f\colon \RR^+\to \RR^+$  с $f(0)=0$ функция 
\begin{equation}\label{loggdt}
	 f\Bigl(\frac{s}{\log (1+g_D)}\Bigr)\log(1+g_D)   \quad\text{--- тестовая функция из $\sbh_0^+(D\setminus D_0)$}. 
\end{equation}
}
 \end{example}

\subsubsection{\underline{Гиперболический  радиус и тестовые функции}}\label{excr} Мо\-т\-и\-в\-и\-р\-о\-вано консу\-льтациями  Ф.\,Г. Авхадиева \cite[\textbf{3.1}]{AvW}, \cite{Av15}.  Пусть $D
\subset \CC_{\infty}$---  область гиперболического типа, т. е. область, имеющая не менее трёх граничных точек на $\CC_{\infty}$. Тогда в любой точке $z\in D$ определен гиперболический, или конформный в случае односвязности $D\not\ni \infty $,  радиус $R(z,D) = \dfrac{1}{\lambda_D(z)}$, $z\in D$,  где $\lambda_D$ – коэффициент метрики Пуанкаре с гауссовой кривизной $=-4$. 
 Известно, что гиперболический радиус $R(\cdot,D)$ строго положителен на $D$ и в случае, когда $\infty \notin \partial D$, обладает свойством 
$\lim\limits_{z\to \partial D}R(z,D)\overset{\eqref{ha}}{=}0$. 

Далее, как обычно, $\nabla$ --- оператор <<набла>>, или Гамильтона, реализирующий градиент; $\Delta=\nabla\cdot \nabla =\nabla^2$ --- оператор Лапласа. 

Гиперболический радиус удовлетворяет уравнению Лиувилля
\begin{equation}\label{RdR}
	R\Delta R=|\nabla R|^2-4, \quad \text{$\nabla R$ --- градиент $R$},
\end{equation}
которое ввиду легко устанавливаемого тождества $\Delta (R^2)=2R\Delta R+2|\nabla R|^2$ {\large(}см. неоднократно употребляемое ниже тождество \eqref{{seq}s}{\large)} может быть записано и  в иных эквивалентных \eqref{RdR} формах:
  \begin{equation*}
	4R\Delta R=\Delta (R^2) -8 \quad \Longleftrightarrow \quad  \Delta(R^2)=4|\nabla R|^2-8.
 \end{equation*}

По Теореме Лёвнера \cite[Theorem 3.23]{AvW} собственная подобласть в $\CC$ выпукла тогда и только тогда, когда 
\begin{equation*}
	\sup_{z\in D} |\nabla R(z,D)|\leq 2.
\end{equation*}
Таким образом, ввиду \eqref{RdR} и по  Предложению  \ref{prggscr}  имеем
\begin{example}\label{ex:Rconv}
Для собственной \underline{выпуклой} подобласти  $D\subset \CC$ с\/ $\infty \notin \partial D$ её {\it конформный радиус $R$ супергармоничен\/} и {\it для произвольной   $s\in \sbh^+( D\setminus D_0)$ при ограниченности образа $ (s/R) (D\setminus D_0)$ с выпуклой возрастающей функцией $f\colon \RR^+\to \RR^+$ с $f(0)=0$   функция $Rf(s/R)$ тестовая.}
\end{example}
Другой подобный \eqref{loggdt} пример  с участием гиперболического радиуса $R$ области $D$ возможен, когда a priori выполнена верхняя оценка  
\begin{equation}\label{hypR}
	\sup_{z\in D} (\Delta R)(z)\leq A<+\infty \quad \text{и \; $A\leq 4N\in \RR$}.
\end{equation}
По сообщению  Ф.\,Г. Авхадиева  достаточное условие для этого --- граница $\partial D$ класса $C^2$ с гельдеровой кривизной, но можно допускать и угловые точки с нетупым  углом, т.\,е.  локально выпуклые в окрестности угловой точки. 

\begin{example}\label{exp1gdR} Пусть выполнено условие \eqref{hypR}. Тогда
\begin{multline}\label{dhff}
\Delta \log (1+NR)=\nabla\cdot \nabla \log(1+NR)=\nabla \cdot \frac{\nabla (1+NR)}{1+NR}\\=
\frac{(1+NR)\nabla\cdot \nabla (1+NR)- \nabla (1+NR)\cdot \nabla (1+NR)}{(1+NR)^2}\\=
\frac{(1+NR)\Delta (NR)-|\nabla NR|^2}{(1+NR)^2} =
\frac{N\Delta R+N^2R\Delta R-N^2|\nabla R|^2}{(1+NR)^2}
\\ \overset{\eqref{RdR}}{=} \frac{N\Delta R-4N^2}{(1+NR)^2} =N\frac{\Delta R-4N}{(1+NR)^2} \overset{\eqref{RdR}}{\leq} 
N\frac{A-4N}{(1+NR)^2} \leq 0
\end{multline}
Таким образом, {\it в условиях   \eqref{hypR} функция $\log (1+NR)$ супергармоническая и при $\infty \notin \partial D$\/} по Предложению 
\ref{prggscr} {\it для любой функции $s\in \sbh^+(D\setminus D_0)$ при ограниченности образа
$\bigl(s/\log (1+NR)\bigr)(D\setminus D_0)$ для выпуклой возрастающей  функции  $f\colon \RR^+\to \RR^+$  с $f(0)=0$ функция 
\begin{equation}\label{loggdtR}
	 f\Bigl(\frac{s}{\log (1+NR)}\Bigr)\log(1+NR)   \quad\text{--- тестовая функция из $\sbh_0^+(D\setminus D_0)$}. 
\end{equation}
}
\end{example} 

\subsubsection{\underline{Функции расстояния и тестовые функции}}\label{distt} В тех случаях, когда тестовые функции в данном \S~6 построены с участием функций Грина или гиперболического радиуса в части или во всей тестовой функции, как правило, при некоторых ограничениях на область $D$, можно перейти от этой тестовой функци  к эквивалентной в определенном смысле функции, зависящей от расстояния до границы
\cite{Sh83}--\cite{HKZh} или её части {\large(}см. и ср. с \cite{BGK09}--\cite{FR13}{\large)}.  
\paragraph{\textrm{\bf G.}  {\texttt{C участием функции Грина}}}\label{disttG}
Отметим некоторые известные факты.
\begin{enumerate}[{1g.}]
	\item\label{gs1} {\it Если граница области $D$ состоит из конечного числа отделенных друг от друга простых замкнутых  кривых класса $C^2$ на $\CC$, 	то\/ \cite{W}, \cite[(2.8)]{Stoll}
	\begin{subequations}\label{seqq}
	\begin{align}
	 g_D(z, z_0)&\leq B_0\dist (z, \partial D), \quad z \in  D\setminus D_0, \; z_0\in D_0,	
	\tag{\ref{seqq}a}\label{{seqq}a}
		\\
	b_0\dist (z,\partial D)&\leq g_D(z, z_0), \quad z \in D\setminus  D_0,\; z_0 \in D_0	
	\tag{\ref{seqq}b}\label{{seqq}b}
\end{align}
\end{subequations}
	 для некоторого $D_0\Subset D$, где\/ $0<b_0 \leq 1\leq B_0<+\infty$ --- постоянные, зависящие только от $z_0,D_0,D$.}
\item\label{gs2} {\it Если  область $D\ni \infty=z_0$ с неполярной границей представима в виде объединения открытых кругов радиуса не меньше некоторого фиксированного строго положительного числа, то для некоторой подобласти $D_0\Subset D$ с  $z_0=\infty \in D_0$ справедлива оценка снизу \eqref{{seqq}b}} \cite[Lemma 1,  Remark 2]{FG12_unb}.

\item\label{gs3}  
{\it Оценка  сверху\/  \eqref{{seqq}a} также справедлива,   если    $D\Subset \CC$ или $D\ni \infty$, т.\,е. $\infty \notin \partial D$, --- область Ляп\-у\-н\-о\-ва--Ди\-ни,\/} т.\,е. с границей класса $C^1$, 
на которой для внутренних нормалей $\vec{n}_{\text{\tiny inner}}(z)$, $z\in \partial D$, в любых точках $z_1,z_2 \in \partial D$ 
выполнено неравенство $\bigl|\vec{n}_{\text{\tiny inner}}(z_1)-\vec{n}_{\text{\tiny inner}}(z_2)\bigr|\leq a\bigl(|z_1-z_2|\bigr)$, где непрерывная функция $a\colon \RR^+\to \RR^+$ с $a(0)=0$ возрастает, функция $t\mapsto a(t)/t$  на $(0,+\infty)$ убывает и $\int\limits_0^1\bigl(a(t)/t \bigr)\, dt<+\infty$.
 Этот результат в  \cite[Theorem 2.3]{Widman} установлен для областей в $\RR^3$, как и предшествующий такой же факт в статье М.\,В.~Келдыша и М.\,А.~Лаврентьева для частного случая областей Ляпунова \cite{KL}, но их доказательства дословно  переносятся на области в $\CC$.
\end{enumerate}
 
Эти оценки позволяют при использовании в Основной теореме, Теореме \ref{th:1} и Следствии \ref{cor:1} тестовых функций  из  Предложения \ref{prgg}, примеров  \eqref{gdgex} и Примеров \ref{ex22} и \ref{exp1gd} заменять всю или часть тестовой функции, содержащей функции Грина $g_D$ и/или $g_{\widehat{D}}$ с $\widehat{D}\supset D$  на функцию от расстояния до границы $\partial D$. Если функция-мажоранта зависит от расстояния только до части $E \subset \partial D$ границы области $D$ как в 
\cite{BGK09}--\cite{GK12}, то область  $\widehat{D}\supset D$ в тестовых функциях с участием функции Грина $g_{\widehat{D}}$ выше, по-видимому, целесообразнее всего выбирать как объединение области $D$ со всеми кругами фиксированного радиуса с центрами на дополнении $\partial D\setminus E$ или же со всеми кругами вида $D\bigl(z,r(z)\bigr)$,  где
$z$ пробегает  дополнение $\partial D\setminus E$, а $r(z)=\dist (z, E)$  (см. и.ср. с применениями в пп. \ref{MdistED}{\bf DLE} ниже). Так, из Предложения \ref{prgg} в его условиях и ограничениях при $s=g_{\widehat{D}}$ и регулярной области   $\widehat{D}\supset D$ с учётом неравенства \eqref{est11} сразу вытекают 
\begin{propos}\label{prDDdist} Если $f\colon [1,+\infty)\to \RR^+$ --- убывающая выпуклая функция с непрерывной правой производной\footnote{Это условие позволяет легко продолжить функцию $f$ как выпуклую убывающую на $\RR$.}  в точке $1$, то для 
\underline{тестовой},\/ --- {\rm по Предложению \ref{prgg},} --- функции $g_D f\bigl(g_{\widehat{D}}/g_D\bigr)$ при условиях 
{\rm \ref{gs1}g} или {\rm\ref{gs2}g--\ref{gs3}g} на обе  области  $D\subset \widehat{D}$,  обеспечивающих  выполнение двусторонних оценок вида  \eqref{seqq} для обеих областей $D$ и $\widehat{D}$, для некоторой подобласти $D_0\Subset D$ с $z_0\in D_0$ на $D\setminus D_0$
c некоторыми постоянными  $0<a_0\leq 1\leq A_0<+\infty$, зависящими  только от $z_0,D_0, D, \widehat{D}$,   справедливы оценки 
\begin{multline}\label{gdfgd}
		a_0\dist(\cdot, \partial D) \, f\left(A_0\,\frac{\dist \bigl(\cdot, \partial \widehat{D}\bigr)}{\dist (\cdot, \partial D)}\right)\leq 
		g_D\,f\Bigl(\frac{g_{\widehat{D}}}{g_D}\Bigr)\\
\leq A_0\dist(\cdot, \partial D) \, f\left(a_0\,\frac{\dist \bigl(\cdot, \partial \widehat{D}\bigr)}{\dist (\cdot, \partial D)}\right).
\end{multline}

\end{propos}
Аналогично, для других  вариантов  тестовых функций  имеем
\begin{propos}\label{prDDdistd1} Если $f\colon \RR^+\to \RR^+$ --- возрастающая выпуклая функция с $f(0)=0$, то для 
\underline{тестовой},\/ --- {\rm из Примера \ref{exgDs} с $D'=D$,} --- функции $f\circ g_D$ при условиях 
{\rm \ref{gs1}g} или {\rm\ref{gs2}g--\ref{gs3}g} на область  $D$,  обеспечивающих  выполнение двусторонних оценок вида  \eqref{seqq}, для некоторой подобласти $D_0\Subset D$ с $z_0\in D_0$ на $D\setminus D_0$
c  постоянными  $0<b_0\leq 1\leq B_0<+\infty$ из  \eqref{seqq} справедливы оценки 
\begin{equation}\label{gdfgdd1}
	f\bigl(b_0\dist(\cdot, \partial D)\bigr)\leq  f(g_D) \leq f\bigl(B_0\dist(\cdot, \partial D)\bigr).
\end{equation}
Если $f\colon \RR^+ \to \RR^+$  --- выпуклая убывающая функция  с правой производной  	$f'_{\text{\rm\tiny right}}(0)<+\infty$, то для 
\underline{тестовой},\/ --- {\rm по Предложению \ref{prgg} с $s=1$,} --- функции $g_D\,f(1/ g_D)$ на $D\setminus D_0$ справедливы оценки
\begin{multline}\label{gdfgdd1qq}
	b_0\dist(\cdot, \partial D)\, f\Bigl(\frac{1}{b_0\dist(\cdot, \partial D)}\Bigr)\leq  
	g_D\,f(1/ g_D)\\
	\leq 	B_0\dist(\cdot, \partial D)\, f\Bigl(\frac{1}{B_0\dist(\cdot, \partial D)}\Bigr).
\end{multline}
\end{propos}

\paragraph{\textrm{\bf R.}  {\texttt{C участием гиперболического радиуса}}}\label{disttR}
Отметим аналогичные оценки  и для гиперболического радиуса $R$ области $D\Subset \CC$ гиперболического типа.
\begin{enumerate}[1R.]
	\item\label{R1} {\it  Для некоторого числа $A=\const_D^+$ всегда выполнены неравенства\/} \cite[Theorem 3.8]{AvW}
				\begin{equation}\label{loges}
		\dist (z,\partial D)	\leq R(z,D)\leq A \dist(z,\partial D)\log \frac{1}{\dist (z,\partial D)} \, .
		\end{equation}
\item\label{R2} Если граница $\partial D$ равномерно совершенна (для конечносвязной области это означает, что на её границе нет изолированных точек), то в правой части \eqref{loges} можно убрать логарифмическую часть $\log \frac{1}{\dist (z,\partial D)}$ \cite[{\bf 4.2}]{Av15}.
\end{enumerate}

Приведем  один из возможных аналогов Предложения \ref{prDDdist}  для конформного радиуса $R$   \underline{выпуклой подобласти} $D\Subset \CC$. В Примере  \ref{ex:Rconv} в качестве функции $s\in \sbh^+(D)$ выберем продолженную функцию Грина $g_{D'}(\cdot , z_0)$ с полюсом в точке  $z_0\in D_0$ регулярной подобласти $D'\subset D$ с  $z_0\in D_0\subset D'$. Тогда для  выпуклой возрастающей функцией $f\colon \RR^+\to \RR^+$ с $f(0)=0$ функция $Rf(g_{D'}/R) \in \sbh_0^+(D\setminus D_0)$ тестовая и имеет место 
\begin{propos}\label{pr:rtrf} Если подобласть $D'\subset D$ вместо $D$ удовлетворяет условию из\/  {\rm\ref{gs1}g}, то c некоторым числами
 $0<a\leq 1\leq A<+\infty$ имеет место  двусторонняя оценка 
\begin{multline*}
	\dist(\cdot, \partial D)\, f\left(a\,\frac{\dist \bigl(\cdot, \CC_{\infty}\setminus D'\bigr)}{\dist (\cdot, \partial D)}\right) 
	\overset{\rm\ref{R1}R}{\leq} 	Rf\Bigl(\frac{g_{D'}}{R}\Bigr) \\
	\overset{\rm\ref{R1}R,\ref{R2}R}{\leq}  A\dist(\cdot, \partial D)\, f\left(A\,\frac{\dist \bigl(\cdot, \CC_{\infty}\setminus D'\bigr)}{\dist (\cdot, \partial D)}\right) 
	\end{multline*}
\end{propos} 
Точно такие же двусторонние  оценки через функции расстояния легко получить и для тестовых функций вида \eqref{loggdtR} из Примера \ref{exp1gdR} при $s=g_{D'}$ с $D'\subset D$.  

\subsubsection{\underline{Тестовые функции для  $\CC$ и для проколотой в точках $\CC_{\infty}$}}\label{tfCC} 
Используется классическая терминология из  \cite{GLO}--\cite{Khab091}. Мы не стремимся в этом п.  \ref{tfCC} к построению каких-либо экзотических тестовых функций, поскольку они уже широко охвачены в наших работах  \cite{Khab91}--\cite{Khab091} вплоть до функций многих переменных \cite{Khsur}, \cite{Khab99}. Поэтому приведем лишь элементарный
 \begin{example}   Рассмотрим две соответственно $\rho$- и $\gamma$-тригонометрически  выпуклые  положительные $2\pi$-периодические  функции $h$ и $k$ на $\RR$. 
Пусть $r_0>0$. 

Тогда для любого $p\geq \rho$ функция 
\begin{equation*}
H(z)=H(re^{i\theta}):=	h(\theta)r^{-p}\geq 0, \quad\text{при  $r_0\leq r<+\infty$, $\theta \in \RR$, $z=re^{i\theta}$,} 
\end{equation*}
 субгармоническая на $\CC\setminus D(r_0)$ и стремится к нулю при $r\to +\infty$, а следовательно тестовая и принадлежит классу $\sbh_0^+\bigl(\CC \setminus (D(r_0); \leq \max\limits_{\theta\in [0,2\pi)} h(\theta)r_0\bigr)$. Более тонкие примеры тестовых функций для $D=\CC$ на основе $\rho$-тригонометрически выпуклых функций построены и нашли применения в \cite{Khab91}, \cite{Khab91s}, \cite{Khsur}, \cite{Khab99} (в иной терминологии).

Далее, рассмотрим  плоскость $D:=\CC_{*}=\CC\setminus \{0\}$ для простоты лишь с двумя проколотыми точками 
 $0$ и $\infty$. Тогда для любых  $p\geq \rho$ и $l\geq \gamma$ функция
\begin{equation*}
	H(re^{i\theta}):=\begin{cases}
	h(\theta)r^{-p}, \quad\text{при  $r_0\leq r<+\infty$},\\
	k(\theta)r^l, \quad\text{при   $0<r\leq r_0/2$}	
	\end{cases}
	\end{equation*}
из тех же соображений тестовая из  класса 
$$\sbh_0^+\bigl(\CC_*\setminus A(r_0/2, r_0); \leq \max\limits_{\theta\in [0,2\pi)} \max\{ h(\theta)r_0^{-p}, 
k(\theta)(r_0/2)^{l}\}\bigr),$$
где $A(\cdot, \cdot)$ --- кольцо, как в \ref{ssCr}, \ref{Sann}.   
\end{example} 

\subsection{Операции над тестовыми функциями} 
Приведенные здесь свойства операций, замкнутых на множестве тестовых функций основаны на широко известных общих свойствах субгармонических функций \cite{Rans}--\cite{Klimek}.  Например, если  $v\in \sbh_0^+(D\setminus D_0, \leq b)$ и $b'\in \RR^+$, то  $b'v\in \sbh_0^+(D\setminus D_0, \leq b'b)$.

Пусть $\tt J$ --- множество произвольной природы, элементы которого обозначаем буквой  $\tt j$. Множество $\tt J$ будет выступать как {\it множество индексов}. Нам будет удобнее  обозначать тестовые функции $v$ с индексом $\tt j$ как пару $(v,{\tt j})$.  Кроме того, рассматриваем множество чисел $b_{\tt j}\in \RR^+$, $\tt j \in \tt J$.

\begin{propos}\label{prsum} Пусть множество $\tt J$  конечно. Для   множества функций $(v,{\tt j})\in \sbh_0^+(D\setminus D_0, \leq b_{\tt j})$, $\tt j\in \tt J$, их сумма из  $\sbh_0^+\bigl(D\setminus D_0, \leq \sum_{\tt j\in J} b_{\tt j}\bigr)$.
\end{propos}

\begin{propos}\label{prsup} Пусть $B:=\sup_{\tt j} b_{\tt J}\in \RR^+$.
Предположим, что семейство тестовых функций $(v,{\tt j})\in \sbh_0^+(D\setminus D_0, \leq b_{\tt j})$ равномерно по $\tt J$ стремится к нулю при приближении к границе $\partial D$. Тогда полунепрерывная сверху регуляризация функции $\sup_{\tt j\in J} (v,{\tt j})$ из 
$ \sbh_0^+(D\setminus \clos D_0, \leq B)$. Более того, если $\tt J$ --- компактное топологическое пространство и функция $(v,\cdot)$ полунепрерывна сверху на $(D\setminus \clos D_0)\times {\tt J}$, то регуляризация уже не нужна \cite[Theorem 2.4.7]{Rans}.
\end{propos}
\begin{propos}\label{prinf} Пусть $\tt J$ --- направленное множество, а семейство  фу\-н\-к\-ций $(v,{\tt j})\in \sbh_0^+(D\setminus D_0, \leq b_{\tt j})$ направлено вниз, или  фильтруется влево. Тогда $\inf_{\tt j\in J} (v,{\tt j})\in \sbh_0^+\bigl(D\setminus \clos D_0,  \leq \inf_{\tt j\in J} b_{\tt j}\bigr)$.
\end{propos}
\begin{propos}\label{prlimsup} Пусть ${\tt J}\subset \RR^+$ неограниченно,  $B:=\limsup\limits_{{\tt j}\to +\infty} b_{\tt j}\in \RR^+$.
Предположим, что семейство тестовых функций $(v,{\tt j})\in \sbh_0^+(D\setminus D_0, \leq b_{\tt j})$ равномерно по $\tt J$ стремится к нулю при приближении к границе $\partial D$. Тогда полунепрерывная сверху регуляризация функции $\limsup\limits_{{\tt j}\to +\infty} (v,{\tt j})$ принадлежит классу $ \sbh_0^+(D\setminus \clos D_0, \leq B)$. 
\end{propos}

\begin{propos}\label{print} Пусть   
$(v, {\tt j})$ --- тестовые функции, $({\tt J},\kappa)$ ---  пространство с мерой $\kappa ({\tt J})<+\infty$, функция $(v,\cdot)$ измерима на $(D\setminus \clos D_0)\times {\tt J}$, функция  $z\mapsto \sup_{\tt j\in J} v(z,{\tt j})$ локально ограниченно сверху на $D\setminus \clos D_0$. Тогда интеграл 
	 $\int_{\tt J} v(z,{\tt j})\, d\kappa ({\tt j})$
 --- тестовая функция \cite[Theorem 2.4.8]{Rans}.
\end{propos}
\begin{propos}\label{prD} Пусть $v\in \sbh_0^+(D\setminus D_0,\leq b)$ и в $D\setminus D_0$ задана система непересекающихся регулярных подобластей $D_{\tt j}$, $\tt j\in J$. Предположим, что можно выбрать  исчерпание $D_n\supset D_0$ области $D$ так, что граница $\partial D_n$ не пересекается ни с одним $D_{\tt j}$ при  больших $n$.  Если  гармонически 
продолжить функцию $v$ внутрь каждой подобласти $D_{\tt j}$, то  полученная преобразованная функция по-прежнему  из класса 
$\sbh_0^+(D\setminus D_0,\leq b)$ \cite[Теорема 2.18]{HK}.

\end{propos}

\begin{propos}\label{prhol} Рассмотрим, наряду $D_0\Subset D$,  еще одну  пару областей  $\varnothing \neq D_0'\Subset D'\neq \CC_{\infty}$. Пусть функция  $h\in \Hol (D\setminus D_0)$ с образом $h(D\setminus D_0)\subset D'\setminus D_0'$ 
удовлетворяет условию: 
для любой подобласти $D'_1\Subset D'$ с $D_0'\subset D'_1$  найдётся подобласть $D_1\Subset D$ c $D_0\subset D$, для которой $h(D\setminus D_1)\subset D'\setminus D_1'$. Тогда для любой тестовой функции $v\in \sbh_0^+(D'\setminus D_0'; \leq b)$ функция $v\circ h$ тестовая на $D\setminus D_0$ из класса  $\sbh_0^+(D\setminus D_0; \leq b)$ \cite[Theorem 2.7.4]{Rans}.
\end{propos}
\section{Основные результаты для нерадиальных мажорант}\label{Snonrad}

\subsection{Операторы Лапласа и набла для суперпозиций и произведений}\label{OLsf}
Этот во многом вычислительный подраздел, наряду с тем, что дает правила вычисления мер Рисса субгармонических функций, полученных суперпозицией и умножением на функцию, как в Теореме G--K, содержит в себе и новое доказательство Теоремы G--K.
Далее для функций $f,g$ класса $C^2$ исходим из элементарной формулы векторного анализа
\begin{equation}\label{mfor}
	\Delta (gf)=f\Delta g+ 2 \nabla g\cdot \nabla f +g\Delta f,
\end{equation}
  где $\nabla g\cdot \nabla f$ --- скалярное произведение градиентов, и её частного случая с $g=1$ 
				\begin{equation}\label{{seq}s}
				|\nabla f|^2=\frac12 \Bigl(\Delta (f^2)-2f\Delta f\Bigr).
		\end{equation}

\begin{theorem}\label{prLn}
Пусть $g$ и $s$ функции класса $C^2$ на некотором открытом множестве $\mathcal O$ в $\CC$, $0\notin g(\mathcal O)$ и 
$f$ --- функция класса $C^2$ на $(s/g)(\mathcal O)$, $f(s/g):=f\circ\dfrac{s}{g}$ --- суперпозиция. 
Справедлива формула
\begin{equation}\label{trf}
	\Delta \Bigl(gf\Bigl(\frac{s}{g}\Bigr)\Bigr)=f'\Bigl(\frac{s}{g}\Bigr)\Delta s -\left(f'\Bigl(\frac{s}{g}\Bigr) \frac{s}{g} -f\Bigl(\frac{s}{g}\Bigr) \right)\Delta g
	+gf''\Bigl(\frac{s}{g}\Bigr) \Bigl|\nabla \frac{s}{g}\Bigr|^2,
\end{equation}
где последний сомножитель в обозначении 
$\nabla g\cdot \nabla s$ для  скалярного произведения градиентов записывается как <<полный квадрат разности>>
\begin{subequations}\label{66g}
\begin{align}
	\Bigl|\nabla \frac{s}{g}\Bigr|^2&=\frac{g^2|\nabla s|^2-2gs\, \nabla g\cdot \nabla s+ s^2|\nabla g|^2}{g^4}
	\tag{\ref{66g}$\nabla$}\label{66gn}
	\\
&\overset{\eqref{mfor}, \eqref{{seq}s}}{=}\frac{1}{2g^4}
\Bigl(g^2\Delta(s^2)-2gs\Delta(gs)+s^2\Delta (g^2) \Bigr)	
\tag{\ref{66g}$\Delta$}\label{66gD}
\end{align}
\end{subequations}

В частности, при $g=1$ --- тождественная единица, имеем
\begin{equation}\label{g11}
	\Delta  (f\circ s)=f'(s)\Delta s +f''(s) |\nabla s|^2
	=\Bigl(f'(s)-sf''(s)\Bigr) \Delta s +\frac{1}{2} f''(s)\Delta (s^2).
\end{equation}
а при $s=1$
\begin{equation}\label{s11}
	\Delta \Bigl(gf\Bigl(\frac1{g}\Bigr)\Bigr)=
	\left (f\Bigl(\frac1{g}\Bigr) -
	\frac{1}{g}\,f'\Bigl(\frac1{g}\Bigr) -\frac{1}{g^2}\,f''\Bigl(\frac{1}{g}\Bigr)\right) \Delta g+
	\frac{1}{2g^3}\,f''\Bigl(\frac{1}{g}\Bigr)\Delta(g^2).
\end{equation}
\end{theorem}
\begin{proof}  Теперь действуем отдельно по каждому слагаемому в правой части \eqref{mfor} 	(ниже в доказательстве роль аргумента в функции $f$ и в производных $f$ всегда играет  функция $s/g$).
			\begin{equation*}
	g\Delta f=g\nabla^2 f	=g \nabla \cdot \Bigl(f' 
		\nabla \frac{s}{g}\Bigr) 	=g\left( f'' 	\Bigl|\nabla \frac{s}{g}\Bigr|^2
	+f' 	\nabla^2\frac{s}{g}\right),
	\quad  2\nabla g\cdot \nabla f=2f'\nabla g\cdot\nabla \frac{s}{g}\,.		
	\end{equation*}
Подставляя в \eqref{mfor}, получаем
\begin{multline}\label{dlff}
	\Delta\bigl(gf(s/g)\bigr)=f\Delta g
	+2f'\nabla g\cdot\nabla \frac{s}{g}
	+g\left( f''
	\Bigl|\nabla \frac{s}{g}\Bigr|^2
	+f'
	\nabla^2\frac{s}{g}\right)\\
	=f\Delta g
	+f' \Bigl(g\Delta \frac{s}{g}+2 \nabla g\cdot \nabla \frac{s}{g} \Bigr)
	+g f'' 	\Bigl|\nabla \frac{s}{g}\Bigr|^2.
\end{multline}
Здесь множитель при $f'$ можно представить в в иде 
\begin{multline*}
	g\Delta \frac{s}{g}+2 \nabla g\cdot \nabla \frac{s}{g}
	=\Bigl(g\Delta \frac{s}{g}+2 \nabla g\cdot \nabla \frac{s}{g}+
	\frac{s}{g} \Delta g \Bigr)-\frac{s}{g} \Delta g\\
	=\Bigl|\text{ снова формула \eqref{mfor} }\Bigr|
	=\Delta \Bigl(g\,\frac{s}{g}\Bigr)-\frac{s}{g} \Delta g
	=\Delta s-\frac{s}{g} \Delta g.
		\end{multline*}
Таким образом, меняется множитель при $f'$ в  \eqref{dlff}, а именно:
\begin{multline*}
	\Delta\bigl(gf(s/g)\bigr)=f\Delta g
	+2f'\nabla g\cdot\nabla \frac{s}{g}
	+g\left( f'' 	\Bigl|\nabla \frac{s}{g}\Bigr|^2
	+f' 	\nabla^2\frac{s}{g}\right)\\
	=f\Delta g
	+f' \Bigl(\Delta s-\frac{s}{g} \Delta g \Bigr)
	+g f'' 	\Bigl|\nabla \frac{s}{g}\Bigr|^2
	=f' \Delta s +\Bigl(f-f'\frac{s}{g} \Bigr)\Delta g
		+g f''
	\Bigl|\nabla \frac{s}{g}\Bigr|^2,
	\end{multline*}
что совпадает с требуемой формулой \eqref{trf}.

Формула \eqref{66gn} --- простая выкладка из  векторного анализа:
\begin{equation*}
\Bigl|\nabla \frac{s}{g}\Bigr|^2= \nabla \frac{s}{g}\cdot\nabla \frac{s}{g}=\left(\frac{g\nabla s-s\nabla g}{g^2}\right)^2
=
\frac{g^2|\nabla s|^2-2gs\, \nabla g\cdot \nabla s+ s^2|\nabla g|^2}{g^4}	\,.
\end{equation*}
 Переход от   \eqref{66gn} к \eqref{66gD} осуществляется за счёт формул \eqref{mfor}--\eqref{{seq}s}, записанной в трёх различных  видах 
\begin{subequations}\label{seq}
\begin{align}
|\nabla s|^2&=\frac12 \Bigl(\Delta (s^2)-2s\Delta s\Bigr)
\tag{\ref{seq}s}\label{{seq}s1}\\
2\nabla g\cdot \nabla s &=\Delta(gs)-g\Delta s-s\Delta g.
\tag{\ref{seq}gs}\label{{seq}gs}\\
|\nabla g|^2&=\frac12 \Bigl(\Delta (g^2)-2g\Delta g\Bigr)
\tag{\ref{seq}g}\label{{seq}g}
\end{align}
\end{subequations}
Подстановка выражений \eqref{seq} в  \eqref{66gn} после приведения подобных и даёт \eqref{66gD}. Оставшиеся формулы \eqref{g11} и \eqref{s11} --- частные случаи \eqref{trf}--\eqref{66g}.
\end{proof}
\begin{proof}[теоремы G--K {\rm (с явной мерой Рисса)}]\label{rr} В случае \eqref{ih}  в правой части  \eqref{trf} остаётся только последнее слагаемое в \eqref{trf}, а оно, очевидно, положительно. Другими словами 
\begin{enumerate}[(i)]
	\item в  {\it случае \eqref{ih} плотность меры Рисса субгармонической функции $gf(s/g)$ задаётся с учётом \eqref{66gD} выражением
		\begin{equation}\label{mi}
		\frac{1}{4\pi g^3}
\Bigl(g^2\Delta(s^2)-2gs\Delta(gs)+s^2\Delta (g^2) \Bigr)	f''(s/g).
	\end{equation}
	}
\end{enumerate}
В случае   \eqref{iih} там же остаются только первое и последнее слагаемое, которые положительны. 
Другими словами 
\begin{enumerate}[(ii)]
	\item в  {\it случае \eqref{iih} плотность меры Рисса субгармонической функции $gf(s/g)$ задаётся с учётом \eqref{66gD} выражением
		\begin{equation}\label{mii}
	f'\Bigl(\frac{s}{g}\Bigr) \,d\nu_s  +	\frac{1}{4\pi g^3}
\Bigl(g^2\Delta(s^2)-2gs\Delta(gs)+s^2\Delta (g^2) \Bigr)	f''(s/g),
	\end{equation}
	где $\nu_s$ --- мера Рисса субгармонической функции $s$.}
\end{enumerate}

В случае \eqref{iiih} вопрос только в положительности промежуточного слагаемого, а это следует из условия $\Delta g\leq 0$ и очевидного свойства $f'(x)x\geq f(x)$ для выпуклых функций   $f\colon \RR^+\to \RR^+$ с $f(0)=0$. При этом
\begin{enumerate}
	\item[{(iii)}] {\it случае \eqref{iih} плотность меры Рисса субгармонической функции $gf(s/g)$ задаётся с учётом \eqref{66gD} выражением
		\begin{multline}\label{miii}
	f'\Bigl(\frac{s}{g}\Bigr) \, d\nu_s +\left(f'\Bigl(\frac{s}{g}\Bigr) \frac{s}{g} -f\Bigl(\frac{s}{g}\Bigr) \right) \, d\nu_{-g}
	\\+\frac{1}{4\pi g^3}
\Bigl(g^2\Delta(s^2)-2gs\Delta(gs)+s^2\Delta (g^2) \Bigr)	f''(s/g),
\end{multline}
где $\nu_s$ и $\nu_{-g}$--- меры Рисса субгармонических функции $s$ и $-g$. }
\end{enumerate}

Для перехода к общему случаю (не класса $C^2$) пользуемся отработанной техникой аппроксимации суб- и супергармонических и выпуклых функций монотонными  последовательностями функций класса $C^2$. То, что в определениях  плотностей мер Рисса в \eqref{mi}--\eqref{miii} участвуют именно операторы Лапласа, т.е. нет градиентов, позволяет утверждать, что \eqref{mi}--\eqref{miii} в этом общем  случае можно рассматривать, как плотности мер Рисса, определённые в смысле теории распределений, т.е. обобщенных функций.    
\end{proof}
\begin{remark}\label{rrfL} Выражения \eqref{mi}--\eqref{miii} определяют и явный вид мер Рисса  тестовых функций из всех Предложений и Примеров подраздела \ref{sstf}.
\end{remark}
\begin{remark} Нам известны  два варианта многомерных обобщений Теоремы \ref{prLn}. Первый из них, практически идентичный, даёт подобные формулы для оператора Лапласа от подобных же суперпозиций в $\RR^k$. Второй, более деликатный, связан с заменой оператора Лапласа на форму Леви в $\CC^n$ {\large(}см. и ср. с \cite[Theorem 2.9.19]{Klimek}{\large)}. Мы намерены вернуться  к этому в ином месте.  
\end{remark}

\subsection{Субгармонические мажоранты-суперпозиции}

В рамках Основной Теоремы, Теоремы \ref{th:1} и Следствия \ref{cor:1} имеет значение лишь вид меры Рисса субгармонической функции-мажоранты $M\in \sbh(D)$  вблизи границы $\partial D$, т.\,е. при  рамочном соглашении (D) из  п.~\ref{ns} Введения в окрестности $D\setminus D_0$, в $D\setminus \clos D_0$ или в $D\setminus \{z_0\}$ при некотором $z_0\in D_0$. Сначала рассмотрим более простой случай суперпозиции выпуклой и субгармонической функции.
\begin{theorem}\label{th:fs}   Пусть  мажоранта имеет вид $M=F\circ s$ --- суперпозиция выпуклой возрастающей функции $F\colon \RR\to \RR$ 
класса $C^2$ с $F(-\infty):=\lim\limits_{x\to -\infty}F(x)$
и   $s\in \sbh(D)\setminus \{\boldsymbol{-\infty}\}$, т.\,е. по Предложению\/ {\rm C}\eqref{iihF} с $g=1$ имеем $F\circ s \in  \sbh(D)$.  Введем обозначения  $\nu_s$ и $\nu_{s^2}$ для мер Рисса функций $s$ и  $s^2\in \sbh(D)$. Пусть  $z_0\in D_0$, $u\in \sbh (D)$ с 
$u(z_0)\neq -\infty$ с мерой Рисса $\nu_u$, $u\leq F\circ s$ на $D\setminus D_0$; $b\in \RR^+$. Тогда найдутся постоянные $C, \overline{C}_M\in \RR^+$ вида \eqref{cz0C+}, с которыми  для любой тестовой функции $v\in \sbh_0^+(D\setminus D_0;\leq b)$   
\begin{multline}\label{mris}
	\int_{D\setminus D_0} v \, d\nu_u \leq \int_{D\setminus D_0}   \bigl(F'(s)-sF''(s)\bigr) v\, d\nu_s\\
	+\frac{1}{2}\int_{D\setminus D_0}v F''(s) \, d \nu_{s^2}+ 	C\overline{C}_M -Cu(z_0).
\end{multline}
В частности, если сумма интегралов в правой части \eqref{mris} конечна\/ {\rm  ($<+\infty$),} то конечен  и интеграл в левой части. В случае  $u=\log |h|$ с ненулевой функцией $h\in \Hol (D)$, 	обращающейся в нуль на последовательности 	${\tt Z}=\{{\tt z}_k\}_{k\in \NN}\subset D$, левая часть \eqref{mris}  будет выглядеть как сумма из левой части \eqref{in:fz0}. 
\end{theorem}   
\begin{proof} Выполнены условия Теоремы \ref{th:1}. Соотношение \eqref{mris} --- в точности 
\eqref{mest+}, поскольку в правой  части  берётся интеграл от $v$ по мере Рисса 
\begin{equation}\label{meraRissa}
	\nu_M=\nu_{F\circ s}=\bigl(F'(s)-sF''(s)\bigr)  \nu_s +\frac{1}{2} F''(s)  \nu_{s^2}
\end{equation}
 --- частный случай \eqref{mii}  п.~(ii) с $g=1$ и $F$ вместо $f$ нашего доказательства теоремы G--K (с явной мерой Рисса) из подраздела \ref{OLsf}.
\end{proof}
\begin{addition}[{\rm к Теореме \ref{th:fs}}]\label{add:gn} {\rm Мера Рисса из \eqref{mris} может определятся и иначе.  Из первого тождества в \eqref{g11} с $F$ вместо $f$,
	переписанного в виде
	\begin{equation*}
		\Delta  (F\circ s)=F'(s)\Delta s +F''(s) |\nabla s|^2  = F'(s) \Delta s+\nabla F'(s) \cdot \nabla s
			\end{equation*}
по первой формуле Грина для подобластей $U\Subset D\setminus \clos D_0$ с границей класса $C^1$ и функцией  $s\in \sbh(D)$ с сужением на $D\setminus \clos D_0$ класса $C^2$ имеем
  \begin{equation}\label{f:byUn}
\nu_{F\circ s}	(U) =\frac{1}{2\pi}\int \bigl (	F'(s) \Delta s+\nabla F'(s) \cdot \nabla s \bigr)\, d \lambda =
\frac{1}{2\pi}\int_{\partial U} F'(s) \frac{\partial s}{\partial \vec{n}_{\rm\tiny out}}	\,d\sigma,
		\end{equation}
		где $\lambda$ --- мера Лебега, $\sigma$ --- мера длины границы $\partial U$, $\frac{\partial}{\partial \vec{n}_{\text{\tiny out}}}$ --- оператор дифференцирования по внешней нормали.} 
\end{addition}
\subsubsection{\underline{Мажоранта-суперпозиция с функцией Грина}}\label{mspfG}
 Пусть $F\colon (-\infty,0)\to \RR$ --- выпуклая возрастающая функция,  доопределенная, как в  Теореме C\eqref{iihF}, равенством  $F(-\infty):=\lim\limits_{x\to -\infty} F(x)$. 
Пусть $z_0 \in D_0\Subset  D\subset \widehat{D}$ и $\widehat{D}$ --- область с неполярной границей. Возможно  $\widehat{D}=D$.	Тогда  $s(z):= -g_{\widehat{D}}(z, z_0)$, где $z\in D$ и $g_{\widehat{D}}(z_0, z_0):=+\infty$, ---{\it  субгармоническая по $z$ функция в области\/ $D$} \cite{Rans}, \cite{HK}.

 Субгармоническое по Теореме C\eqref{iihF} сужение на $D$ функции-суперпозиции  $M:=F\circ \bigl(-g_{\widehat{D}}(\cdot, z_0)\bigr)\in \sbh \bigl(\widehat{D}\bigr)$ 
и обсуждается  в данном   п.~\ref{mspfG}.   Поскольку сужение 
	$g_{\widehat{D}}:=g_{\widehat{D}}(\cdot, z_0)$ на $D\setminus \{z_0\}$ --- гармоническая функция, то в данном случае с функцией  $F$ класса $C^2$ мера Рисса функции $F\circ (-g_{\widehat{D}})$ задаётся как 
		\begin{equation*}
\nu_{F\circ (-g_{\widehat{D}}) }	\overset{\eqref{g11}}{=}
\frac{1}{4\pi} F''(-g_{\widehat{D}})\Delta{\bigl((g_{\widehat{D}})^2\bigr)} \lambda	\overset{\eqref{g11}}{=}\frac{1}{2\pi} F''(-g_{\widehat{D}}) |\nabla g_{\widehat{D}}|^2\lambda  
	\end{equation*}
	на $D\setminus \{z_0\}$,  где $(g_{\widehat{D}})^2\in \sbh^+\bigl(D\setminus \{z_0\}\bigr)$, $\lambda$ --- мера Лебега. Иная форма из Дополнения \ref{add:gn} к Теореме \ref{th:fs} в обозначениях после \eqref{f:byUn} --- 
	\begin{equation}\label{f:byUng}
\nu_{F\circ (-g_{\widehat{D}})}	(U) \overset{\eqref{f:byUn}}{=}
\frac{1}{2\pi}\int_{\partial U} F'(-g_{\widehat{D}}) \frac{\partial (-g_{\widehat{D}})}{\partial \vec{n}}	\,d\sigma, \quad U\Subset D. 
		\end{equation}
		В частности, если $t_0\in (0,+\infty)$, $0<t<t_0$, и рассматриваем только {\it относительно компактные регулярные подобласти  
		\begin{equation}\label{UtC2}
			U_t:=\{z\in D\colon g_{\widehat{D}} (z,z_0)>t\}\Subset D \quad \text{ в $D$ с границей класса $C^1$}
		\end{equation}
		 и в точности с линиями  уровня  $\{z\in D\colon g_{\widehat{D}} (z,z_0)=t\}$, совпадающими с границей $\partial U_t$,\/} то  
		функция $g_{\widehat{D}} (\cdot ,z_0) -t$, рассматриваемая только на $U_t$, обладает всеми свойствами, полностью характеризующими  функции Грина области $U_t$ с полюсом $z_0$ \cite[Теорема 5.27]{HK}, т.\,е. для выбранного  числа $t>0$
\begin{equation}\label{gUt0}
	g_{\widehat{D}} (\cdot ,z_0) -t\Bigm|_{U_t}=g_{U_t} \quad \text{\it --- функция Грина для  $U_t$ с полюсом $z_0$}. 
\end{equation}
				Следовательно, 		\eqref{f:byUng} при таких $0<t_2<t_1<t_0$		для подобласти-разности $U_{t_2<t_1}:=U_{t_2}\setminus \clos U_{t_1}$ можно переписать в виде
		\begin{multline}\label{f:byUng+}
\nu_{F\circ (-g_{\widehat{D}})}	(U_{t_2<t_1}) \overset{\eqref{f:byUng}}{=}
\frac{1}{2\pi}  \int_{\partial U_{t_2<t_1}}F'(-g_{\widehat{D}})\frac{\partial (-g_{\widehat{D}})}{\partial \vec{n}}	\,d\sigma\\
= F'(-t_2) \frac{1}{2\pi}\int_{\partial U_{t_2}}\frac{\partial (-g_{\widehat{D}}+t_2)}{\partial \vec{n}_{\text{\tiny out for $U_{t_2}$}}}	\,d\sigma
+F'(-t_1)\frac{1}{2\pi}  \int_{\partial U_{t_1}}\frac{\partial (-g_{\widehat{D}}+t_1)}{\partial \vec{n}_{\text{\tiny inner for $U_{t_1}$}}}	\, d\sigma \\
\overset{\eqref{gUt0}}{=}F'(-t_2) \frac{1}{2\pi}\int_{\partial U_{t_2}}\frac{\partial (-g_{U_{t_2}})}{\partial \vec{n}_{\text{\tiny out for $U_{t_2}$}}}	\,d\sigma
+F'(-t_1)\frac{1}{2\pi}  \int_{\partial U_{t_1}}\frac{\partial (-g_{U_{t_1}})}{\partial \vec{n}_{\text{\tiny inner for $U_{t_1}$}}} . 
		\end{multline}
		Поскольку в данных условиях гладкости границ $\partial U_{t_1}$ и $\partial U_{t_2}$ 
		\begin{equation*}
		\frac{1}{2\pi}\frac{\partial (-g_{U_{t_2}})}{\partial \vec{n}_{\text{\tiny out for $U_{t_2}$}}}, \quad
			\frac{1}{2\pi} \frac{\partial g_{U_{t_1}}}{\partial \vec{n}_{\text{\tiny inner for $U_{t_1}$}}}
		\end{equation*}
		--- это ядра Пуассона (с центром $z_0$) или плотности гармонических мер в точке $z_0$ соответственно   для областей $U_{t_2}$ и $U_{t_1}$ на  границе $\partial U_{t_2}$ 		и на границе $\partial U_{t_1}$, то, домноженные на $1/2\pi$, первый из интегралов в правой части  \eqref{f:byUng+}  равен $1$, а второй равен $-1$. Таким образом, из \eqref{f:byUng+} следует
		
		\begin{equation*}
			\nu_{F\circ (-g_{\widehat{D}})}	(U_{t_2<t_1})=F'(-t_2)-F'(-t_1), \quad -t_0 <-t_1<-t_2 <0,			
		\end{equation*}
		или, в терминах плотности меры Рисса на линиях уровня функции Грина $g_{\widehat{D}}$
				\begin{equation}\label{F2dev}
			\frac{d}{dt}\nu_{F\circ (-g_{\widehat{D}})}	(U_{t})=F''(-t), \quad -t_0 <-t <0.			
		\end{equation}
		
		Рассмотрим раздельно некоторые специальные случаи.
					\paragraph{\textrm{\bf G.}  {\texttt{Случай $\widehat{D}=D$ и суперпозиции с  $-g_D$}}} Во  избежание технических ос\-л\-о\-ж\-н\-ений  считаем, что вместе с  $\widehat{D}=D$ --- регулярная область, начиная с некоторого $t_0>0$, все области $U_t$ из \eqref{UtC2} при $0<t\leq t_0$ такие же, как описано при  \eqref{UtC2}. Для широкого класса конечносвязных областей реализации этого подробно описаны
						в \cite[Гл.~V, \S~1]{St}. Функция $F$ класса $C^2$ та же, что и в начале п.~\ref{mspfG};   $g_D:=g_D(\cdot, z_0)$. В рамках  соглашений данного  пп.~{\bf G} имеет место 
\begin{theorem}\label{th:fsg} Пусть мажоранта  $M=F\circ (-g_D)$, --- очевидно, субгармоническая на $D$, --- и функция $u\in \sbh (D)$ с $u(z_0)\neq -\infty$ и мерой Рисса $\nu_u$ удовлетворяют условию $u\leq M$ на $D\setminus U_{t_0}$, $z_0\in U_{t_0}$, $b\in \RR^+$. Тогда  найдутся постоянные $C,\overline{C}_M\in \RR^+$ вида \eqref{cz0C+} с  $D_0:=U_{t_0}$, с которыми  
\begin{enumerate}[{\rm (a)}]
	\item\label{gFt0a} для любой выпуклой возрастающей функции   $f\colon \RR^+ \to \RR^+$ с $f(0)=0$ и $f(t_0)\leq b$ выполнено неравенство
\begin{equation}\label{itotsfF}
	\int_{D\setminus U_{t_0}} (f\circ g_D) \, d \nu_u\leq \int_0^{t_0}f(t)F''(-t) \, dt+C\overline{C}_M-Cu(z_0);
\end{equation}
	\item\label{gFt0b} для любой выпуклой убывающей функции   $f\colon \RR^+ \to \RR^+$   с правой производной 
	$f'_{\text{\rm\tiny right}}(0)<+\infty$ и с $f(0)\leq b/t_0$  выполнено неравенство
\begin{equation}\label{itotsfFg}
	\int_{D\setminus U_{t_0}} g_D\,f\bigl({1}/{g_D}\bigr) \, d \nu_u\leq \int_0^{t_0} tf(1/t)F''(-t) \, dt+C\overline{C}_M-Cu(z_0).
\end{equation}
\end{enumerate}
	В частности, если  интеграл в правой части \eqref{itotsfF} или \eqref{itotsfFg} конечен, то конечен  и интеграл в левой части соответственно \eqref{itotsfF} или \eqref{itotsfFg}. Для $u=\log |h|$ с ненулевой функцией $h\in \Hol (D)$, 	обращающейся в нуль на 	${\tt Z}=\{{\tt z}_k\}_{k\in \NN}\subset D$, левая часть \eqref{itotsfF} или  \eqref{itotsfFg} будет соответственно выглядеть как сумма
	$\sum\limits_{{\tt z}_k\in D\setminus U_{t_0}} f\bigl(g_D({\tt z}_k, z_0)\bigr)$ или 
	$\sum\limits_{{\tt z}_k\in D\setminus U_{t_0}} g_D({\tt z}_k, z_0) \,f\bigl(1/g_D({\tt z}_k, z_0)\bigr)$.	
	\end{theorem}				
\begin{proof} \textit{Случай}\/ \eqref{gFt0a}. Функция $f\circ g_D$ из  класса $\sbh(D\setminus U_{t_0}; \leq b)$ по Предложению \ref{prggscr}\eqref{hb} при выборе $g:=1$ и $s:=g_D$ и условию $f(t_0)\leq b$, т.\,е., в частности, тестовая функция 
из Предложения \ref{prggscr} п.~\ref{tfiii}.
Тогда по Теореме  \ref{th:1} неравенство  \eqref{mest+} переписывается в требуемом виде 
	\begin{multline*}
		\int_{D\setminus U_{t_0}} (f\circ g_D) \, d\nu_u \leq 
		\int_{D\setminus U_{t_0}} (f\circ g_D) \, d\nu_{M}+ 	C\overline{C}_M -Cu(z_0)\\
		=\int_{D\setminus U_{t_0}} (f\circ g_D) \, d\nu_{F\circ(-g_D)}+ 	C\overline{C}_M -Cu(z_0)
		= \int_0^{t_0} f(t) \, d \nu_{F\circ (-g_{D})}	(U_{t})+ 	C\overline{C}_M -Cu(z_0)\\
				\overset{\eqref{F2dev}}{=} 
			\int_0^{t_0} f(t) F''(-t)\, d t+ 	C\overline{C}_M -Cu(z_0).
		\end{multline*}
\textit{Случай}\/ \eqref{gFt0b}. Ввиду условия $f'_{\text{\rm\tiny right}}(0)<+\infty$  функция $f$ может быть продолжена на $\RR$ как аффинная на
 $(-\infty,0]$ и выпуклая убывающая на $\RR$. По Предложению   \ref{prgg} с $s:=1$ и $g:=g_D$ функция $g_D\,f(1/ g_D)$ тестовая из  класса 
$\sbh_0^+(D\setminus U_{t_0};\leq b)$. Дальнейшие рассуждения и выкладки такие же,  как и в Случае  \eqref{gFt0a}.

\end{proof}

\begin{remark}\label{obrF}  Для всякой функции $M\colon D\to \RR$, постоянной на каждой линии уровня функции Грина $g_D(\cdot, z_0)$ с полюсом  $z_0\in D$ и непрерывной в точке $z_0\in D$, очевидно, можно однозначно построить некоторую функцию $F\colon (-\infty, 0)\to \RR$, для которой существует предел   $F(-\infty) :=\lim\limits_{x\to -\infty} F(x)$ и $M:=F\circ (-g_D)$ на $D$. Допустим, что $F$ --- возрастающая в правой окрестности точки $-\infty$ функция класса $C^2$ на $(-\infty,0)$. Из равенства \eqref{g11}, записанного  с $f=F$ и $s=-g_D$ в виде
\begin{equation}\label{g11+}
	\frac{1}{2\pi}\Delta  \bigl(F\circ (-g_D)\bigr)=	\frac{1}{2\pi}F''(-g_D) |\nabla g_D|^2=\frac{1}{4\pi} F''(-g_D)\Delta \bigl((g_D)^2\bigr)
\end{equation}
легко понять,  что $M=F\circ (-g_D)\in \sbh (D)$  тогда и только тогда, когда 
$F$ --- выпуклая на $(-\infty, 0)$, а следовательно, ввиду возрастания справа от $-\infty$, и всюду возрастающая функция
{\large(}обязательно с плотностью \eqref{g11+} её меры Рисса $\nu_{F\circ (-g_D)}$, сосредоточенной, ввиду конечности $M(z_0)$, вне точки $z_0${\large)}. Таким образом, Теорема   
\ref{th:fsg} охватывает при  $F$ класса $C^2$ {\it все возможные случаи мажорант\/} $M\colon D\to \RR$, постоянных на линиях уровня функции Грина области $D$ с фиксированным полюсом $z_0\in D$.
\end{remark}	
\begin{remark}\label{f2pr} Используя аппроксимацию произвольной выпуклой функции $F\colon (-\infty,0)\to \RR^+$ убывающей последовательностью выпуклых функций класса $C^2$, в предположении, что для одной из аппроксимирующих функций интеграл в в правой части \eqref{itotsfF} конечен,  условие {\it <<$F$ из класса $C^2$>>\/} в Теореме \ref{th:fsg} можно снять, если переписать интеграл в правой части \eqref{itotsfF} в виде 
\begin{equation}\label{f2prf}
	\int_{-t_0}^0f(-t)dF_{\text{\tiny right}}'(t)=-\int_0^{t_0} f(t) \, d F_{\text{\tiny left}}'(-t) .
\end{equation}
 
\end{remark}
					
\paragraph{\textrm{\bf LG.}  {\texttt{Случай $\widehat{D}=D$ и суперпозиция с $-\log (1+g_D)$}}}		Остаёмся в рамках соглашений пп.~{\bf G} перед Теоремой \ref{th:fsg}. Из Примера \ref{exp1gd}  $-\log(1+g_D)\in \sbh (D)$. Учитывая  Теорему C\eqref{iihF}, 	для мажоранты
 $M:=F\circ \bigl(-\log (1+g_D)\bigr)\in \sbh(D)$ 
\begin{multline}\label{denFlog}    
\Delta M=\Delta \Bigl(F\circ \bigl(-\log (1+g_D)\bigr)\Bigr)=\nabla \cdot \nabla \Bigl(F\circ \bigl(-\log (1+g_D)\bigr)\Bigr)
-\nabla \cdot \left(F'\bigl(1+g_D\bigr) \frac{-\nabla g_D}{1+g_D} \right)\\=
 F''\bigl(-\log (1+g_D)\bigr)\frac{|\nabla g_D|^2}{(1+g_D)^2}
+ F'\bigl(-\log (1+g_D)\bigr) \frac{-(1+g_D)\nabla^2 g_D+\nabla g_D \cdot \nabla g_D}{(1+g_D)^2} 
\\
= \Bigl(F''\bigl(-\log (1+g_D)\bigr)+ F'\bigl(-\log (1+g_D)\bigr)\Bigr)\frac{|\nabla g_D|^2}{(1+g_D)^2}
\\ 
\overset{\eqref{{seq}s}}{=}
 \Bigl(F''\bigl(-\log (1+g_D)\bigr)+ F'\bigl(-\log (1+g_D)\bigr)\Bigr)\frac{\Delta \bigl((g_D)^2\bigr)}{2(1+g_D)^2}
\end{multline}
и плотность её меры Рисса $\nu_M$ равна деленному на $2\pi$ последнему  выражению.
\begin{theorem}\label{th:loggD} 
Пусть  $M:=F\circ \bigl(-\log (1+g_D)\bigr)\in \sbh(D)$  и функция $u\in \sbh (D)$ с $u(z_0)\neq -\infty$ и мерой Рисса $\nu_u$ удовлетворяют условию $u\leq M$ на $D\setminus D_0$, $z_0\in D_0$, $b\in \RR^+$. Тогда  найдутся $C,\overline{C}_M\in \RR^+$ вида \eqref{cz0C+}, с которыми  для любых выпуклой возрастающей функции   $f\colon \RR^+ \to \RR^+$ с $f(0)=0$ и функции $s\in \sbh^+(D\setminus D_0)$ при ограничениях   
\begin{equation}\label{1gD2l}
	t_0:=\sup_{z\in D\setminus D_0} \frac{s(z)}{\log \bigr(1+g_D(z,z_0)\bigl)} <+\infty, \quad 
	  f(t_0)\sup_{\partial D_0}g_D\leq b
\end{equation}
выполнено неравенство 
\begin{multline}\label{itotsfF+}
	\int_{D\setminus D_0}  f\Bigl(\frac{s}{\log (1+g_D)}\Bigr)\log (1+g_D) \, d \nu_u\\
	\overset{\eqref{denFlog}}{\leq} \frac{1}{2\pi} \int_{D\setminus D_0} 	f\Bigl(\frac{s}{\log (1+g_D)}\Bigr)
	\frac{g_D\,\Delta \bigl((g_D)^2\bigr) 
	}{2(1+g_D)^2}(F''+F')\circ(-\log (1+g_D)\bigr)  \, d \lambda \\+C\overline{C}_M-Cu(z_0).
\end{multline}
	В частности, если  интеграл в правой части \eqref{itotsfF+} конечен, то конечен  и интеграл в левой части \eqref{itotsfF+}.
	Для $u=\log |h|$ с  $0\neq h\in \Hol (D)$ и $h({\tt Z})=0$ для 	${\tt Z}=\{{\tt z}_k\}_{k\in \NN}\subset D$, левая часть \eqref{itotsfF+}  будет выглядеть как сумма
		\begin{equation*}
	\sum_{{\tt z}_k\in D\setminus D_0}  f\left(\frac{s({\tt z}_k)}{\log \bigl(1+g_D({\tt z}_k, z_0)\bigr)}\right)\log \bigl(1+g_D({\tt z}_k, z_0)\bigr)  .	
	\end{equation*}
\end{theorem}						
	Доказательство опускаем, поскольку оно выводится из Теоремы \ref{th:1} так же, как доказательство Теоремы 	\ref{th:fsg}, но для тестовых функций  $f\bigl(\frac{s}{\log (1+g_D)}\bigr)\log (1+g_D)$ из Примера  \ref{exp1gd}, приведённых в \eqref{loggdt}. При этом  ко второму неравенству в \eqref{1gD2l} и в правой части \eqref{itotsfF+} применяем очевидное неравенство $\log(1+g_D)\leq g_D$.
		
	Один из наиболее простых выборов функции $s\in \sbh^+(D)$, удовлетворяющих условиям \eqref{1gD2l} с некоторой постоянной $b$, --- это степень $g_D^{p}$ с $p\geq 1$ или суперпозиция $q \circ g_D$ c выпуклой возрастающей функция $q\colon \RR^+\to \RR^+$, для которой 
	$q(t)=O(t)$ при $0<t\to 0$.  
	\subsubsection{\underline{Мажоранта-суперпозиция с гиперболическим радиусом}}\label{MgcR} Далее, как и в  п.~\ref{excr} функция $R\colon D\to \RR^+$ --- гиперболический радиус области $D$ с не менее чем тремя граничными точками. Предполагаем, что $\infty \notin \partial D$.
	\paragraph{\textrm{\bf R.}  {\texttt{Выпуклая область $D$ и суперпозиция с  $-R$}}} 
		Пусть, как и в п.~\ref{mspfG}, $F\colon (-\infty,0)\to \RR$ --- выпуклая возрастающая функция,  доопределенная, как в  
		п.~\ref{mspfG}, по непрерывности  в точке $-\infty$. Для выпуклой области $D$, как отмечено выше в Примере \ref{ex:Rconv}, её гиперболический, или конформный, радиус супергармоничен, т.\,е. $-R\in \sbh (D)$. Тогда по  
		Теореме C\eqref{iihF} функция $M:=F\circ (-R)$ субгармонична в $D$ и для $F$ класса $C^2$ плотность её меры Рисса $\nu_M$ равна
			\begin{multline*}
\frac{1}{2\pi} \Delta \bigl(F\circ (-R)\bigr)=\frac{1}{2\pi} \nabla^2 \bigl(F\circ (-R)\bigr) =\frac{1}{2\pi} \nabla \cdot \bigl(F'(-R)\nabla (-R)\bigr)\\
=\frac{1}{2\pi} F''(-R)|\nabla R|^2+\frac{1}{2\pi}F'(-R) \nabla^2 (-R)	,		
			\end{multline*}
	откуда по уравнению Лиувилля \eqref{RdR} для конформного радиуса имеем
		\begin{multline}\label{df:LFR}
	\frac{1}{2\pi} \Delta \bigl(F\circ (-R)\bigr)	=\frac{1}{2\pi} F''(-R)(R\Delta R+4)-\frac{1}{2\pi} F'(-R) \Delta R\\
	=\frac{2}{\pi} F''(-R) +\frac{1}{2\pi}\Bigl(F''(-R)R-F'(-R)\Bigr)\Delta R.
	\end{multline}
Следующий результат доказывается так же, как Теоремы 	\ref{th:fsg} и \ref{th:loggD}, но для тестовых функций  из Примера  \ref{ex:Rconv}.

\begin{theorem}\label{th:Rconv} В соглашениях и обозначениях этого пп.~{\rm{\ref{MgcR}{\bf R}}} для  $u\in \sbh(D)$ с $u(z_0)\neq -\infty$ и мерой Рисса $\nu_u$ с ограничением $u\leq F\circ (-R)$ на $D\setminus D_0$ для любого числа $b\in \RR^+$ найдутся постоянные
 $C,\overline{C}_M$ вида \eqref{cz0C+}, с которыми  для любых 
выпуклой возрастающей функции   $f\colon \RR^+ \to \RR^+$ с $f(0)=0$ и функции $s\in \sbh^+(D\setminus D_0)$ при ограничениях   
\begin{equation*}\label{1gD2l+}
	t_0:=\sup_{z\in D\setminus D_0} \frac{s(z)}{R(z,D)} <+\infty, \quad 
	f (t_0) \sup_{z\in \partial D_0}   R(z,D)\leq b
\end{equation*}
выполнено неравенство 
\begin{multline}\label{itotsfF++}
	\int_{D\setminus D_0} R f\Bigl(\frac{s}{R}\Bigr) \, d \nu_u
		\overset{\eqref{df:LFR}}{\leq}  \frac{1}{2\pi}\int_{D\setminus D_0} 	Rf\Bigl(\frac{s}{R}\Bigr) \Bigl( 4F''(-R) 
		\\+\bigl(F''(-R)R-F'(-R)\bigr)\Delta R\Bigr)\, d \lambda +C\overline{C}_M-Cu(z_0).
\end{multline}
В частности, если  интеграл в правой части \eqref{itotsfF++} конечен, то конечен  и интеграл в левой части \eqref{itotsfF++}.
Для $u=\log |h|$ с функцией $0\neq h\in \Hol (D)$, 	обращающейся в нуль на 	${\tt Z}=\{{\tt z}_k\}_{k\in \NN}\subset D$, левая часть \eqref{itotsfF++}  будет выглядеть как сумма
	$\sum_{{\tt z}_k\in D\setminus D_0}  R({\tt z}_k,D) f\bigl({s({\tt z}_k)}/{R({\tt z}_k,D)}\bigr) $.	
\end{theorem}
	\paragraph{\textrm{\bf LR.}  {\texttt{Область $D$ гиперболического типа и суперпозиция с  $-\log R$}}}  При ап\-р\-иорных  условиях  \eqref{hypR}  можем использовать  тестовую функцию 	вида \eqref{loggdtR} из Примера \ref{exp1gdR}.
Кроме того, для  области $D\subset \CC_{\infty}$ гиперболического типа с гиперболическим радиусом  $R$ всегда  $-\log R\in \sbh(D)$ \cite[{\bf 3.1}]{Av15} в силу  равенства 
\begin{equation}\label{DlogR}
\Delta \log R=-\frac{4}{R^2}<0 \quad\text{на $D$}.
\end{equation}
 Для выпуклой возрастающей функции $F\colon \RR \to \RR$ класса $C^2$ и субгармонической, по Теореме C\eqref{iihF},   мажоранты $M:=F\circ (-\log R)$ на $D$ с плотностью меры Рисса относительно меры Лебега, определяемой делением на $2\pi$ любого из равных друг другу выражений 
\begin{multline}\label{df:lognR}
	\Delta M=\nabla^2 \bigl(F\circ (-\log R)\bigr)\overset{\eqref{DlogR}}{=}
	F''(-\log R)\,\begin{cases}
		&\bigl|\nabla (\log  R)\bigr|^2
			\\ &\bigl|\nabla R|^2/R^2
\end{cases}
\Biggm\}
+\dfrac{4}{R^2}\,F'(-\log R)
\\
	\overset{\eqref{{seq}s}}{=}
\begin{cases}
	&\dfrac12\,F''(-\log R)\Delta \log^2 R +\dfrac{4}{R^2} \bigl(F''(-\log R)\log R +F'(-\log R)\bigr),\\
	&\dfrac12\,F''(-\log R)\dfrac{\Delta R^2}{R^2} -\dfrac1R F''(-\log R) \Delta R+ \dfrac{4}{R^2} F'(-\log R),
\end{cases}
\end{multline}
также можно легко выписать  аналог Теоремы \ref{th:loggD} с  тестовыми функциями из Примера \ref{exp1gdR} или любыми иными.
Возможный вариант ---
\begin{theorem}\label{th:loggD++} 
В соглашениях и обозначениях этого пп.~{\rm{\ref{MgcR}{\bf LR}}} для  функции $u\in \sbh (D)$ с $u(z_0)\neq -\infty$ и мерой Рисса $\nu_u$ 
с ограничением $u\leq F\circ (-\log R)$ на $D\setminus D_0$ для любого числа $b\in \RR^+$ найдутся постоянные
  $C,\overline{C}_M\in \RR^+$ вида \eqref{cz0C+}, с которыми  для любых выпуклой возрастающей функции   $f\colon \RR^+ \to \RR^+$ с $f(0)=0$ и функции $s\in \sbh^+(D\setminus D_0)$ при ограничениях   
\begin{equation*}
	t_0:=\sup_{z\in D\setminus D_0} \frac{s(z)}{\log (1+NR)} <+\infty, \quad 
	  f(t_0)\sup_{\partial D_0}R\leq b/N
\end{equation*}
выполнено неравенство 
\begin{multline}\label{itotsfF+l}
	\int_{D\setminus D_0}  f\Bigl(\frac{s}{\log (1+NR)}\Bigr)\log (1+NR) \, d \nu_u\\
	\overset{\eqref{denFlog}}{\leq} \frac{N}{4\pi} \int_{D\setminus D_0} 	f\Bigl(\frac{s}{\log (1+NR)}\Bigr)
	\Bigl(	F''(-\log R)R\Delta \log^2 R \\
	+\dfrac{8}{R} \bigl(F''(-\log R)\log R +F'(-\log R)\bigr)\Bigr)\, \, d\lambda
		+C\overline{C}_M-Cu(z_0).
\end{multline}
	В частности, если  интеграл в правой части \eqref{itotsfF+l} конечен, то конечен  и интеграл в левой части \eqref{itotsfF+l}.
	Для $u=\log |h|$ с\/  $0\neq h\in \Hol (D)$ и $h({\tt Z})=0$ для 	${\tt Z}=\{{\tt z}_k\}_{k\in \NN}\subset D$, левая часть \eqref{itotsfF+l}  будет выглядеть как сумма
		\begin{equation*}
	\sum_{{\tt z}_k\in D\setminus D_0}  f\left(\frac{s({\tt z}_k)}{\log \bigl(1+NR({\tt z}_k, z_0)\bigr)}\right)\log \bigl(1+NR({\tt z}_k, z_0)\bigr)  .	
	\end{equation*}
\end{theorem}						

	\subsubsection{\underline{Мажоранты от расстояния до подмножества на $\partial D$}}\label{MdistED} 
Пусть $E\subset \CC_{\infty}$ --- {\it непустое замкнутое собственное подмножество в\/} $\CC_{\infty}$, 
 \begin{equation}\label{dDzE}
d_E(z):=\dist (z,E) \quad\text{при $z\in \CC$}, \quad  d_E(\infty):=\begin{cases}
	+\infty \quad &\text{при $\infty\notin E$},\\
	0 \quad &\text{при $\infty\in E$}.
\end{cases}
\end{equation}
Для  различных форм записей или для краткости введём  обозначения 
\begin{subequations}\label{seq:d}
\begin{align}
\mathcal O&:=\complement  E:= \CC_{\infty}\setminus E\neq \varnothing  
\quad\text{--- дополнение $E$ до $\CC_{\infty}$,}
\tag{\ref{seq:d}O}\label{{seq}O}	
\\
\intertext{являющееся {\it открытым собственным подмножеством в\/ $\CC_{\infty}$,}}
d_{\CC_{\infty}\setminus \mathcal O}(z) &=d_{\complement \mathcal O}=d_E(z) , \quad z\in \CC.
\tag{\ref{seq:d}d}\label{{seq}d}
\end{align}
\end{subequations}

Рассмотрим сначала 
\paragraph{\textrm{\bf D.}  {\texttt{Случай  $E=\complement D$      и мажоранты $M=F\circ (-d_E)$}}}
В этом пп.~{\bf D} предполагаем, что  $D:=\CC_{\infty} \setminus E$ --- {\it область\/} такая же, как в пп.~\ref{mspfG}{\bf G} и для нее выполнено условие \ref{gs1}g или пара условий \ref{gs2}g--\ref{gs3}g из п.~\ref{distt}.  

\begin{theorem}\label{distgF} Пусть $F\colon (-\infty, 0)\to \RR$ --- выпуклая функция класса $C^2$ и  функция $u\in \sbh (D)\setminus \{\boldsymbol{-\infty}\}$ удовлетворяет неравенству  $u\leq M:=F\circ (-d_{\complement D})$ на $D$; постоянные\/ $0< b_0\leq 1\leq B_0<+\infty$ из \eqref{seqq};  $b\in \RR^+$. Тогда найдутся подобласть\/ $D_0\Subset D$ и постоянные 
  $C,\overline{C}_M\in \RR^+$ вида \eqref{cz0C+}, с которыми 
		\begin{enumerate}[{\rm (a)}]
		\item\label{fdF-1} для любой   выпуклой возрастающей функции $f\colon \RR^+\to \RR^+$ с $f(0)=0$ при $t_0:=B_0\sup\limits_{\partial D_0} d_{\complement D}$
		 и $ f(t_0)\leq b$ справедливо неравенство
		\begin{equation}\label{itotsfFdtd}
	\int_{D\setminus D_0} f(b_0d_{\complement D}) \, d \nu_u\leq \frac{1}{B_0^2}\int_0^{t_0}f(t)F''\Bigl	(-\frac{1}{B_0}t\Bigr) \, dt+C\overline{C}_M-Cu(z_0);
\end{equation}
	\item\label{fdF-2} для любой выпуклой убывающей функции   $f\colon \RR^+ \to \RR^+$   с правой производной 
	$f'_{\text{\rm\tiny right}}(0)<+\infty$ и  с $f(0)B_0\sup\limits_{\partial D_0}d_{\complement D}\leq b$  выполнено неравенство
	\begin{multline}\label{itotsfFgd}
\int_{D\setminus D_0} d_{\complement D}\,f\Bigl(\frac{1}{b_0d_{\complement D}}\Bigr) \, d \nu_u\\
\leq  \frac{1}{b_0B_0^2}\int_0^{t_0} tf({1}/{t})F''\bigl(-{t}/{B_0}\bigr)  \, dt+C\overline{C}_M-Cu(z_0).
\end{multline}
		\end{enumerate}
В частности, если  интеграл в правой части \eqref{itotsfFdtd} или \eqref{itotsfFgd} конечен, то конечен  и интеграл в левой части соответственно \eqref{itotsfFdtd} или \eqref{itotsfFgd}. Для $u=\log |h|$ с ненулевой функцией $h\in \Hol (D)$, 	обращающейся в нуль на 	${\tt Z}=\{{\tt z}_k\}_{k\in \NN}\subset D$, левая часть \eqref{itotsfFdtd} или  \eqref{itotsfFgd} будет соответственно выглядеть как сумма
			$\sum\limits_{{\tt z}_k\in D\setminus D_0} f\bigl(b_0d_{\complement D}({\tt z}_k)\bigr)$ или 
	$\sum\limits_{{\tt z}_k\in D\setminus D_0} d_D({\tt z}_k) \,f\bigl(1/b_0d_{\complement D}({\tt z}_k)\bigr)$.	

\end{theorem}
\begin{proof}
Достаточно воспользоваться соответственно Случаями \eqref{gFt0a} и \eqref{gFt0b} Теоремы \ref{th:fsg}
и оценками \eqref{gdfgdd1} и \eqref{gdfgdd1qq} Предложения \ref{prDDdistd1} для тестовых функций,  а также  оценкой $u\leq F(-d_{\complement D})\leq F(-g_D/B_0)$ на $D\setminus D_0$, где подобласть $D_0\Subset D$ расширена настолько, что  $z_0\in  U_{t_0}\subset D_0$. 
\end{proof}
\begin{remark}\label{rrdg} См. Замечание \ref{f2pr} с \eqref{f2prf}.
\end{remark}

\paragraph{\textrm{\bf d.}  {\texttt{О функции расстояния}}}
Напомним некоторые свойства функции расстояния, предварив их необходимыми определениями из \cite{CS}.
\begin{definition}[{\cite[Definitions 1.1.1, 2.1.1]{CS}}]\label{df:scnc}
Пусть $S\subset \CC$. Непрерывная функция $q\colon S\to \RR$ называется {\it полувогнутой\/} (с линейным модулем) на  $S$, если существует число $K\in \RR^+$, для которого   
\begin{equation*}
	tq(z)+(1-t)q(z')-q\bigl(tz+(1-t)z'\bigr)\leq 	\frac{K}{2} t(1-t)|z-z'|^2 
\end{equation*}
для всех $0\leq t \leq 1$ и $z,z'\in S$, для которых отрезок $[z,z']$ содержится  в $S$. Число  $K$ называется {\it постоянной полувогнутости\/} для $q$ на $S$. 

Через $\scl (\mathcal O')$ обозначается класс всех полувогнутых на $S$ функций; $\scl_{\loc} (S)$ --- класс локально полувогнутых функций, т.\,е. полувогнутых на каждом компактном в $\CC$ подмножестве из $S$.
\end{definition}

В случае {\it открытого выпуклого  множества\/} $S\subset \CC$  функция $g$ полувогнута на $S$ с постоянной полувогнутости $K$ тогда и только тогда, когда функция $z\mapsto g(z)-\frac{K}{2} |z|^2$, $z\in S$, вогнутая на $S$ \cite[Proposition 1.1.3]{CS}.

\begin{definition}[{\large(}
{\cite[Definition 2.2.1]{CS}, \cite[{\bf 1}]{Nour}{\large)}}]\label{df:intsp} Замкнутое множество $S\subset \CC$ удовлетворяет {\it условию внутренней сферы\footnote{interior sphere condition}} для некоторого $r>0$, если $S$ --- объединение замкнутых кругов радиуса $r$, т.\,е. для каждой точки  $z\in S$ найдётся точка $z'\in \CC$, для которой
$z\in \overline{D}(z',r)\subset S$.  Множество $S$ удовлетворяет {\it условию внешней сферы\footnote{exterior sphere condition}} для некоторого $r>0$, если для каждой точки $z\in \partial S$ найдётся точка $z'\notin S$, для которой ${D}(z',r)\cap S=\varnothing$ и $|z-z'|=r$. 

Если значение $r>0$ не уточняется, то говорят, что $S$ удовлетворяет {\it равномерному условию\/}
соответственно  {\it внутренней} и {\it внешней сферы}.
\end{definition}

{\textsc{Свойства} \underline{функции расстояния}} {{\large(}\cite[Proposition 2.2.2]{CS}, \cite[Proposition 3.2]{ACTS}, \cite[{\bf 14.6}]{GT}{\large)}}. {\it Пусть $E\subset \CC$ непустое и замкнутое подмножество в $\CC$. Тогда
\begin{enumerate}
\item\label{dE0} Функция $d_E$ удовлетворяет условию Липшица
\begin{equation*}
	\bigl|d_E(z)-d_E(z')\bigr|\leq |z-z'| \quad\text{для всех $z,z'\in \CC$,}
	\end{equation*}
дифференцируема почти всюду  на $\mathcal O':=\CC\setminus E$ и даже на $\mathcal O'\setminus {\it\Sigma} $, где  $\it\Sigma$ --- объединение не более чем счётного числа спрямляемых дуг,  а также для неё имеет место  уравнение эйконала 
\begin{equation}\label{eikon}
	\begin{cases}
		\bigl|\nabla d_E\bigr|^2&=1\quad\text{почти всюду на $\mathcal O'$ и даже на $\mathcal O'\setminus\it\Sigma$},\\
		d_E\Bigm|_{\partial \mathcal O'}&=0.
	\end{cases}
\end{equation}
 \item\label{dEd} Если $\mathcal O':=\CC\setminus E$ --- ограниченная в $\CC$ область с границей класса $C^2$ с кривизной $k$, то функция $d_E$ класса $C^2$ в  $\{z\in \clos \mathcal O\colon d_E(z)<k\}$.  
	\item\label{dE1} $d_E^2:=(d_E)^2\in \scl (\CC)$  --- функция с постоянной полувогнутости, равной $2$, т.\,е. функция 
$z \mapsto d_E^2(z)-|z|^2$ вогнутая на $\CC$. 

	\item\label{dE2} $d_E\in \scl_{\loc} (\CC \setminus E)$. Более точно и общ\'о, для любого $S\subset \CC$ при условии $\dist (S,E)>0$  функция $d_E$ полувогнута на $S$ с постоянной полувогнутости, равной $1/\dist (S,E)$. В частности, если $S$ ещё и выпуклое множество, то функция 
	\begin{equation}\label{conde}
	z\mapsto  d_E(z) -\frac{1}{2\dist (S,E)}\,|z|^2 \quad\text{вогнута на $S$.}
	\end{equation}
	\item\label{dE3} Если множество $E$ удовлетворяет условию внутренней сферы для  некоторого $r>0$, то функция $d_E$ полувогнута на замыкании в $\CC$ множества $\CC\setminus E$ с постоянной полувогнутости, равной $1/r$.
\item\label{dE-} $d_E$ не может быть локально полувогнутой на всей плоскости\/ $\CC$.
\end{enumerate}
}
\begin{propos}\label{pr:sprdE} Пусть $\varnothing \neq E\subset \CC$ замкнутое в $\CC$	; $\mathcal O':=\CC\setminus E$. 
\begin{enumerate}
	\item\label{sph1} Функция $d_E^2$ --- $\delta$-субгармоническая на $\CC$ 
	с зарядом Рисса $\nu_{d_E^2}\leq \frac{2}{\pi} \lambda$ на $\CC$, где, как и выше, $\lambda$ --- мера Лебега на $\CC$.
\item\label{sph2} Функция $d_E$ --- $\delta$-субгармоническая на $\mathcal O'$. При этом для любого выпуклого подмножества $S\subset \mathcal O'$ при условии 
$\dist (S,E)>0$  
её заряд Рисса $\nu_{d_E}$ удовлетворяет условию  $\nu_{d_E}\leq \frac{1}{\pi\dist (S,E)} \lambda$ на $S$.
\end{enumerate}
\end{propos}
\begin{proof}
\ref{sph1}. Из п.~\ref{dE1} Свойств функция $z\mapsto d_E^2(z)-|z|^2$ --- вогнутая и, в частности,  супергармоническая на $\CC$,  а оператор Лапласа от неё отрицателен в смысле распределений. Следовательно, $d_E^2$ --- разность выпуклых функций, первая из которых $z\mapsto |z|^2$, т.\,е. $d_E^2\in \dsbh(\CC)$ с зарядом Рисса  $\nu_{d_{E}^2}:=\frac{1}{2\pi} \Delta d_{E}^2\leq  \frac{1}{2\pi}\Delta |\cdot |^2=\frac{2}{\pi}\lambda$. 

\noindent
\ref{sph2}. Из п.~\ref{dE2} Свойств  $d_E$ --- $\delta$-субгармоническая функция в каждом, конечно же, выпуклом круге $D(z,r)\Subset \mathcal O'$, а значит и на $\mathcal O'$. На каждом выпуклом $S$ при условии $\dist (S, E)>0$ функция из \eqref{conde} супергармоническая, и тогда для заряда Рисса $\nu_{d_E}$  функции $d_E\in \dsbh(S)$ в смысле теории распределений имеем
\[ 
\nu_{d_E}:=\frac{1}{2\pi} \Delta d_E\leq  \frac{1}{2\pi}\frac{1}{2\dist (S,E)}\Delta |\cdot |^2=\frac{1}{\pi \dist (S,E)}\lambda. 
\]
\end{proof}

\paragraph{\textrm{\bf L.}  {\texttt{Функция $F\circ \log (1/d_E)$}}}
Далее будет использованы, наряду с  рассмотренной в пп. {\bf D}, и другие возможные версии мажорант $M$ на $D$, зависящих только от функции расстояния $d_E$, когда $E=\complement D$ --- дополнение области $D$ до $\CC_{\infty}$ или $E\subset \partial D$  --- часть границы области $D$ {\large(}см. и ср. \cite{BGK09}--\cite{FR13}{\large)}. 

Для непустого собственного замкнутого подмножества  $E\subset \CC_{\infty}$ продолжаем использовать обозначения
\eqref{dDzE}--\eqref{seq:d} в предположении, что  $E\neq \{\infty\}$.   Как и в \cite[Примеры]{KhF},  рассмотрим  функции 
\begin{equation*}
z\mapsto \log |z| , \quad  z\mapsto \log^+ |z|, \quad   z\mapsto \log (1+|z|),  \qquad z \in \CC,
\end{equation*}
субгармонические в $\CC$. Суперпозиции  этих функций с любой  функцией $h\in \Hol (\mathcal O)$, т.\,е.  $\log |h|$, $\log^+|h| $,  $\log \bigl(1+|h|\bigr)$ --- субгармонические,  а две последние ещё  и непрерывные функции в $\mathcal O\overset{\eqref{{seq}O}}{=}\CC_{\infty}\setminus E$ \cite[Следствие 2.5.7]{Klimek}.	
Следовательно, точные верхние грани этих функций при $h_w(z):=\frac{1}{w-z}$ по индексу $w\in E$, равные соответственно 
\begin{equation}\label{df:lL}
\log \frac1{d_E}\,, \quad \log^+ \frac1{d_E}\geq 0\, , \quad
\log \Bigl(1+\frac1{d_E}\Bigr)\geq 0,
\end{equation}
будучи непрерывными, также являются субгармоническими функциями на $\mathcal O$. Эти функции, если исключить постоянные, можно считать в определенном смысле субгармоническими функциями на $\mathcal O$ минимального роста вблизи   $\partial \mathcal O$ при условии  зависимости их исключительно  от  функции расстояния $d_E$.  

Если $F\colon \RR\to \RR$ --- возрастающая выпуклая функция, то  по Теореме C\eqref{iihF} субгармонична и суперпозиция функции $F$ с любой из трёх функций \eqref{df:lL}. Традиционны  степенная и экспоненциальная функции $F$.
\begin{enumerate}
	\item[{[p]}] При $p\in [1,+\infty)$ --- субгармоническая суперпозиция степенной функции $x\mapsto x^p$, $x\in \RR^+$, с функциями из \eqref{df:lL}:
	\begin{equation}\label{df:lLp}
	\left(\log^+ \frac1{d_E}\right)^p\geq 0 \quad\text{или} \quad 
	\log^p \Bigl(1+\frac1{d_E}\Bigr)\geq 0,
	\end{equation}
	\item[{[e]}] При $p\geq 0$ --- субгармоническая суперпозиция экспоненциальной  функции $x\mapsto \exp(px)$, $x\in \RR$, с функциями из \eqref{df:lL}, из которых выпишем только суперпозицию с первой функцией:
	\begin{equation}\label{df:lLe}
	\left(\frac{1}{d_E}\right)^p=\exp\left(p \log \frac{1}{d_E}\right). 
	\end{equation}
\end{enumerate}
Функции из \eqref{df:lL} отличаются друг от друга вблизи границы   $\partial  \mathcal O$ не более чем на константу. Тогда суперпозиции  $F$ с функциями из \eqref{df:lL}  в рамках  рассматриваемой тематики различаются 
вблизи  $\partial  \mathcal O$ несущественно,  если выпуклая функция  $F$ растёт не быстрее экспоненциальной функции, см. \eqref{df:lLp}--\eqref{df:lLe}. Поэтому здесь разбирается только случай суперпозиции 
\begin{equation}\label{dDzEM}
F\circ \log \frac{1}{d_E} \in \sbh (D), 
\end{equation}
с {\it выпуклой возрастающей функцией  $F\colon \RR\to \RR$ класса $C^2$.\/}

Предположим, что функция  расстояния $d_E$ из класса $C^2$ в некотором открытом подмножестве открытого множества  $\mathcal O$,  реализации чего см. в Свойствах   \ref{dE0}--\ref{dEd} функции расстояния из пп.~{\bf d}.  Плотность меры Рисса $\nu_{F\circ \log \frac{1}{d_E}}$ функции \eqref{dDzEM} относительно меры Лебега $\lambda$ на таком подмножестве определяется   любым из равных друг другу выражений 
\begin{multline}\label{DeltaME}
d\nu_{F\circ \log \frac{1}{d_E}}=\frac{1}{2\pi}\Delta \bigl( F\circ (-\log {d_E})\bigr) \, d\lambda
=\frac{1}{2\pi}\nabla \cdot \bigl(-F'(-\log {d_E})  \nabla \log {d_E}\bigr)\, d\lambda\\
=\frac{1}{2\pi}F''(-\log {d_E})|\nabla \log {d_E}|^2\, d\lambda+\frac{1}{2\pi}F'(-\log {d_E})\Delta \log\frac{1}{{d_E}}\, d\lambda\\
=\frac{1}{2\pi}F''(-\log {d_E}) \frac{|\nabla {d_E}|^2}{{d_E}^2}\, d\lambda - \frac{1}{2\pi}F'(-\log {d_E}) \frac{{d_E}\nabla^2 {d_E}-|\nabla {d_E}|^2}{{d_E}^2}\, d\lambda\\
\overset{\eqref{eikon}}{=}\frac{1}{2\pi}\bigl(F''(-\log {d_E}) + (1-d_E\Delta d_E) F'(-\log {d_E}) \bigr)\frac{1}{d_E^2}\, d\lambda.
 \end{multline}

\paragraph{\textrm{\bf DL.}  {\texttt{Случай  $E=\CC_{\infty} \setminus D$ и мажоранты $M:=F\circ \log (1/d_E)$}}}
В этом пп.~{\bf DL} рассматриваем только области $D$, удовлетворяющие 
условию \ref{gs1}g, в частности, с границей класса $C^2$. Из вычислений оператора Лапласа в \eqref{DeltaME} по аналогии с Теоремами
\ref{th:loggD}--\ref{th:loggD++} легко следует

\begin{theorem}\label{distgFElog} Пусть $F\colon \RR\to \RR$ --- выпуклая функция класса $C^2$ и  функция $u\in \sbh (D)\setminus \{\boldsymbol{-\infty}\}$ удовлетворяет неравенству  $u\leq M:=F\circ (-\log d_{\complement D})$ на $D$; 	постоянные\/ $0< b_0\leq 1\leq B_0<+\infty$ из \eqref{seqq} и соотношения \eqref{gdfgdd1} Предложения\/ {\rm  \ref{prDDdistd1}};
		$b\in \RR^+$. Тогда найдутся подобласть\/ $D_0\Subset D$ и постоянные 
	числа\/ $0<a\leq 1\leq A <+\infty$ и\/     
	$C,\overline{C}_M\in \RR^+$ вида \eqref{cz0C+}, с которыми 
	\begin{enumerate}[{\rm (a)}]
		\item\label{fdF-1+} для любой   выпуклой возрастающей функции $f\colon \RR^+\to \RR^+$ с $f(0)=0$ при $t_0:=B_0\sup\limits_{\partial D_0} d_{\complement D}$
		и $ f(t_0)\leq b$ справедливо неравенство
		\begin{multline}\label{itotsfFdtdq}
		\int_{D\setminus D_0} f(b_0d_{\complement D}) \, d \nu_u
			\leq 
		\frac{1}{B_0^2} f(d_{\complement D})\frac{1}{2\pi}\bigl(F''(-\log d_{\complement D})\\
		 + (1-d_{\complement D}\Delta d_{\complement D}) F'(-\log {d_{\complement D}}) \bigr)\frac{1}{d_{\complement D}^2}\, d\lambda
		+C\overline{C}_M-Cu(z_0);
		\end{multline}
		\item\label{fdF-2+} для любой выпуклой убывающей функции   $f\colon \RR^+ \to \RR^+$   с правой производной 
		$f'_{\text{\rm\tiny right}}(0)<+\infty$ и  с $f(0)B_0\sup\limits_{\partial D_0}d_{\complement D}\leq b$  выполнено неравенство
				\begin{multline}\label{itotsfFgdq}
		\int_{D\setminus D_0} d_{\complement D}\,f\Bigl(\frac{1}{b_0d_{\complement D}}\Bigr) \, d \nu_u\\
		\leq 
		\frac{1}{b_0B_0^2} f(d_{\complement D})\frac{1}{2\pi}\bigl(F''(-\log d_{\complement D})\\
		+ (1-d_{\complement D}\Delta d_{\complement D}) F'(-\log {d_{\complement D}}) \bigr)\frac{1}{d_{\complement D}^2}\, d\lambda +C\overline{C}_M-Cu(z_0).
		\end{multline}
	\end{enumerate}
	В частности, если  интеграл в правой части \eqref{itotsfFdtdq} или \eqref{itotsfFgdq} конечен, то конечен  и интеграл в левой части соответственно \eqref{itotsfFdtdq} или \eqref{itotsfFgdq}. Для $u=\log |h|$ с ненулевой функцией $h\in \Hol (D)$, 	обращающейся в нуль на 	${\tt Z}=\{{\tt z}_k\}_{k\in \NN}\subset D$, левая часть \eqref{itotsfFdtdq} или \eqref{itotsfFgdq} будет соответственно выглядеть как сумма
		$\sum\limits_{{\tt z}_k\in D\setminus D_0} f\bigl(b_0d_{\complement D}({\tt z}_k)\bigr)$ или 
	$\sum\limits_{{\tt z}_k\in D\setminus D_0} d_D({\tt z}_k) \,f\bigl(1/b_0d_{\complement D}({\tt z}_k)\bigr)$.	
	
\end{theorem}

\paragraph{\textrm{\bf DLE.}  {\texttt{Случай  \underline{строгого включения} $E\subset \partial D$  и мажоранты $M:=F\circ \log (1/d_E)\bigm|_D$}}}

Области $D$ и  замкнутому непустому  подмножеству $E\subset \partial D$ сопоставим   расширение области $D$ до области 
\begin{equation}\label{wD_E}
\widehat{D}_E=D\bigcup \Bigl\{D\bigl(z,r(z)\bigr)\colon z\in \partial D\setminus E, \; r(z):=\dist (z, E)\Bigr\}\supset D.
\end{equation}

Кроме того, расстояния $\dist (\cdot, E)$ или $\dist d(\cdot, \partial D)$  будем обозначать и через $d_D(\cdot, E)$ или 
$d_D(\cdot, \partial D)$, где соответственно $\cdot$ --- точки из области $D$  и только.

\begin{theorem}\label{th:distEDD} В обозначениях \eqref{dDzE}, \eqref{dDzEM}, \eqref{wD_E} 
 и соглашениях об области $ D$ и функции $F$ данного п.~{\rm \ref{MdistED}} предположим ещё, что  как область $D$, так и её расширение $\widehat{D}_E$ удовлетворяют условию {\rm\ref{gs1}g} или одновременно паре условий {\rm\ref{gs2}g}--{\rm\ref{gs3}g} из п.~{\rm\ref{distt}}. Пусть $u\in \sbh(D)$ с $u(z_0)\neq -\infty$ и мерой Рисса $\nu_u$ с ограничением $u\overset{\eqref{dDzE}}{\leq} F\circ \bigl(-\log d_E\bigr)$ на $D\setminus D_0$.  Тогда для любого числа $b\in \RR^+$ найдутся постоянные  $C,\overline{C}_M$ вида \eqref{cz0C+}, с $C$, зависящей и от $E$, с которыми для любой выпуклой убывающей функции $f\colon \RR^+\to [0,b]$  с правой производной 
	$f'_{\text{\rm\tiny right}}(0)<+\infty$,  можно подобрать  числа $0<a\leq 1\leq A<+\infty$, зависящие только от $D$, $z_0$ и $E$, так, что в обозначениях \eqref{dDzE} 
\begin{multline}\label{itotsfFE+}
	\int\limits_{D\setminus D_0} d_D(\cdot , \partial D) \, f\left(A\frac{d_D(\cdot , E)}{d_D(\cdot , \partial D)}\right)\,  d\nu_u
	\overset{\text{\eqref{dDzEM}--\eqref{DeltaME}}}{\leq}\\
	A\int\limits_{D\setminus D_0} d_D(\cdot , \partial D) \,f\left(a\frac{d_D(\cdot , E)}{d_D(\cdot , \partial D)}\right) \frac{1}{2\pi}\bigl(F''(-\log {d_E}) + (1-d_E\Delta d_E) F'(-\log {d_E}) \bigr)\frac{1}{d_E^2}\, d\lambda\\
	+C\overline{C}_M-Cu(z_0).
\end{multline}
В частности, если  интеграл в правой части \eqref{itotsfFE+} конечен, то конечен  и интеграл в левой части \eqref{itotsfFE+}.
Для $u=\log |h|$ с ненулевой функцией $h\in \Hol (D)$, 	обращающейся в нуль на 	${\tt Z}=\{{\tt z}_k\}_{k\in \NN}\subset D$, левая часть \eqref{itotsfFE+}  будет выглядеть как сумма
		\begin{equation*}
	\sum_{{\tt z}_k\in D\setminus D_0} d_D({\tt z}_k , \partial D) \, f\left(A\frac{d_D({\tt z}_k , E)}{d_D({\tt z}_k , \partial D)}\right)	.
	\end{equation*}
\end{theorem}
\begin{proof} При условии {\rm\ref{gs1}g} или  паре условий {\rm\ref{gs2}g}--{\rm\ref{gs3}g} из п.~{\rm\ref{distt}} область $D\Subset \CC$ регулярна, а граница $\partial \widehat{D}_E$ неполярна. Таким образом, существуют как  функция Грина $g_D:=g_D(\cdot, z_0)$ для $D$, так  и функция Грина $g_{\widehat{D}_E}:=g_{\widehat{D}_E}(\cdot, z_0)$ для $\widehat{D}_E$, выбранная  с тем же полюсом $z_0$. 
Функция  $f$ может быть продолжена на $\RR$ как выпуклая убывающая. По Предложению \ref{prgg} функция $g_Df\bigl(g_{{\widehat{D}_E}}/g_D\bigr)$ тестовая и принадлежит классу $\sbh_0^+(D\setminus D_0; \leq b_0b)$, где $b_0=\sup\limits_{z\in \partial D_0}g_D(z,z_0)$ зависит только от $D$, $D_0$ и $z_0$.  По Теореме \ref{th:1}  в условиях доказываемой теоремы найдутся требуемые постоянные $C,\overline{C}_M$, с которыми
\begin{multline}\label{mest+E}
\int\limits_{D\setminus D_0} g_Df\Bigl(\frac{g_{{\widehat{D}_E}}}{g_D}\Bigr)  \,d {\nu}_u 		 
\leq 	\int_{D\setminus D_0}  g_Df\Bigl(\frac{g_{{\widehat{D}_E}}}{g_D}\Bigr)  \,d {\nu}_M	+		C\,	\overline{C}_M -C\, u(z_0) \\
\overset{\text{\eqref{dDzEM}--\eqref{DeltaME}}}{=}
 \int_{D\setminus D_0}  g_Df\Bigl(\frac{g_{{\widehat{D}_E}}}{g_D}\Bigr)  Q_{F,D,E}\,d \lambda	+		C\,	\overline{C}_M -C\, u(z_0), 
\end{multline}
где для краткости введено обозначение
\begin{equation*}
Q_{F,D,E}\overset{\eqref{DeltaME}}{=}\frac{1}{2\pi}\bigl(F''(-\log {d_E}) + (1-d_E\Delta d_E) F'(-\log {d_E}) \bigr)\frac{1}{d_E^2}.
\end{equation*}

   По Предложению \ref{prDDdist} с некоторыми постоянными $0<a'\leq 1\leq A'<+\infty$, связанными только с $D$, $E$, $z_0$, имеем  
\begin{equation}\label{gdfgd+d}
	g_D \, f\left(A'\,\frac{\dist \bigl(\cdot, \partial \widehat{D}_E\bigr)}{\dist (\cdot, \partial D)}\right)\leq g_D \,f\Bigl(\frac{g_{\widehat{D}_E}}{g_D}\Bigr)
\leq g_D \, f\left(a'\,\frac{\dist \bigl(\cdot, \partial \widehat{D}_E\bigr)}{\dist (\cdot, \partial D)}\right)
\end{equation}
на $D\setminus D_0$. Для области $\widehat{D}_E$ по её построению \eqref{wD_E} выполняются равенства
\begin{equation*}\label{dDzE+}
d_D(z, E)\overset{\eqref{dDzE}}{=}\dist (z,E)=\dist(z, \partial \widehat{D}_E)
\quad \text{для всех $z\in D$.}
\end{equation*}
Поэтому \eqref{gdfgd+d} может быть переписано в виде
\begin{equation*}
	g_D \, f\left(A'\,\frac{d_D (\cdot, E)}{d_D (\cdot, \partial D)}\right)\leq g_D \,f\Bigl(\frac{g_{\widehat{D}_E}}{g_D}\Bigr)
\leq g_D \, f\left(a'\,\frac{d_D(\cdot, E)}{d_D(\cdot, \partial D)}\right) \quad\text{на $D\setminus D_0$}
\end{equation*}
 Из этиx неравенств и  \eqref{mest+E} следует 
\begin{multline*}
\int\limits_{D\setminus D_0} g_D\, f\left(A'\,\frac{d_D (\cdot, E)}{d_D (\cdot, \partial D)}\right) \,d {\nu}_u 		 
\\ \leq \int_{D\setminus D_0}  g_D f\left(a'\,\frac{d_D(\cdot, E)}{d_D(\cdot, \partial D)}\right) Q_{F,D,E}\,d \lambda	+		C\,	\overline{C}_M -C\, u(z_0).
\end{multline*}
Отсюда согласно верхним и нижним оценкам  \eqref{seqq} на функцию Грина $g_D$, подбирая достаточно большие  $A\geq 1$ и  
малые $a>0$, а также меняя $C\geq 0$, получаем требуемое равномерное по указанным  $f$ неравенство \eqref{itotsfFE+}.
\end{proof}
\begin{remark}\label{r:hr} Подобные  результаты из Теоремы \ref{th:1} можно получить и через гиперболический радиус, используя 
тестовые функции из подраздела \ref{excr}, их оценки в подразделе \ref{distt} через функцию расстояния (Предложение \ref{pr:rtrf}	и 
замечание после него) для мажорант-суперпозиций из подраздела \ref{MgcR}, строящихся через гиперболический радиус. Последняя Теорема \ref{th:distEDD}	близка и параллельна части исследований из 
\cite{BGK09}--\cite{FR13}, дополняет  их, но без  увязки наших результатов с возможными специальными свойствами  функции расстояния до подмножества  $E\subset \partial D$, как это сделано в разных вариантах в \cite{BGK09}--\cite{FR13}.
\end{remark}

\subsubsection{\underline{Перспективы и заключительные замечания}}\label{persp}
Дальнейшее развитие ---
\begin{itemize}
		\item  интегральные версии основных результатов данной статьи, в которых  поточечные ограничения посредством мажоранты $M$ заменяются ограничениями на интегралы с весом, на пути   распространения  обобщённой формулы Пуассона--Йенсена \eqref{f:PJ} из Предложения \ref{pr:2}
на тестовые функции из Определения \ref{testv} и их меры Рисса на основе той или иной степени реализации Гипотезы \ref{g:1}; 
\item перенос результатов на области $D$ в  $\RR^m$, $m\in \NN$, с применениями к оценкам объёмов, или  $(2n-2)$-мер Хаусдорфа,   нулевых множеств  голоморфных функций $f$ (с учётом кратности) с ограничениями  на их рост в псевдовыпуклой области $D\subset \CC^n$ через известную классическую формулу Пуанкаре--Лелона, напрямую, --- с коэфициентом, зависящим только от $n$, --- связывающей эти объёмы 	с мерой Рисса (плюри-)cубгармонической функции $\log |f|$.
	\end{itemize}
 
Их  изложение будет дано   в планируемых  последующих наших работах. 

Последний многомерный субгармонический  вариант развития наших результатов  поддаётся реализации без существенного роста сложностей и в значительной своей части  анонсирован и/или реализован в \cite{KhSP16}--\cite{KhaAbdRoz18}.
<<Плюрисубгармонический>> же перенос результатов на псевдовыпуклые области в $\CC^n$ 
для плюрисубгармонической мажоранты  $M$, основанный на плюрисубгармоничности по существу,  более сложен, но  всё же  вполне  перспективен на основе комплексной теории потенциала и аналитических дисков \cite{Klimek}, \cite{Khab01}, \cite{CR+}, 
\cite{Poletsky}, \cite{Po99}.  При этом будут важны именно равномерные относительно мажоранты $M$ оценка \eqref{mest} в Основной Теореме   и  явная  форма постоянных $C$ в \eqref{cz0C}  и $\overline{C}_M$ в \eqref{{mest}C} из Основной Теоремы, для чего только  мы и прослеживали в подготовительных результатах к ней  и её доказательстве вид этих постоянных.

\end{fulltext}


\begin{thebibliography}{86}

\Bibitem{Rans}
\by Th.~Ransford
\book Potential Theory in the Complex Plane
\publ Cambridge University Press
\publaddr Cambridge
\yr 1995


\RBibitem{L} 
\by Н.\,С.~Ландкоф 
\book Основы современной теории потенциала 
\publ Наука 
\publaddr М.
\yr 1966

\RBibitem{HK} 
\by У.~Хейман, П.~Кеннеди
\book Субгармонические функции
\publ Мир
\publaddr М.
\yr 1980

\Bibitem{HII} 
\by W.\,K.~Hayman 
\book Subharmonic functions
\vol II 
\publ Academic Press
\publaddr London
\yr 1989



\Bibitem{Ho}
\by Lars H\"ormander
\book Notions of Convexity
\bookinfo Progress in Mathematics
\publ Birkh\"auser 
\publaddr Boston
\yr 1994

\Bibitem{Doob}
\by  J.\,L.~Doob 
\book  Classical Potential Theory and Its Probabilistic Counterpart
\publ   Springer-Verlag 
\publaddr N.-Y.
\yr 1984

\Bibitem{Klimek}
\by  M.~Klimek
 \book  Pluripotential Theory
\publ   Clarendon Press 
\publaddr  Oxford
\yr 1991



\RBibitem{St} 
\by С.~Стоилов
 \book  Теория функций комплексного переменного
\vol 2 
\bookinfo пер. с рум.
\publ     ИЛ
\publaddr  М. 
\yr 1962

\Bibitem{Ar_d} 
\by M.\,G.~Arsove
\paper Functions representable as differences of subharmonic functions
\jour  Trans. Amer. Math. Soc.
\vol 75
\pages  327--365
\yr 1953

\RBibitem{Gr}
\by А.\,Ф.~Гришин, Нгуен~Ван~Куинь, И.\,В.~Поединцева
\paper Теоремы о представлении $\delta$-субгармонических функций
\jour  Biсник Харкiвського нацiонального унiверситету iменi В.Н. Каразiна 
\serial  Cepiя <<Математика, прикладна математика i механiка>>
\vol {\rm №~1133}
\pages  56--75 
\yr 2014
\finalinfo  http:// vestnik-math.univer.kharkov.ua/Vestnik-KhNU-1133-2014-grish.pdf


\RBibitem{KhT15} 
\by B. Khabibullin, N. Tamindarova
\paper On Distribution of Zeros of Holomorphic Functions: Blaschke-type Conditions
\inbook Материалы международной научно-практической конференции <<Комплексный анализ и его приложения>>
\bookinfo 16--19 июня 
\publ Брянский государственный университет 
\publaddr Брянск 
\pages  67--73 
\yr 2015 
\finalinfo http://arxiv.org/abs/1505.05715




\Bibitem{Colwell}
\by P.~Colwell 
 \book  Blaschke Products
\publ  The University of  Michigan Press
\publaddr  Ann Arbor 
\yr 1985


 


\RBibitem{Dj73}
\by  М.\,М.~Джрбашян
\paper  Теория факторизации и граничных свойств функций, мероморфных в круге
\jour  УМН
\vol 28
\issue 4(172) 
\pages  3--14 
\yr 1973


\RBibitem{Sh83}
\by   Ф.\,А.~Шамоян
\paper  О нулях аналитических в круге функций, растущих вблизи границы
\jour  Изв. АН Арм. ССР. Математика.
\vol XVIII
\issue 1
\pages 15--27 
\yr 1983 


\Bibitem{HKZh}
\by  H.~Hedenmalm, B.~Korenblum, K.~Zhu
 \book  Theory of Bergman spaces
\vol 199 
\bookinfo Graduate Texts in Mathematics
\publ  Springer-Verlag
\publaddr  N.-Y. 
\yr 2000

\RBibitem{Kh06}
\by   Б.\,Н.~Хабибуллин
\paper   Нулевые подмножества, представление мероморфных функций и характеристики Неванлинны в круге
\jour  Матем. сб.
\vol 197
\issue 2
\pages 117--136  
\yr 2006

\RBibitem{GLO}
\by  А.\,А.~Гольдберг,  Б.\,Я.~Левин,  И.\,В.~Островский
\paper Целые и мероморфные функции
\inbook Итоги науки и техники. Современные проблемы математики. Фундам. напр.
\vol 85
\publ  ВИНИТИ
\publaddr  М. 
\pages  5--185 
\yr 1991 


\Bibitem{Levin96} 
\by Levin~B.~Ya.
\book Lectures on entire functions. 
Transl. Math. Mono\-graphs 
\vol 150
 \publaddr Providence RI
\publ Amer. Math. Soc
 \yr 1996


\RBibitem{Khab91} 
\by   Б.\,Н.~Хабибуллин
\paper  Множества единственности в пространствах целых функций одной переменной
\jour  Изв. АН СССР, сер. матем. 
\vol   55
\issue 5
\pages 1101--1123
\yr 1991



\RBibitem{Khab91s}
\by   Б.\,Н.~Хабибуллин
\paper  Теорема единственности для субгармонических функций конечного порядка
\jour  Матем. сб.
\vol    182
\issue 6
\pages 811--827 
\yr  1991 



\RBibitem{Khab92}
\by Б.~Н.~Хабибуллин
\paper О типе целых и мероморфных функций
\jour Матем. сб.
\yr 1992
\vol 183
\issue 11
\pages 35--44

\RBibitem{Khsur} 
\by Б.~Н.~Хабибуллин
 \book  Полнота систем экспонент и множества единственности  
\vol 2 
\bookinfo издание 4-ое, дополненное
\publ    РИЦ БашГУ 
\publaddr  Уфа 
\yr 2012



\RBibitem{Khab091}
\by Б.~Н.~Хабибуллин
\paper Последовательность нулей голоморфных функций, представление мероморфных функций. II.~Целые функции
\jour Матем. сб.
\yr 2009
\vol 200
\issue 2
\pages 129--158


\Bibitem{KV}
\by  A.\,A.~Kondratyuk,  Ya.\,V.~Vasyl'kiv
\paper  Growth majorants and quotient representations of meromorphic functions
\jour  CMFT -- Computational Methods and Function Theory
\vol     1
\issue  2
\pages  595--606
\yr  2001


 


\Bibitem{BGK09}
\by  A. Borichev, L. Golinskii, S. Kupin
\paper  A Blaschke-type condition and its application to complex Jacobi
matrices
\jour  Bull. London Math. Soc.
\vol     41
\pages  117--123
\yr  2009

\Bibitem{BGK16}
\by  A.~Borichev, L.~Golinskii, S.~Kupin
\paper On zeros of analytic functions satisfying non-radial growth conditions
\pages 1--18
\finalinfo \href{https://arxiv.org/abs/1603.04104}{https://arxiv.org/abs/1603.04104}\;

\Bibitem{FG09}
\by  
S. Favorov, L. Golinskii
\paper  A Blaschke-type condition for analytic and subharmonic functions and application
to contraction operators
\jour  Amer. Math. Soc. Transl.
\vol     226
\issue  2
\pages  37--47
\yr  2009


\Bibitem{FG12}
\by  S.\,Yu.~Favorov, L.\,B.~Golinskii
\paper  Blaschke-Type Conditions for Analytic and Subharmonic Functions in
the Unit Disk. Local Analogs and Inverse Problems
\jour  CMFT -- Computational Methods and Function Theory
\vol     12
\issue  1
\pages  151--166
\yr  2012 


\Bibitem{FRdisc13}
\by S.\,Ju.~Favorov, L.\,D.~Radchenko
\paper On Analytic and Subharmonic Functions in Unit Disc Growing Near a Part of the Boundary
\jour Zh. Mat. Fiz. Anal. Geom.
\yr 2013
\vol 9
\issue 3
\pages 304--315

\RBibitem{FR13_D}
\by Liudmyla D. Radchenko, Sergej Ju. Favorov
\paper On subharmonic functions in a unit
ball growing near a part of the
boundary
\jour Збiрник праць Iн-ту математики НАН Укра\"\iни (Zbirnyk Prats Instytutu Matematyky NAN Ukrainy)
\vol 10
\issue 4--5
\yr 2013
\pages 338--356 
\finalinfo \href{http://www.twirpx.com/file/1724787/}{http://www.twirpx.com/file/1724787/}\;

\Bibitem{GK12}
\by  L. Golinskii, S. Kupin
\paper  A Blaschke-type condition for analytic functions on finitely connected domains. Applications to complex perturbations of a finite-band selfadjoint operator
\jour  J. Math. Anal. Appl.
\vol     389
\issue  2
\pages  705--712
\yr  2012




\Bibitem{FG12_unb}
\by S. Favorov, L. Golinskii
\paper Blaschke-type conditions in unbounded domains, generalized convexity and applications in perturbation theory
\jour Rev. Matem. Iberoamericana
\yr 2015
\vol 31
\issue 1
\pages 1--32
\finalinfo http:// arxiv.org/pdf/1204.4283.pdf



\RBibitem{FR13}
\by С.\,Ю.~Фаворов, Л.\,Д.~Радченко
\paper Мера Рисса функций, субгармонических во внешности компакта
\jour Математичнi Студiї
\yr 2013
\vol 40
\issue 2
\pages 149--158


\RBibitem{Khab01}
 \by Б.~Н.~Хабибуллин
\paper Двойственное представление суперлинейных функционалов и его применения в~теории функций.~II
\jour Изв. РАН. Сер. матем.
\yr 2001
\vol 65
\issue 5
\pages 167--190

\RBibitem{Ch05}
\by Л.~Ю.~Чередникова
\paper Последовательности неединственности для весовых алгебр голоморфных функций в~единичном круге
\jour Матем. заметки
\yr 2005
\vol 77
\issue 5
\pages 775--787

\RBibitem{Kh07}
\by Б.~Н.~Хабибуллин
\paper Последовательности нулей голоморфных функций,
представление мероморфных функций и~гармонические миноранты
\jour Матем. сб.
\yr 2007
\vol 198
\issue 2
\pages 121--160

\RBibitem{KhCh08_1} 
\by Б.~Н.~Хабибуллин, Ф.~Б.~Хабибуллин, Л.~Ю.~Чередникова
\paper Подпоследовательности нулей для классов голоморфных  функций, их устойчивость и энтропия линейной связности.~I
\jour Алгебра и анализ
\yr 2008
\vol 20
\issue 1
\pages 146--189

\RBibitem{KhCh08}
\by Б.~Н.~Хабибуллин, Ф.~Б.~Хабибуллин, Л.~Ю.~Чередникова
\paper Подпоследовательности нулей для классов голоморфных  функций, их устойчивость и энтропия линейной связности.~II
\jour Алгебра и анализ
\yr 2008
\vol 20
\issue 1
\pages 190--236

\RBibitem{KuKh09}
\by Е.~Г.~Кудашева, Б.~Н.~Хабибуллин
\paper Распределение нулей голоморфных функций умеренного роста в~единичном круге и представление в~нем мероморфных функций
\jour Матем. сб.
\yr 2009
\vol 200
\issue 9
\pages 95--126


\RBibitem{Kh10}
\by B. N. Khabibullin
\paper Generalizations of Nevanlinna’s Theorems
\jour Математичнi Студiї
\yr 2010
\vol 34
\issue 2
\pages  197--206  
\finalinfo http://www.vntl.com/im/pdf/34\_2\_197\_206.pdf


\Bibitem{Gamelin} 
\by  T.\,W.~Gamelin
 \book  Uniform Algebras and Jensen Measures
\publ Cam\-b\-r\-i\-d\-ge Univ. Press
\publaddr  Cam\-b\-r\-id\-ge
\yr   1978


\Bibitem{Koosis96} 
\by  P.~Koosis
 \book  Le\c cons sur le th\'eor\`eme de Beurling et Malliavin
\publ Les Publications CRM
\publaddr  Mont\-r\'e\-al
\yr   1996

\bibitem{RansfordT} Ransford~T.~J., \emph{Jensen  measures}, in: Approximation, Complex Analysis, and
Potential Theory ( Mont\-r\'e\-al, QC, 2000), NATO Sci. Ser. II Math. Phys. Chem. {\bf 37}, Kluwer, 2001, 221--237.

\Bibitem{CR}
\by B.\,J.~Cole, T.\,J.~Ransford
\paper Subharmonicity without upper semicontinuity
\jour  J. Func. Anal.
\yr  1997
\vol 147
\pages  420--442

 

\Bibitem{CR+}
\by B.\,J.~Cole, T.\,J.~Ransford
\paper Jensen measures and harmonic measures
\jour  J. Reine Angew. Math. 
\yr  2001
\vol 541
\pages  29--53




\RBibitem{Khab99VB}
\by  Б.\,Н.~Хабибуллин
\paper Меры Йенсена на открытых множествах
\jour  Вестник Башкирского университета
\yr  1999
\vol 2
\pages  3--6
\finalinfo 

\RBibitem{Khab03}
\by Б.~Н.~Хабибуллин
\paper Критерии (суб-)гармоничности и~продолжение (суб-)гармонических функций
\jour Сиб. матем. журн.
\yr 2003
\vol 44
\issue 4
\pages 905--925


	\RBibitem{Khab99}
	\by Б.~Н.~Хабибуллин
\paper Полнота систем целых функций
в~пространствах голоморфных функций
\jour Матем. заметки
\yr 1999
\vol 66
\issue 4
\pages 603--616

\Bibitem{Sar71}
\jour J.  Func.  Analysis  
\vol 7 
\pages 359--385 
\yr 1971
\paper Representing  Measures for $R(X)$
and Their Green’s Functions
\by D.~Sarason

\Bibitem{And81}
\jour J.  Func.  Analysis  
\vol 43
\pages 360--367 
\yr 1981
\paper Green’s  Function,  Jensen  Measures,  and
Bounded  Point Evaluations
\by S.\,L.~Anderson


\Bibitem{BH}
\by J.~Bliedtner, W.~Hansen 
\book Potential Theory – An Analytic and Probabilistic Approach to Balayage
\yr 1986
\publ Springer-Verlag
\publaddr Berlin-Heidelberg-N.Y.-Tokyo

\RBibitem{KhS09}
\by  Б.\,Н.~Хабибуллин
\paper Нули голоморфных функций с ограничениями на рост в области
\inbook  Исследования по математическому анализу
\serial Математический форум (Итоги науки. Юг России)
\vol 3
\publ  Владикавказский научный центр РАН и РСО--А
\publaddr  Владикавказ  
\pages  282--291
\yr 2009 
\finalinfo http://elibrary.ru/download/90353677.pdf \;



\RBibitem{KhB15} 
\by B. Khabibullin, T. Baiguskarov
\paper Holomorphic minorants of (pluri-)subharmonic functions
\inbook Материалы Всероссийской научной конференции 
<<Алгебра, анализ и смежные вопросы математического 
моделирования>>
\bookinfo 26--27 June 
\publaddr Владикавказ
\pages  67--70 
\yr 2015 
\finalinfo http://arxiv.org/pdf/1506.01746.pdf



\RBibitem{BaiKha16}
\by Т.~Ю.~Байгускаров, Б.~Н.~Хабибуллин
\paper Голоморфные миноранты плюрисубгармонических функций
\jour Функц. анализ и его прил.
\yr 2016
\vol 50
\issue 1
\pages 76--79
\mathnet{http://mi.mathnet.ru/faa3225}
\crossref{http://dx.doi.org/10.4213/faa3225}
\transl
\jour Funct. Anal. Appl.
\yr 2016
\vol 50
\issue 1
\pages 62--65


\RBibitem{BaiKha16_mz}
\by  Б.~Н.~Хабибуллин, Т.~Ю.~Байгускаров
\paper 	Логарифм модуля голоморфной функции как миноранта для субгармонической функции
\jour Матем. заметки
\yr 2016
\vol 99
\issue 4
\pages 588--602

\bibitem{HN} 
\by W.~Hansen, I.~Netuka
\paper Convexity properties of harmonic measures
\jour Adv. Math. 
\yr 2008
\vol 218
\issue 4
\pages 1181--1223

\Bibitem{Poletsky} 
\by E.\,A.~Poletsky
\paper Holomorphic currents
\jour Indiana Univ. Math. J.
\yr 1993
\vol 42
\pages 85--144

\Bibitem{Po99}
\by E.\,A.~Poletsky
\paper Disk envelopes of functions II
\jour J. Funct. Anal.
\yr 1999
\vol 163
\pages 111--132

\RBibitem{Nat}
\by   И.\,П.~Натансон
 \book  Теория функций вещественной переменной
\publ Наука
\publaddr  М.
\yr   1974

\RBibitem{Schwartz}
\by  Л.~Шварц
 \book  Анализ 
\vol I 
\publ Наука
\publaddr  М.
\yr   1967


\RBibitem{Kond} 
\by   А.\,А.~Кондратюк
 \book  Ряды Фурье и мероморфные функции
\publ Вища школа. Изд-во при Львовском университете
\publaddr  Львов
\yr   1988



\Bibitem{AvW}
 \by F.\,G.~Avkhadiev, K.-J.~Wirths
\book Schwarz-Pick Type Inequalities
\publ Birkh\"auser Verlag
\publaddr Basel-Boston-Berlin
\yr 2009

\RBibitem{Av15} 
\by Ф.\,Г.~Авхадиев
\paper Интегральные неравенства в областях
гиперболического типа и их применения 
\jour Матем. сборник
\vol 206
\issue 12 
\yr 2015 
\pages 3--28

\Bibitem{W}
\by E. Wojcicka 
\book Functions of bounded characteristic in multiply connected domains
\publ Doctoral Dissertation
\publaddr  Univ. of South Carolina 
\yr 1985

\Bibitem{Stoll}
\by M. Stoll
\paper A characterization of Hardy--Orlicz spaces
on planar domains
\jour Proc. Amer. Math Soc.
\vol 117
\issue 4
\yr 1993
\pages 1031--1038

\Bibitem{Widman}
\by  K.-O. Widman
\paper Inequalities for the Green’s function and boundary continuity of the gradient of solutions of elliptic differential equations
\jour  Math. Scand. 
\vol 21 
\yr 1967
\pages 17--37

\RBibitem{KL}
\by М.\,В.~Келдыш, М.\,А.~Лаврентьев
\paper Об одной оценке для функции Грина
\jour Докл. Акад. наук СССР
\vol 24
\yr 1939
\pages 102--103



\bibitem{JR}
\by Yu. Riihentaus
\paper Convex Functions and Subharmonic Functions  
\jour Potential Analysis
\vol 5
\yr 1996
\pages 301--309

\RBibitem{KhF}
\by Ф.\,Б.~Хабибуллин
\paper Устойчивость (под)последовательностей нулей для классов голоморфных функций умеренного роста в единичном круге
\jour Уфимск. матем.
журн.
\yr 2011
\vol 3
\issue 3
\pages 152--163


\Bibitem{CS}
\by P.~Cannarsa, C.~Sinestrari
\book Semiconcave Functions, Hamilton–Jacobi Equations, and Optimal Control
\publ Birkh\"auser
\publaddr Boston 
\serial Progress in Nonlinear Differential Equations
and Their Applications
\vol 58
\yr 2004

\Bibitem{Nour}
\by Ch.~Nour, R.\,J.~Stern, J.~Takche
\jour Control and Cybernetics 
\vol 38 
\yr 2009 
\issue 4B 
\paper The union of uniform closed balls conjecture 
\pages 1525--1534

\Bibitem{ACTS} 
\by P.~Albano, P.~Cannarsa, Khai~T.~Nguyen,  C.~Sinestrari
\paper Singular gradient flow of the distance
function and homotopy equivalence
\jour Math. Annalen
\yr 2013
\vol 356
\issue 1
\pages  23--43


\RBibitem{GT}
\by Д.~Гилбарт, Н.~Трудингер
\book Эллиптические дифференциальные уравнения с частными производными
\publaddr М.
\publ Наука
\yr 1989

\RBibitem{KhSP16}
\by B.\,N.~Khabibullin
\paper  On the distribution of zeros for holomorphic functions of several variables (Joint work with N. Tamindarova.) 
\inbook Материалы Всероссийской научной конференции 
<<Алгебра, анализ и смежные вопросы математического 
моделирования>>
\bookinfo 25th St.Petersburg Summer Meeting in Mathematical Analysis: 
Tribute to Victor Havin (1933-2015). June 25-30, 2016
\publaddr St. Petersburg, Russia
\publ St. Petersburg Department of Steklov Mathematical Institute and St. Petersburg State University
\pages  28 
\yr 2016 
\finalinfo \href{https://www.researchgate.net/publication/309155188}{https://www.researchgate.net/publication/309155188}\;


\RBibitem{KhT_Ufa16}
\by Б.\,Н.~Хабибуллин, Н.\, Р.~Таминдарова
\paper Распределение нулей голоморфных функций и тестовые функции 
\bookinfo Уфимская международная математическая конференция. Тезисы докладов.  г. Уфа, 27.09--30.09.2016
\publaddr Уфа
\publ  ИМВЦ  УНЦ РАН, БашГУ, БГПУ
\pages 172--173 
\finalinfo \href{https://www.researchgate.net/publication/309119860}{https://www.researchgate.net/publication/309119860}\;

\RBibitem{KhR_Br16}
\by  А.\,П.~Розит,  Б.\,Н.~Хабибуллин 
\paper К распределению нулевых множеств голо\-м\-о\-р\-ф\-ных функций
\bookinfo Материалы региональной научно-практической конференции «Комплексный анализ и его приложения» (23-24 ноября 2016 г.) 
\publaddr Брянск
\publ  Брянский государственный университет
\pages 31--32 
\finalinfo \href{http://elibrary.ru/item.asp?id=27553586}{http://elibrary.ru/item.asp?id=27553586}\;

\Bibitem{KhT16_b} 
\by Bulat Khabibullin, Nargiza Tamindarova
\paper Distribution of Zeros for Holomorphic Functions: Hadamard- and Blaschke-type Conditions
\inbook Abstracts of International Workshop on ``Non-harmonic Analysis and Differential Operators''
\procinfo May 25-27, 2016 
\publ Institute of Mathematics and Mechanics of Azerbaijan National Academy of Sciences
\publaddr  Baku, Azerbaijan
\pages  63
\yr 2016 
\finalinfo \href{https://www.researchgate.net/publication/302246768}{https://www.researchgate.net/publication/302246768}\;


\Bibitem{KhT_Lob17}
\by B.\,N.~Khabibullin, N.\,R.~Tamindarova
\paper Subharmonic test functions and the distribution of zero sets of holomorphic functions
\jour Lobachevskii Journal of Mathematics
\vol  38
\issue  1
\pages 38--43
\yr  2017
\finalinfo \href{https://arxiv.org/abs/1606.06714}{https://arxiv.org/abs/1606.06714}\;

\Bibitem{KhT_Az17}
\by B.\,N.~Khabibullin, N.\,R.~Tamindarova
\paper Uniqueness Theorems for Subharmonic and Holomorphic Functions of Several Variables on a Domain 
\jour Azerbaijan Journal of Mathematics
\yr 2017 
\vol 7
\issue 1
\pages 70--79
\finalinfo \href{http://azjm.org/index.php/azjm/article/view/391}{http://azjm.org/index.php/azjm/article/view/391}\;

\RBibitem{KhaRoz18}
\by Б.~Н.~Хабибуллин, А.~П.~Розит
\paper К распределению нулевых множеств голоморфных функций
\jour Функц. анализ и его прил.
\yr 2018
\vol 52
\issue 1
\pages 26--42

\RBibitem{KhaAbdRoz18}
\by Б.~Н.~Хабибуллин, З.~Ф.~Абдуллина, А.~П.~Розит
\paper Теорема единственности и субгармонические тестовые функции
\jour Алгебра и анализ
\yr 2018
\vol 30
\issue 2
\pages 318--334

\end{thebibliography}
\end{document}